\documentclass[a4paper,12pt]{article}
\usepackage{ucs}
\usepackage{amsfonts, amssymb, amsmath, amsthm, amscd}
\usepackage{graphicx}
\usepackage{cite}
\ifx\undefined \pdfgentounicode \else
\input{glyphtounicode} \pdfgentounicode=1
\fi

\author{A.A. Vasil'eva\footnote{Faculty of Mechanics and Mathematics, Lomonosov Moscow State University; Moscow Center for Fundamental and Applied Mathematics.}}
\title{Kolmogorov widths of an intersection of Besov classes with dominating mixed smoothness in a Besov space \footnote{The work was supported by the Ministry of Education and Science of the Russian Federation as part of the program of the Moscow Center for Fundamental and Applied Mathematics under the Agreement No. 075-15-2025-345.}}
\date{}
\begin{document}

\maketitle

\newenvironment{Biblio}{%
                  \renewcommand{\refname}{\footnotesize REFERENCES}%
                  \begin{thebibliography}{99}%
                  \itemsep -0.2em}%
                  {\end{thebibliography}}

\def\inff{\mathop{\smash\inf\vphantom\sup}}
\renewcommand{\le}{\leqslant}
\renewcommand{\ge}{\geqslant}
\newcommand{\sgn}{\mathrm {sgn}\,}
\newcommand{\inter}{\mathrm {int}\,}
\newcommand{\dist}{\mathrm {dist}}
\newcommand{\supp}{\mathrm {supp}\,}
\newcommand{\R}{\mathbb{R}}
\newcommand{\Z}{\mathbb{Z}}
\newcommand{\N}{\mathbb{N}}
\newcommand{\Q}{\mathbb{Q}}
\theoremstyle{plain}
\newtheorem{Trm}{Theorem}
\newtheorem{trma}{Theorem}

\newtheorem{Cor}{Corollary}
\newtheorem{Lem}{Lemma}
\newtheorem{Sta}{Proposition}

\theoremstyle{definition}
\newtheorem{Def}{Definition}
\newtheorem{Rem}{Remark}

\newtheorem{Sup}{Assumption}
\newtheorem{Supp}{Assumption}
\newtheorem{Not}{Notation}
\newtheorem{Exa}{Example}
\renewcommand{\proofname}{\bf Proof}
\renewcommand{\thetrma}{\Alph{trma}}
\renewcommand{\theSupp}{\Alph{Supp}}

\begin{abstract}
In this paper, we obtain order estimates for the Kolmogorov widths of an intersection of a finite family of Besov classes $SB_{p_j,\theta_j}^{\overline{r}_j}(\mathbb{T}^d)$ with dominating mixed smoothness in a Besov space $B_{q,\sigma}^{\overline{l}}(\mathbb{T}^d)$ in the case $2<q, \, \sigma <\infty$ when the parameters satisfy certain conditions of general position.
\end{abstract}


\section{Introduction}

In \cite{galeev1, galeev2, vas_int_sob} the problem of estimating the Kolmogorov widths of an intersection of Sobolev classes in $L_q$ on a one‐dimensional torus was studied. Here we consider the similar problem of estimating the Kolmogorov widths of an intersection of periodic Besov classes with dominating mixed smoothness $SB^{\overline{r}_j}_{p_j,\theta_j}(\mathbb{T}^d)$ in $B^{\overline{l}}_{q,\sigma}(\mathbb{T}^d)$ for $2<q, \, \sigma <\infty$. The proof employs the recent result \cite{vas_mix_sev} about the estimates for the Kolmogorov widths of an intersection of finite‐dimensional balls in a mixed norm.

We give necessary definitions and notation.

Let $d\in \N$, $\mathbb{T}^d= [0, \, 2\pi]^d$ (the endpoints of the intervals $[0, \, 2\pi]$ are ``glued''). 
We denote by $(\cdot, \, \cdot)$ the standard inner product on $\R^d$. Given a function $f\in L_1(\mathbb{T}^d)$, we write its Fourier series: $$f \mapsto \sum \limits _{\overline{k}\in \Z^d} c_{\overline{k}}(f) e^{i(\overline{k}, \, \cdot)}.$$

Let $\eta_0: \R \rightarrow [0, \, 1]$ be an infinitely smooth function such that $\eta_0|_{[‐1, \, 1]} = 1$, $\eta_0|_{\R \backslash [‐3/2, \, 3/2]} = 0$. Given $l\in \N$, we set $\eta_l(t) = \eta_0(2^{‐l}t) ‐ \eta_0(2^{‐l+1}t)$, $t\in \R$.

For $\overline{m}=(m_1, \, \dots, \, m_d)\in \N^d$, $\overline{\xi} = (\xi_1, \, \dots, \, \xi_d)$ we denote $\eta _{\overline{m}}(\overline{\xi}) = \prod _{j=1}^d \eta_{m_j}(\xi_j)$. Let
$$
\Delta_{\overline{m}}f(\cdot) = \sum \limits _{\overline{k} \in \Z^d} \eta _{\overline{m}}(\overline{k}) c_{\overline{k}}(f) e^{i (\overline{k}, \, \cdot)}, \quad f\in L_1(\mathbb{T}^d).
$$

\begin{Def}
\label{besov_sp_def}
Let $\overline{r}\in (0, \, \infty)^d$, $1\le p\le \infty$, $1\le \theta \le \infty$. Given $x\in L_p(\mathbb{T}^d)$, we set $$\|x\|_{B^{\overline{r}}_{p,\theta}(\mathbb{T}^d)} = \Bigl\|\bigl\{ 2^{(\overline{m}, \, \overline{r})}\|\Delta _{\overline{m}}x\|_{L_p(\mathbb{T}^d)}\bigr\}_{\overline{m}\in \N^d}\Bigr\|_{l_\theta}.$$ The Besov space with dominating mixed smoothness $B^{\overline{r}}_{p,\theta}(\mathbb{T}^d)$ is defined by the formula
$$
B^{\overline{r}}_{p,\theta}(\mathbb{T}^d) = \{x\in L_p(\mathbb{T}^d):\, \|x\|_{B^{\overline{r}}_{p,\theta}(\mathbb{T}^d)} < \infty\};
$$
it is equipped with the norm $\|\cdot\|_{B^{\overline{r}}_{p,\theta}(\mathbb{T}^d)}$. The unit ball is denoted by $SB^{\overline{r}}_{p,\theta}(\mathbb{T}^d)$.
\end{Def}

The Besov spaces with dominating mixed smoothness first appeared in the paper of Amanov \cite{amanov} (see also \cite{amanov1}). These spaces were defined in terms of differences, similarly to the classical Besov spaces \cite{besov_space}. In \cite{itogi_nt, liz_nik, tr_scm} various definitions of the Besov spaces in terms of Fourier series were given; see also, e.g., \cite{romanyuk, galeev_besov, bazarkh2}. In \cite{tr_scm, bazarkh2} the above definition \ref{besov_sp_def} or its analogue for Besov spaces on $\R^d$ was used. (To be more precise, in \cite{bazarkh2}  the more general class of spaces was considered; it included the classical Besov spaces and the Besov spaces with dominating mixed smoothness.) In \cite{itogi_nt, galeev_besov}, instead of $\Delta_{\overline{m}}$, the operators $\delta_{\overline{m}}$ defined by formula $$\delta_{\overline{m}} f(\cdot)= \sum \limits _{\overline{k}\in \square_{\overline{m}}} c_{\overline{k}}(f)e^{i(\overline{k},\, \cdot)}, \quad \square_{\overline{m}}=\{k\in \Z^d:\; 2^{m_j‐1}\le |k_j|<2^{m_j}, \; 1\le j\le d\}$$ were considered. It is well‐known that for $1<p<\infty$ this definition is equivalent to Definition \ref{besov_sp_def}. For $p=1$ and $p=\infty$, instead of $\delta_{\overline{m}}$ the convolution with the de la Vall\'ee Poussin kernel was used (see \cite{romanyuk}).

We also mention the paper \cite{haroske_skr}, where the weighted Besov spaces on $\R^d$ with Muckenhoupt weights were studied.

In what follows, given $\overline{m}=(m_1, \, \dots, \, m_d)$, we denote
$$
m= m_1+\dots + m_d.
$$

From \cite[Theorem 3.1]{bazarkh2} it follows that, for $\overline{r}\in (0, \, \infty)^d$, the space $B^{\overline{r}}_{p,\theta}(\mathbb{T}^d)$ is isomorphic to the sequence space
$$
b^{\overline{r}}_{p,\theta} = \{x=(x_{\overline{m},j})_{\overline{m}\in \Z_+^d, \, j\in J_{\overline{m}}}:\; \|x\|_{b^{\overline{r}}_{p,\theta}}<\infty\},
$$
where the index sets $J_{\overline{m}}$ satisfy the inequalities
\begin{align}
\label{card_j_m}
c_1(d)2^m\le \# J_{\overline{m}} \le c_2(d) 2^m
\end{align}
for some $c_1(d)>0$, $c_2(d)>0$, and the
norm $\|\cdot\|_{b^{\overline{r}}_{p,\theta}}$ is defined by
\begin{align}
\label{x_b_norm}
\|x\|_{b^{\overline{r}}_{p,\theta}} = \Bigl( \sum \limits _{\overline{m}\in \Z_+^d} 2^{\theta(\overline{r}, \, \overline{m})‐m\theta/p} \Bigl( \sum \limits _{j\in J_{\overline{m}}} |x_{\overline{m},j}| ^p\Bigr)^{\theta/p}\Bigr)^{1/\theta}
\end{align}
for finite $p$, $\theta$; if $p=\infty$ or $\theta=\infty$, the definition of the norm is naturally modified. This isomorphism is constructed in terms of wavelet decompositions of functions (here we denote the system of wavelets by $\{\psi_{\overline{m},j}\}_{\overline{m}\in \Z_+^d, \, j\in J_{\overline{m}}}$; for details, see \cite{bazarkh2}); more precisely, it is defined as follows: $f\mapsto (\langle f, \, \psi_{\overline{m},j}\rangle)_{\overline{m}\in \Z_+^d, \, j\in J_{\overline{m}}}$ (here $\langle \cdot, \, \cdot\rangle$ is the standard inner product in $L_2(\mathbb{T}^d)$). We see that the formula does not depend on $\overline{r}$, $p$ and $\theta$.

We also notice that, for $1<p<\infty$, in order to construct the isomorphism of the Besov space and the sequence space we can use the equivalent definition with the operators $\delta_{\overline{m}}$ and apply Marcinkiewicz multipliers theorem (see \cite[Ch. 2, section 2.3, Theorem 18]{itogi_nt}, \cite[Chapter III, section 15.3]{besov_iljin_nik}, \cite[section 1.5.3]{nikolski_sm}) and Marcinkiewicz discretization theorem (see \cite[Theorem B]{galeev2}, \cite[Vol. 2, Ch. X, Theorem 7.5]{zigmund}).

Now we give the definition of the Kolmogorov widths.
\begin{Def}
Let $X$ be a normed space, and let $M\subset X$, $n\in \Z_+$. The Kolmogorov $n$-width of the set $M$ in the space $X$ is defined by
$$
d_n(M, \, X) = \inf _{L\in {\cal L}_n(X)} \sup _{x\in M} \inf
_{y\in L} \|x-y\|;
$$
here ${\cal L}_n(X)$ is the family of all subspaces in $X$ of dimension at most $n$.
\end{Def}

The problem of estimating the widths of Sobolev and Nikolskii classes was studied in \cite{babenko, mityagin, makovoz, ismagilov, majorov, bib_kashin, kashin_sma, kulanin1, kulanin2, galeev85, galeev87, teml3, teml4, mal25, tikh60} (the paper \cite{mal25} is a recent result, which completely solves the Sobolev width problem on an interval; the critical case of $W^1_1$ in $L_q$, $2<q<\infty$, is considered). For more details, see \cite{itogi_nt, kniga_pinkusa, teml_book, alimov_tsarkov}.

The study of the problem of the Kolmogorov widths of periodic Besov classes with dominating mixed smoothness in the space $L_q(\mathbb{T}^d)$ was pioneered by Galeev \cite{galeev0} and Romanyuk \cite{romanyuk0} and was continued in \cite{romanyuk, galeev_besov, bazarkh1}; see also \cite{galeev11}. In \cite{mal21} order estimates for  $d_n(B^1_{1,\theta}[0, \, 1], \, L_q[0, \, 1])$ were obtained. In \cite{stasyuk} the problem of Kolmogorov widths of classes $MB^{\Omega}_{p,\theta}$ in $L_q$ was studied. In \cite{akishev} instead of the Lebesgue space the Lorentz space was considered, and the Besov classes were anisotropic. In \cite{nguen, vkn_b} the problems of estimating the Gelfand and Bernstein widths of Besov classes with dominating mixed smoothness were studied.

The problem on widths of Besov classes with dominated mixed smoothness in the Besov space 
is simpler, from the point of view of discretization, than the 
problem on $L_q$-wdiths,
but the solution of the former problem should be based on the knowledge of widths 
of finite-dimensional balls in the mixed norm. A~more detailed account of estimates of widths of 
the finite-dimensional balls $B_{p,\theta}^{t,k}$ in the space $l_{q,\sigma}^{t,k}$ 
will be given in~\S 3 (see Theorem~ \ref{1mixed} and Remark~\ref{eeee_rem}). In \cite{vas_mix_sev} estimates for the widths of intersections of finite‐dimensional balls in a mixed norms were obtained for $2\le q, \, \sigma<\infty$; thus, we will consider the more general problem of the widths of intersections of Besov classes.

The problem of Kolmogorov widths of an intersection of Sobolev classes in $L_q$ was studied in \cite{galeev1, galeev2, vas_int_sob}.

Now we formulate the main results of this paper.

Let $s\in \N$, $1\le p_j\le \infty$, $1\le \theta_j\le \infty$, 
$\overline{r}_j=(r_j, \, \dots, \, r_j)\in (0, \, \infty)^d$, $j=1, \, \dots, \, s$, $\overline{l}=(l, \, \dots, \, l)\in (0, \, \infty)^d$. Suppose that $2< q, \, \sigma<\infty$. We consider the problem of estimating the Kolmogorov widths of an intersection of Besov classes\footnote{We need the conditions $r_j>0$, $l>0$ in order to apply the result of \cite{bazarkh2} about the isomorphisms of Besov spaces and the sequence spaces. D.B. Bazarkhanov reported that this result holds for arbitrary smoothness, but this assertion is not published yet. Hence we formulate the main results only for the spaces of strictly positive smoothness.}
$$
d_n(\cap _{j=1}^s SB^{\overline{r}_j}_{p_j,\theta_j}(\mathbb{T}^d), \, B^{\overline{l}}_{q, \sigma}(\mathbb{T}^d)).
$$

We give notation for order equalities and inequalities. Let $X$, $Y$ be sets, and let $f_1$, $f_2:\ X\times Y\rightarrow \mathbb{R}_+$. We write $f_1(x, \, y)\underset{y}{\lesssim} f_2(x, \, y)$ (or $f_2(x, \, y) \underset{y}{\gtrsim} f_1(x, \, y)$) if for each $y\in Y$ there is $c(y)>0$ such that $f_1(x, \, y)\le c(y)f_2(x, \, y)$ for all $x\in X$; $f_1(x, \, y)\underset{y}{\asymp} f_2(x, \, y)$ if $f_1(x, \, y)
\underset{y}{\lesssim} f_2(x, \, y)$ and $f_2(x, \,
y)\underset{y}{\lesssim} f_1(x, \, y)$.

We denote $\mathfrak{Z} = (d, \, s, \, r_1, \, \dots, \, r_s, \, l, \, q, \, \sigma, \, p_1, \, \dots, \, p_s, \, \theta_1, \, \dots, \, \theta_s)$.

By $Sb^{\overline{r}}_{p,\theta}$, we denote the unit ball in the space $b^{\overline{r}}_{p,\theta}$ (see \eqref{x_b_norm}).

It was mentioned above that the formulas defining the isomorphism of the spaces $B^{\overline{r}}_{p,\theta}(\mathbb{T}^d)$ and $b^{\overline{r}}_{p,\theta}$ are the same for all $\overline{r}$, $p$, $\theta$; therefore,
\begin{align}
\label{discr_eq} d_n(\cap _{j=1}^s SB^{\overline{r}_j}_{p_j,\theta_j}(\mathbb{T}^d), \, B^{\overline{l}}_{q, \sigma}(\mathbb{T}^d)) \underset{\mathfrak{Z}}{\asymp} d_n(\cap _{j=1}^s Sb^{\overline{r}_j}_{p_j,\theta_j}, \, b^{\overline{l}}_{q, \sigma})
\end{align}
if the left‐ or the right‐hand side of \eqref{discr_eq} is finite for $n=0$.

We set $\overline{\alpha}_j = (\alpha_j, \, \dots, \, \alpha_j) \in \R^d$, where $\alpha_j= r_j‐l$.

Throughout this paper we suppose that the following conditions hold:
\begin{align}
\label{non_emb} \alpha_1<\alpha_2<\dots < \alpha_s, \quad \alpha_1 ‐ \frac{1}{p_1} > \alpha_2 ‐ \frac{1}{p_2} > \dots > \alpha_s ‐ \frac{1}{p_s};
\end{align}
the analogue of these conditions appeared in \cite{vas_int_sob}, where the problem of estimating of intersections of periodic Sobolev classes on $\mathbb{T}^1$ and intersections of Sobolev classes on a John domain was considered.

\begin{Rem}
Suppose that all values $\alpha_j$ are different, as well as all values $\alpha_j‐\frac{1}{p_j}$ are different $(j=1, \, \dots, \, s)$. Then the first chain of the inequalities in \eqref{non_emb} can be obtained by the appropriate numeration. If the second chain of the inequalities in \eqref{non_emb} fails, then $\alpha_i>\alpha_j$, $\alpha_i‐\frac{1}{p_i} > \alpha_j ‐ \frac{1}{p_j}$ for some $i>j$. Applying the isomorphisms of the spaces $B^{\overline{r}}_{p,\theta}(\mathbb{T}^d)$ and $b^{\overline{r}}_{p,\theta}$, we see that $SB^{\overline{r}_i}_{p_i,\theta_i}(\mathbb{T}^d) \subset M\cdot SB^{\overline{r}_j}_{p_j,\theta_j}(\mathbb{T}^d)$, where $M=M(\mathfrak{Z})\in (0, \, \infty)$. Hence the widths of the intersection of the balls over $k\in\{1, \, \dots, \, s\}$ have the same order as the widths of the intersection over $k\in \{1, \, \dots, \, s\}\backslash \{j\}$. Therefore, we can exclude ``unnecessary'' balls and reduce the problem to the case when \eqref{non_emb} holds. Thus, only the ``degenerate'' case will not be studied in the present paper, when $\alpha_i=\alpha_j$ or $\alpha_i‐\frac{1}{p_i} = \alpha_j‐\frac{1}{p_j}$ for some $i\ne j$.
\end{Rem}

From \eqref{non_emb} it follows that
\begin{align}
\label{p1ps} p_1>p_2> \dots >p_s.
\end{align}

Given $2< q<\infty$, $1\le p\le \infty$, we set 
\begin{align}
\label{om_pq}
\omega_{p,q} = \begin{cases} 0 & \text{for }p> q, \\ \frac{1/p-1/q}{1/2-1/q} & \text{for }2<p\le q, \\ 1 & \text{for }1\le p\le 2.\end{cases} 
\end{align}

In what follows we denote $\log x := \log_2 x$.

In Theorems \ref{main1}‐‐\ref{main5}, \ref{main6} the order estimate for the widths will be given by the formula
\begin{align}
\label{teor_dn_est} d_n(\cap _{j=1}^s SB^{\overline{r}_j}_{p_j,\theta_j}(\mathbb{T}^d), \, B^{\overline{l}}_{q, \sigma}(\mathbb{T}^d)) \underset{\mathfrak{Z}}{\asymp} n^{‐\alpha_*}(\log n)^{(d‐1)\beta_*};
\end{align}
the values $\alpha_*$ and $\beta_*$ will be written explicitly.

First we consider two simple cases, when $p_j\ge q$ for all $j=1, \, \dots, \, s$ or $p_j\le 2$ for all $j=1, \, \dots, \, s$.
\begin{Trm}
\label{main1} Let $d\in \N$, $s\in \N$, $2< q, \, \sigma<\infty$, $l>0$, $r_j>0$, $q\le p_j\le \infty$, $1\le \theta_j\le \infty$, $\alpha_j:=r_j‐l$, $j=1, \, \dots, \, s$, $\alpha_s>0$. Suppose that \eqref{non_emb} holds. In the case $1\le \theta_s<\sigma$ we also assume that $\alpha_s \ne \frac{\omega_{\theta_s,\sigma}}{2}$. We set
$$
\alpha_*=\alpha_s, \quad \beta_* = \begin{cases} \alpha_* +\frac{1}{\sigma}‐\frac{1}{\max\{\theta_s, \, 2\}} & \text{ if }\alpha_s > \frac{\omega_{\theta_s,\sigma}}{2}, \\ \frac{2\alpha_*}{\sigma} & \text{ if } \alpha_s < \frac{\omega_{\theta_s,\sigma}}{2}.\end{cases}
$$
Then \eqref{teor_dn_est} holds.
\end{Trm}

\begin{Trm}
\label{main2} Let $d\in \N$, $s\in \N$, $2< q, \, \sigma<\infty$, $l>0$, $r_j>0$, $1\le p_j\le 2$, $1\le \theta_j\le \infty$, $\alpha_j:=r_j‐l$, $j=1, \, \dots, \, s$. Let $\alpha_1+\frac 1q‐\frac{1}{p_1}>0$, $\alpha_1 \ne \frac{1}{p_1}$. Suppose that \eqref{non_emb} holds.
In the case $2<\theta_1<\sigma$ we also assume that $\alpha_1+\frac 1q‐\frac{1}{p_1} \ne \frac {\omega _{\theta_1,\sigma}}{q}$.
We set
$$
\alpha_* = \begin{cases} \alpha_1+\frac 12‐\frac{1}{p_1} & \text{ if }\alpha_1 > \frac{1}{p_1}, \\ \frac{q}{2}\Bigl(\alpha_1+\frac 1q‐\frac{1}{p_1}\Bigr) & \text{ if }\alpha_1 < \frac{1}{p_1},\end{cases}
$$
$$
\beta_* = \begin{cases} \alpha_*+\frac{1}{\sigma}‐ \frac{1}{\max\{\theta_1, \, 2\}}, & \alpha_1> \frac{1}{p_1}, \\ \alpha_* + \frac{1}{\sigma} ‐ \frac{1}{\theta_1}, & \alpha_1<\frac{1}{p_1}, \;  \alpha_1+\frac 1q‐\frac{1}{p_1} > \frac{\omega_{\theta_1,\sigma}}{q}, \\
\frac{2\alpha_*}{\sigma}, & \alpha_1< \frac{1}{p_1}, \; \alpha_1+\frac 1q‐\frac{1}{p_1} < \frac{\omega_{\theta_1,\sigma}}{q}.
\end{cases}
$$
Then \eqref{teor_dn_est} holds.
\end{Trm}

Now we consider the more complicated case, when 
\begin{align}
\label{ijk_ne_i_k}
\{j\in \overline{1, \, s}:\; p_j<q\}\ne \varnothing, \; \{j\in \overline{1, \, s}:\; p_j>2\}\ne \varnothing.
\end{align}
In addition, we suppose that
\begin{gather}
\label{p_ne_2q} p_j\notin \{2, \, q\}, \; j=1, \, \dots, \, s, \\
\label{frac_aj} \frac{\alpha_{i_1}‐\alpha_{j_1}}{1/p_{i_1}‐1/p_{j_1}} \ne \frac{\alpha_{i_2}‐\alpha_{j_2}}{1/p_{i_2}‐1/p_{j_2}}, \; i_1>j_1, \; i_2>j_2, \; (i_1,\, j_1) \ne (i_2, \, j_2).
\end{gather}
From \eqref{p_ne_2q} it follows that $\{1, \, \dots, \, s\} = I\sqcup J\sqcup K$, where
\begin{align}
\label{ijk_def}
I = \{j\in \overline{1, \, s}:\; p_j>q\}, \; J = \{j\in \overline{1, \, s}:\; 2<p_j<q\}, \; K= \{j\in \overline{1, \, s}:\; p_j<2\}.
\end{align}

As in \cite{vas_int_sob}, we use the following notation.

We define the numbers $\lambda_{i,j}$ ($i\in I$, $j\in J\sqcup K$) and $\tilde \lambda_{i,j}$ ($i\in I\sqcup J$, $j\in K$) by the equations
\begin{align}
\label{lam_ij_def} \frac 1q = \frac{1‐\lambda_{i,j}}{p_i} + \frac{\lambda_{i,j}}{p_j}, \quad \frac 12 = \frac{1‐\tilde\lambda_{i,j}}{p_i} + \frac{\tilde\lambda_{i,j}}{p_j}.
\end{align}
The numbers $i_0\in I$, $j_0\in J\sqcup K$ (for $I\ne \varnothing$) and $i_1\in I\sqcup J$, $j_1\in K$ (for $K\ne \varnothing$) are defined by
\begin{align}
\label{i01j01}
\begin{array}{c}
(i_0, \, j_0) = \operatorname{argmax} _{i\in I, \, j\in J\sqcup K}\{(1‐\lambda_{i,j})\alpha_i + \lambda_{i,j} \alpha_j\}, \\
(i_1, \, j_1) = \operatorname{argmax} _{i\in I\sqcup J, \, j\in K}\{(1‐\tilde\lambda_{i,j})\alpha_i + \tilde\lambda_{i,j} \alpha_j\}.
\end{array}
\end{align}
The functions $h_0$, $h_1$, $h_2:[1, \, q/2] \rightarrow \R\cup\{‐\infty\}$ are defined as follows:
\begin{align}
\label{h0_def}
h_0(t) = \begin{cases}t((1‐\lambda_{i_0,j_0})\alpha_{i_0} + \lambda _{i_0,j_0} \alpha_{j_0})&\text{if}\; I\ne \varnothing, \\ ‐\infty &\text{if}\; I= \varnothing,\end{cases}
\end{align}
\begin{align}
\label{h1_def}
h_1(t) = \begin{cases}t\Bigl((1‐\tilde\lambda_{i_1,j_1})\alpha_{i_1} + \tilde\lambda _{i_1,j_1} \alpha_{j_1} ‐\frac 12\Bigr)+\frac 12& \text{if}\; K\ne \varnothing, \\ ‐\infty & \text{if}\; K= \varnothing,\end{cases}
\end{align}
\begin{align}
\label{h2_def}
h_2(t) = \begin{cases}\max _{j\in J} \varphi_j(t) & \text{if}\; J\ne \varnothing, \\ ‐\infty & \text{if}\; J= \varnothing,\end{cases}
\end{align}
where 
\begin{align}
\label{phi_j_def}
\varphi_j(t) = t\left(\alpha_j ‐\frac 12 \cdot \frac{1/p_j‐1/q}{1/2‐1/q}\right) +\frac 12 \cdot \frac{1/p_j‐1/q}{1/2‐1/q}, \quad j=1, \, \dots, \, s.
\end{align}
We set 
\begin{align}
\label{h_def_012}
h(t) = \max \{h_0(t), \, h_1(t), \, h_2(t)\}, \quad t\in [1, \, q/2].
\end{align}

In \cite{vas_int_sob} it was proved that if the function $h$ has the unique minimum point $t_*\in [1, \, q/2]$ and $h(t_*)>0$, then $d_n(\cap _{j=1}^s W^{\alpha_j}_{p_j}(\mathbb{T}), \, L_q(\mathbb{T}))$ has the order $n^{‐h(t_*)}$. Considering the problem of the widths $d_n(\cap _{j=1}^s SB^{\overline{r}_j}_{p_j,\theta_j}(\mathbb{T}^d), \, B^{\overline{l}}_{q, \sigma}(\mathbb{T}^d))$ we need to study the function $h$ in more details. We will prove the following assertion.

\begin{Sta}
\label{h_prop} Let $2<q<\infty$ and let the conditions \eqref{non_emb}, \eqref{ijk_ne_i_k}, \eqref{p_ne_2q}, \eqref{frac_aj} hold. Let $1<t_1<t_2<\dots<t_{L‐1}<\frac q2$ be the corner points of $h$, $t_0:=1$, $t_L:=q/2$. Then the following assertions hold:
\begin{enumerate}
\item if $2\le l\le L‐1$, then $h|_{[t_{l‐1}, \, t_l]} = \varphi_{j(l)}|_{[t_{l‐1}, \, t_l]}$ for some $j(l)\in J$;

\item if $K = \varnothing$, then $h|_{[t_0, \, t_1]} = \varphi_{j(1)}|_{[t_0, \, t_1]}$, where $j(1)=s\in J$;

\item if $K\ne \varnothing$, then $h|_{[t_0, \, t_1]} = h_1|_{[t_0, \, t_1]}$; if, in addition, $i_1\in J$, then $h(t)=\varphi_{i_1}(t)$ in a right semi‐neighborhood of the point $t_1$; if $i_1\in I$, then $L=2$, $(i_0, \, j_0)= (i_1, \, j_1)$ and $h|_{[t_1, \, t_2]} = h_0|_{[t_1, \, t_2]}$;

\item if $I=\varnothing$, then $h|_{[t_{L‐1}, \, t_L]} = \varphi_{j(L)}|_{[t_{L‐1}, \, t_L]}$, where $j(L)=1\in J$;

\item if $I\ne \varnothing$, then $h|_{[t_{L‐1}, \, t_L]} = h_0|_{[t_{L‐1}, \, t_L]}$; if, in addition, $j_0\in J$, then $h(t)=\varphi_{j_0}(t)$ in a left semi‐neighborhood of the point $t_{L‐1}$; if $j_0\in K$, then $L=2$, $(i_1, \, j_1) = (i_0, \, j_0)$ and $h|_{[t_0, \, t_1]} = h_1|_{[t_0, \, t_1]}$;

\item the sequence $\{j(l)\}_l$ strictly decreases in $l$.
\end{enumerate}
\end{Sta}

The indices $j(l)$ can be found by the algorithm from \cite[the end of \S 2]{vas_int_sob}. They are defined for $l\in \{2, \, 3, \, \dots, \, L‐1\}$, as well as for $l=1$ in the case $K=\varnothing$ and for $l=L$ in the case $I=\varnothing$.

\begin{Rem}
\label{i1j2_j0jl1}
In assertion 3 of Proposition \ref{h_prop}, for $i_1\in J$, we get $j(2)=i_1$ (for $L\ge 3$, see assertion 1, and for $L=2$, see assertions 4, 5). In assertion 5, for $j_0\in J$, we get $j(L‐1) = j_0$ (for $L\ge 3$, see assertion 1, for $L=2$, see assertions 2, 3). 
\end{Rem}

We first consider the cases when the strict minimum of the function $h$ is attained at the endpoint of $[1, \, q/2]$.

In Theorems \ref{main3}, \ref{main4} the minimum of $h$ is attained at $t=1$.

\begin{Trm}
\label{main3} Let $d\in \N$, $s\in \N$, $2< q, \, \sigma<\infty$, $l>0$, $r_j>0$, $\alpha_j:=r_j‐l$, $1\le p_j\le\infty$, $1\le \theta_j\le \infty$, $j=1, \, \dots, \, s$. Suppose that \eqref{non_emb}, \eqref{ijk_ne_i_k}, \eqref{p_ne_2q}, \eqref{frac_aj} hold. Let $K = \varnothing$, $\alpha_s > \frac{\omega_{p_s,q}}{2}$. In the case $\omega_{p_s,q}<\omega_{\theta_s,\sigma}$ we also assume that $\alpha_s\ne \frac{\omega_{\theta_s, \sigma}}{2}$. We set
$$
\alpha_*=\alpha_s, \quad \beta_* = \begin{cases} \alpha_*+ \frac{1}{\sigma}‐ \frac{1}{\max\{\theta_s, \, 2\}} & \text{if }\alpha_s >\frac{\max\{\omega_{\theta_s,\sigma}, \, \omega_{p_s,q}\}}{2}, \\ \frac{2\alpha_*}{\sigma} & \text{if } \alpha_s <\frac{\max\{\omega_{\theta_s,\sigma}, \, \omega_{p_s,q}\}}{2}. \end{cases}
$$
Then \eqref{teor_dn_est} holds.
\end{Trm}

\begin{Trm}
\label{main4} Let $d\in \N$, $s\in \N$, $2< q, \, \sigma<\infty$, $l>0$, $r_j>0$, $\alpha_j:=r_j‐l$, $1\le p_j\le\infty$, $1\le \theta_j\le \infty$, $j=1, \, \dots, \, s$. Suppose that \eqref{non_emb}, \eqref{ijk_ne_i_k}, \eqref{p_ne_2q}, \eqref{frac_aj} hold. Let $K \ne \varnothing$. We define the numbers $\tilde{\lambda}_{i,j}$ by \eqref{lam_ij_def}, and the indices $i_1$, $j_1$, by \eqref{i01j01}. Let 
\begin{align}
\label{3al_st}
\alpha_*:= (1‐\tilde \lambda _{i_1,j_1})\alpha_{i_1} + \tilde \lambda _{i_1,j_1} \alpha_{j_1}> \frac 12.
\end{align}
We denote the numbers $\theta_*$ and $\beta_*$ by 
\begin{align}
\label{3bet_st}
\frac{1}{\theta_*} = \frac{1‐\tilde \lambda _{i_1,j_1}}{\theta_{i_1}} + \frac{\tilde \lambda _{i_1,j_1}}{\theta_{j_1}},
\quad
\beta_* = \alpha_*+ \frac{1}{\sigma}‐ \frac{1}{\max\{\theta_*, \, 2\}}.
\end{align}
Then \eqref{teor_dn_est} holds.
\end{Trm}

In Theorems \ref{main5}, \ref{main_non_emb} the minimum of the function $h$ is attained at $t=q/2$.

\begin{Trm}
\label{main5} Let $d\in \N$, $s\in \N$, $2< q, \, \sigma<\infty$, $l>0$, $r_j>0$, $\alpha_j:=r_j‐l$,  $1\le p_j\le\infty$, $1\le \theta_j\le \infty$, $j=1, \, \dots, \, s$. Suppose that \eqref{non_emb}, \eqref{ijk_ne_i_k}, \eqref{p_ne_2q}, \eqref{frac_aj} hold. Let $I=\varnothing$, $\frac{1}{p_1}‐\frac 1q<\alpha_1< \frac{\omega_{p_1,q}}{2}$. In the case $\omega_{p_1,q} > \omega_{\theta_1,\sigma}$ we also suppose that $\alpha_1 +\frac 1q ‐ \frac{1}{p_1} \ne \frac{\omega_{\theta_1,\sigma}}{q}$. We set $$\alpha_* = \frac q2 \Bigl(\alpha_1+\frac 1q‐\frac{1}{p_1}\Bigr),$$
$$
\beta_* = \begin{cases} \alpha_*+\frac{1}{\sigma} ‐ \frac{1}{\theta_1}, & \alpha_1+ \frac 1q ‐ \frac{1}{p_1}>\frac{\min\{\omega_{p_1,q}, \, \omega_{\theta_1,\sigma}\}}{q}, \\ \frac{2\alpha_*}{\sigma}, & \alpha_1 + \frac 1q ‐ \frac{1}{p_1} <\frac{\min\{\omega_{p_1,q}, \, \omega_{\theta_1,\sigma}\}}{q}. \end{cases}
$$
Then \eqref{teor_dn_est} holds.
\end{Trm}

\begin{Trm}
\label{main_non_emb} Let $d\in \N$, $s\in \N$, $2< q, \, \sigma<\infty$, $l>0$, $r_j>0$, $\alpha_j:=r_j‐l$, $1\le p_j\le\infty$, $1\le \theta_j\le \infty$, $j=1, \, \dots, \, s$. Suppose that \eqref{non_emb}, \eqref{ijk_ne_i_k}, \eqref{p_ne_2q}, \eqref{frac_aj} hold. Let $I \ne \varnothing$. We define the numbers $\lambda_{i,j}$ by \eqref{lam_ij_def}, and the indices $i_0$, $j_0$, by  \eqref{i01j01}. Let 
$(1‐\lambda _{i_0,j_0})\alpha_{i_0} + \lambda _{i_0,j_0} \alpha_{j_0}< 0$.
Then
$$
d_n(\cap _{j=1}^s SB^{\overline{r}_j}_{p_j,\theta_j}(\mathbb{T}^d), \, B^{\overline{l}}_{q, \sigma} (\mathbb{T}^d)) =\infty.
$$
\end{Trm}

\begin{Rem}
The formulas for estimates of the widths $d_n(SB_{p,\theta}^{\overline{r}}(\mathbb{T}^d), \, B_{q,\sigma}^{\overline{l}}(\mathbb{T}^d))$ for $2<q, \, \sigma<\infty$ are the particular cases of Theorems \ref{main1}, \ref{main2}, \ref{main3}, \ref{main5} for $s=1$.
\end{Rem}

Now, let the strict minimum of the function $h$ be attained at the point $t_{l_*}$ for some $1\le l_*\le L‐1$ (see Proposition \ref{h_prop}). We define the numbers $\hat \alpha_i$, $\hat p_i$, $\hat \theta_i$, $i=1, \, 2$, as follows:
\begin{enumerate}
\item if $h|_{[t_{l_*‐1}, \, t_{l_*}]} = \varphi_{j(l_*)}|_{[t_{l_*‐1}, \, t_{l_*}]}$, $j(l_*)\in J$, then $\hat \alpha_1=\alpha_{j(l_*)}$, $\hat p_1 = p_{j(l_*)}$, $\hat \theta_1 = \theta_{j(l_*)}$;
\item if $l_*=1$, $h|_{[t_0, \, t_1]} = h_1|_{[t_0, \, t_1]}$, then $\hat \alpha_1 = (1‐\tilde \lambda_{i_1,j_1})\alpha _{i_1} + \tilde \lambda_{i_1,j_1} \alpha _{j_1}$, $\hat p_1 = 2$, $\frac{1}{\hat \theta_1} = \frac{1‐\tilde \lambda_{i_1,j_1}}{\theta_{i_1}} + \frac{\tilde \lambda_{i_1,j_1}}{\theta _{j_1}}$;
\item if $h|_{[t_{l_*}, \, t_{l_*+1}]} = \varphi_{j(l_*+1)}|_{[t_{l_*}, \, t_{l_*+1}]}$, $j(l_*+1)\in J$, then $\hat \alpha_2=\alpha_{j(l_*+1)}$, $\hat p_2 = p_{j(l_*+1)}$, $\hat \theta_2 = \theta_{j(l_*+1)}$;
\item if $l_* = L‐1$, $h|_{[t_{L‐1}, \, t_L]} = h_0 |_{[t_{L‐1}, \, t_L]}$, then $\hat \alpha_2 = (1‐ \lambda_{i_0,j_0})\alpha _{i_0} + \lambda_{i_0,j_0} \alpha _{j_0}$, $\hat p_2 = q$, $\frac{1}{\hat \theta_2} = \frac{1‐\lambda_{i_0,j_0}}{\theta_{i_0}} + \frac{ \lambda _{i_0,j_0}} {\theta _{j_0}}$.
\end{enumerate}

\begin{Rem}
\label{rem_pi_hat}
We have $\hat p_i\in [2, \, q]$, $i=1, \, 2$, and $\hat p_1< \hat p_2$ (see assertion 6 of Proposition \ref{h_prop} and \eqref{p1ps}).
\end{Rem}

We set 
\begin{gather}
\label{al_st_def} \alpha_* = \frac{\hat \alpha_2\omega_{\hat p_1,q} ‐\hat \alpha_1 \omega_{\hat p_2,q}}{2(\hat \alpha_2 ‐ \hat \alpha_1) + \omega_{\hat p_1,q}‐\omega_{\hat p_2,q}}, \\
\label{a1_def} \begin{array}{c}A_1 = \frac{\hat \alpha_2 \omega_{\hat p_1,q}}{2} ‐ \frac{\hat \alpha_1 \omega_{\hat p_2,q}}{2} + \\ + \Bigl(\Bigl(\hat \alpha_1‐\frac{\omega_{\hat p_1,q}}{2}\Bigr)\Bigl(\frac{1}{\hat \theta_2}‐\frac{1}{\sigma}\Bigr) ‐ \Bigl(\hat \alpha_2‐\frac{\omega_{\hat p_2,q}}{2}\Bigr)\Bigl(\frac{1}{\hat \theta_1}‐ \frac{1}{\sigma} \Bigr)\Bigr), \end{array}
\\
\label{a2_def} A_2 = (\hat \alpha_2\omega_{\hat p_1,q} ‐ \hat \alpha_1\omega_{\hat p_2,q})/\sigma, \\
\label{b_def} B = \hat \alpha_2‐\hat \alpha_1 ‐ \frac{\omega_{\hat p_2,q}}{2} + \frac{\omega_{\hat p_1,q}}{2}, \\
\label{beta_st_i_def}\beta_*^1 = A_1/B, \quad \beta_*^2 = A_2/B,
\end{gather}
$$\omega'_{p,q} = \frac{1/p‐1/q}{1/2‐1/q}.$$

\begin{Trm}
\label{main6} Let $d\in \N$, $s\in \N$, $2< q, \, \sigma<\infty$, $l>0$, $r_j>0$, $\alpha_j := r_j‐l$, $1\le p_j\le \infty$, $1\le \theta_j\le \infty$, $j=1, \, \dots, \, s$. Suppose that \eqref{non_emb}, \eqref{ijk_ne_i_k}, \eqref{p_ne_2q}, \eqref{frac_aj} hold. Let $1\le l_*\le L‐1$, and let $t_{l_*}$ be the strict minimum point of $h$. Let the numbers $\alpha_*$, $\beta^1_*$, $\beta^2_*$ be defined by formulas \eqref{al_st_def}‐‐\eqref{beta_st_i_def}.
\begin{enumerate}
\item If $\omega'_{\hat p_1,q} \ge \omega'_{\hat \theta_1,\sigma}$, $\omega'_{\hat p_2,q} \ge \omega'_{\hat \theta_2,\sigma}$, we set $\beta_*= \beta_*^1$.

\item If $\omega'_{\hat p_1,q} \le \omega'_{\hat \theta_1,\sigma}$, $\omega'_{\hat p_2,q} \le \omega'_{\hat \theta_2,\sigma}$, we set $\beta_*= \beta_*^2$.

\item Let $(\omega'_{\hat p_1,q} ‐ \omega'_{\hat \theta_1,\sigma})(\omega'_{\hat p_2,q} ‐ \omega'_{\hat \theta_2,\sigma})<0$. We define the number $\hat\lambda$ by the equation $\omega'_{\hat p,q} = \omega' _{\hat \theta, \sigma}$, where 
\begin{align}
\label{hat_p_th_def}
\frac{1}{\hat p} = \frac{1‐\hat \lambda}{\hat p_1} + \frac{\hat \lambda}{\hat p_2}, \quad \frac{1}{\hat \theta} = \frac{1‐\hat \lambda}{\hat \theta_1} + \frac{\hat \lambda}{\hat \theta_2}.
\end{align}
We set 
\begin{align}
\label{hat_alp_def}
\hat \alpha = (1‐\hat \lambda) \hat \alpha_1 + \hat \lambda \hat \alpha_2.
\end{align}
Suppose that 
\begin{align}
\label{zeta_def}
\zeta:=\hat \alpha ‐\frac{\omega'_{\hat p,q}}{2} \ne 0.
\end{align}
We set $\beta_*=\beta_*^1$ if $(\omega'_{\hat p_1,q} ‐ \omega'_{\hat \theta_1,\sigma})\zeta>0$, $\beta_*=\beta_*^2$ if $(\omega'_{\hat p_1,q} ‐ \omega'_{\hat \theta_1,\sigma})\zeta<0$.
\end{enumerate}
Then \eqref{teor_dn_est} holds.
\end{Trm}

\begin{Rem}
\label{hat_lam_0_1} In the part 3 of Theorem \ref{main6} we have $\hat \lambda \in (0, \, 1)$.
\end{Rem}

\begin{Rem}
The number $\alpha_*$ is positive since $t_{l_*}$ is the strict minimum point of $h$; see \eqref{str_min_h}.
\end{Rem}

In the case  $q=2$, $\sigma \in [2, \,\infty)$ or $\sigma=2$, $q\in [2, \,\infty)$, estimates can also be obtained under certain conditions of general position (since estimates of widths of intersections of finite-dimensional balls in the mixed norm for $(q, \, \sigma)\in [2, \,\infty)^2$) are available; 
this  should be less  technically challenging than the one considered here, but this case is not considered due to the large volume of the paper.

The paper is organized as follows. In \S 2 we prove Proposition \ref{h_prop} and some auxilliary assertions which will be used in the proof of the lower estimates for the widths. In \S 3 we prove the upper estimates in Theorems \ref{main1}‐‐\ref{main5} and \ref{main6} (in Theorem \ref{main_non_emb} the estimate is trivial). In \S 4 we prove the lower estimates.

\section{The proof of Proposition \ref{h_prop}}

Let $N\in \N$, $1\le p\le \infty$. We denote by $l_p^N$ the space $\R^N$ with the norm $\|(x_1, \, \dots, \, x_N)\|_{l_p^N} = \Bigl( \sum \limits _{j=1}^N |x_j|^p\Bigr)^{1/p}$ for $p<\infty$, $\|(x_1, \, \dots, \, x_N)\| _{l_p^N} = \max _{1\le j\le N} |x_j|$ for $p=\infty$. Let $B_p^N$ be the unit ball in the space $l_p^N$.

The estimates of the widths $d_n(B_p^N, \, l_q^N)$ were obtained in \cite{gluskin1, bib_gluskin, garn_glus, pietsch1, stesin}; here we formulate these results for the cases that will be considered below.

\begin{trma}
\label{glus} {\rm (see \cite{bib_gluskin}).} Let $1\le p\le q<\infty$, $q>2$, $0\le n\le N/2$. Then
$$
d_n(B_p^N, \, l_q^N) \underset{q}{\asymp} \min \{1, \,
n^{-1/2}N^{1/q}\} ^{\omega_{p,q}}.
$$
\end{trma}

\begin{trma}
\label{p_s} {\rm (see \cite{pietsch1, stesin}).} Let $1\le q\le p\le
\infty$, $0\le n\le N$. Then
$$
d_n(B_p^N, \, l_q^N) = (N-n)^{1/q-1/p}.
$$
\end{trma}

Let us formulate the theorem about the estimates for the Kolmogorov widths of an intersection of a family of finite‐dimensional balls in the case $2<q<\infty$.

\begin{trma}
\label{fin_dim_inters} {\rm (see \cite{vas_ball_inters}, \cite[Proposition 1]{vas_int_sob}).}
Let $A$ be a finite non‐empty set, and let $1\le p_\alpha\le \infty$, $\nu_\alpha>0$, $\alpha \in A$, $2<q<\infty$,
\begin{align}
\label{m_set_def}
M_0 = \cap _{\alpha \in A} \nu _\alpha B_{p_\alpha}^N, 
\end{align}
$N\ge 2n$. We define the numbers $\lambda_{\alpha,\beta}$ and $\tilde \lambda_{\alpha,\beta}$ by the equations
\begin{gather}
\label{q_l_ab}
\frac{1}{q}= \frac{1‐\lambda_{\alpha,\beta}}{p_\alpha} + \frac{\lambda_{\alpha,\beta}}{p_\beta}, \quad p_\alpha > q, \; p_\beta < q,
\\
\label{2_tl_ab}
\frac{1}{2}= \frac{1‐\tilde\lambda_{\alpha,\beta}}{p_\alpha} + \frac{\tilde\lambda_{\alpha,\beta}}{p_\beta}, \quad p_\alpha > 2, \; p_\beta < 2.
\end{gather}
Then  
\begin{align}
\label{dnm0q2}
\begin{array}{c}
d_n(M_0, \, l_q^N) \underset{q}{\asymp} \min \left\{ \min _{\alpha \in A}d_n(\nu_\alpha B_{p_\alpha}^N, \, l_q^N), \, \min _{p_\alpha>q, \, p_\beta< q} \nu_\alpha ^{1-\lambda_{\alpha,\beta}}\nu_\beta^{\lambda_{\alpha,\beta}},\right. \\ \left. \min _{p_\alpha> 2, \, p_\beta< 2} \nu_\alpha ^{1-\tilde\lambda_{\alpha,\beta}}\nu_\beta^{\tilde\lambda_{\alpha,\beta}}d_n(B_2^N, \, l_q^N)\right\}.
\end{array}
\end{align}
If $\{\alpha:\; p_\alpha<q\}\ne \varnothing$, $\{\alpha:\; p_\alpha>2\}\ne \varnothing$, $1\le \frac{\nu_\alpha}{\nu_\beta} \le N^{1/p_{\alpha} ‐ 1/p_{\beta}}$ for all $\alpha$, $\beta \in A$ such that $p_\alpha \le p_\beta$, then in \eqref{dnm0q2} the value $\min _{\alpha \in A}d_n(\nu_\alpha B_{p_\alpha}^N, \, l_q^N)$ can be replaced by $\min _{\alpha \in A:\; 2\le p_\alpha \le q}d_n(\nu_\alpha B_{p_\alpha}^N, \, l_q^N)$.
\end{trma}

\begin{Rem}
If we replace minimum by infimum, this result will be also true for infinite $A$; the reduction of the case of arbitrary set $A$ to the case of the finite one is similar to \cite[\S 5]{vas_mix_sev}.
\end{Rem}

The next result refines Theorem \ref{fin_dim_inters} under the condition $p_\alpha \notin \{2, \, q\}$, $\alpha \in A$.
\begin{trma}
\label{expl} {\rm (see \cite[\S 3, Lemma 2]{vas_dif_der}).} Let $A$ be a finite non‐empty set, let $2< q<\infty$, $p_\alpha\notin \{2, \, q\}$ for each $\alpha\in A$, $N^{2/q}\le n \le N/2$, and let the set $M_0$ be given by formula \eqref{m_set_def}.
\begin{enumerate}
\item Let $p_{\alpha_*} < 2$, $\nu_{\alpha_*} \le \nu_\beta$ for all $\beta \in A$. Then $d_n(M_0, \, l_q^N) \underset{q}{\asymp} \nu_{\alpha_*}n^{‐1/2}N^{1/q}$.

\item Let $p_{\alpha_*} > q$, $\nu_{\alpha_*} N^{1/p_\beta ‐1/p_{\alpha_*}}\le \nu_\beta$ for all $\beta \in A$. Then $d_n(M_0, \, l_q^N) \underset{q}{\asymp} \nu_{\alpha_*} N^{1/q‐1/p_{\alpha_*}}$.

\item Let $2<p_{\alpha_*}<q$, 
\begin{align}
\label{nuan12n1q}
\nu_{\alpha_*}(n^{1/2}N^{‐1/q})^{\frac{1/p_\beta‐1/p_{\alpha_*}}{1/2‐1/q}} \le \nu_\beta
\end{align}
for all $\beta \in A$. Then 
\begin{align}
\label{dn_m0_j}
d_n(M_0, \, l_q^N) \underset{q}{\asymp} \nu_{\alpha_*} (n^{‐1/2}N^{1/q})^{\frac{1/p_{\alpha_*}‐1/q}{1/2‐1/q}}.
\end{align}

\item Let $p_{\alpha_*} > q$, $p_{\beta_*}< q$. Suppose that 
\begin{align}
\label{nu_ab}
\begin{array}{c}
\nu_{\alpha_*} ^{1‐\lambda_{\alpha_*,\beta_*}}\nu_{\beta_*}^{\lambda_{\alpha_*,\beta_*}} \le \nu_{\alpha_*} ^{1‐\lambda_{\alpha_*,\gamma}}\nu_\gamma^{\lambda_{\alpha_*,\gamma}}, \quad \gamma \in A, \; p_\gamma < q,
\\
\nu_{\alpha_*} ^{1‐\lambda_{\alpha_*,\beta_*}}\nu_{\beta_*}^{\lambda_{\alpha_*,\beta_*}} \le \nu_\gamma ^{1‐\lambda_{\gamma,\beta_*}}\nu_{\beta_*}^{\lambda_{\gamma,\beta_*}}, \quad \gamma \in A, \; p_\gamma > q,
\end{array}
\end{align}
\begin{align}
\label{nu_ab1_lem2}
\nu_{\alpha_*}  \le \nu_{\beta_*}(n^{1/2}N^{‐1/q})^{\frac{1/p_{\alpha_*}‐1/p_{\beta_*}}{1/2‐1/q}}, \quad \nu_{\alpha_*}\ge \nu_{\beta_*} N^{1/p_{\alpha_*}‐1/p_{\beta_*}}.
\end{align}
Then
$d_n(M_0, \, l_q^N) \underset{q}{\asymp} \nu_{\alpha_*} ^{1‐\lambda_{\alpha_*,\beta_*}}\nu_{\beta_*}^{\lambda_{\alpha_*,\beta_*}}$.

\item Let $p_{\alpha_*} > 2$, $p_{\beta_*}< 2$. Suppose that
\begin{align}
\label{nu_ab_lem2}
\begin{array}{c}
\nu_{\alpha_*} ^{1‐\tilde\lambda_{\alpha_*,\beta_*}}\nu_{\beta_*}^{\tilde\lambda_{\alpha_*,\beta_*}} \le \nu_{\alpha_*} ^{1‐\tilde\lambda_{\alpha_*,\gamma}}\nu_\gamma^{\tilde\lambda_{\alpha_*,\gamma}}, \quad \gamma \in A, \; p_\gamma < 2,
\\
\nu_{\alpha_*} ^{1‐\tilde\lambda_{\alpha_*,\beta_*}}\nu_{\beta_*}^{\tilde\lambda_{\alpha_*,\beta_*}} \le \nu_\gamma ^{1‐\tilde\lambda_{\gamma,\beta_*}}\nu_{\beta_*}^{\tilde\lambda_{\gamma,\beta_*}}, \quad \gamma \in A, \; p_\gamma > 2,
\end{array}
\end{align}
\begin{align}
\label{nu_ab2_lem2}
\nu_{\alpha_*}  \le \nu_{\beta_*}, \quad \nu_{\alpha_*}\ge \nu_{\beta_*} (n^{1/2}N^{‐1/q})^{\frac{1/p_{\alpha_*}‐1/p_{\beta_*}}{1/2‐1/q}}.
\end{align}
Then
$d_n(M_0, \, l_q^N) \underset{q}{\asymp} \nu_{\alpha_*} ^{1‐\tilde\lambda_{\alpha_*,\beta_*}}\nu_{\beta_*}^{\tilde \lambda_{\alpha_*,\beta_*}}n^{‐1/2}N^{1/q}$.
\end{enumerate}
\end{trma}

\begin{Lem}
\label{h_kriterii}
Suppose that \eqref{non_emb}, \eqref{ijk_ne_i_k}, \eqref{p_ne_2q}, \eqref{frac_aj} hold. Let the function $h$ be defined by \eqref{h0_def}‐‐\eqref{h_def_012}. Then the following assertions hold.
\begin{enumerate}
\item Let $j_*\in J$, $t\in [1, \, q/2]$. Then $h(t) = \varphi_{j_*}(t)$ if and only if 
\begin{align}
\label{alp_jj}
(\alpha_{j_*} ‐ \alpha_j)t \ge \frac 12 \cdot \frac{1/p_{j_*} ‐ 1/p_j}{1/2 ‐ 1/q}(t‐1), \quad j=1, \, \dots, \, s.
\end{align}
\item Let $t\in [1, \, q/2]$, $I\ne \varnothing$. Then $h(t) = h_0(t)$ if and only if
\begin{align}
\label{alp_jj0}
(\alpha_{i_0} ‐ \alpha_{j_0}) t\ge \frac 12 \cdot \frac{1/p_{i_0} ‐ 1/p_{j_0}}{1/2‐1/q}(t‐1).
\end{align}
\item Let $t\in [1, \, q/2]$, $K \ne \varnothing$. Then $h(t) = h_1(t)$ if and only if
\begin{align}
\label{alp_jj1}
(\alpha_{i_1} ‐ \alpha_{j_1}) t\le \frac 12 \cdot \frac{1/p_{i_1} ‐ 1/p_{j_1}}{1/2‐1/q}(t‐1).
\end{align}
\item Let $i\in I$. Then $h(t)>\alpha_i t$ for all $t\in [1, \, q/2]$.

\item Let $j\in K$. Then $h(t)>(\alpha_j‐1/p_j)t + 1/2$ for all $t\in [1, \, q/2]$.
\end{enumerate}
\end{Lem}
\begin{proof}
{\it Proof of necessity in parts 1--3.} In part 1, the inequality \eqref{alp_jj} follows from the relations $\varphi_{j_*}(t) \ge \varphi_j(t)$ for $j\in J$, $\varphi_{j_*}(t) \ge h_0(t) \ge ((1‐\lambda_{j,j_*})\alpha_j + \lambda_{j,j_*} \alpha_{j_*})t$ for $j\in I$, $\varphi_{j_*}(t) \ge h_1(t) \ge ((1‐\tilde\lambda_{j_*,j})\alpha_{j_*} + \tilde\lambda_{j_*,j} \alpha_j‐1/2)t + 1/2$ for $j\in K$ (see \eqref{i01j01}‐‐\eqref{h_def_012}).

In part 2, the inequality \eqref{alp_jj0} follows from the relations $h_0(t)\ge \varphi_{j_0}(t)$ in the case $j_0\in J$, $h_0(t) \ge h_1(t) \ge ((1‐\tilde\lambda_{i_0,j_0})\alpha_{i_0} + \tilde\lambda_{i_0,j_0} \alpha_{j_0}‐1/2)t + 1/2$ in the case $j_0\in K$.

In part 3, the inequality \eqref{alp_jj1} follows from the relations $h_1(t)\ge \varphi_{i_1}(t)$ in the case $i_1\in J$, $h_1(t) \ge h_0(t) \ge ((1‐\lambda_{i_1,j_1})\alpha_{i_1} + \lambda_{i_1,j_1} \alpha_{j_1})t$ in the case $i_1\in I$.

{\it Proof of part 4.} By \eqref{ijk_ne_i_k}, we have $J\sqcup K \ne \varnothing$. Let $j\in J\sqcup K$. Then $j>i$ (see \eqref{p1ps}, \eqref{ijk_def}), which implies $\alpha_j>\alpha_i$ (see \eqref{non_emb}). Hence
$$
h(t) \ge h_0(t) \stackrel{\eqref{i01j01}, \eqref{h0_def}}{\ge} ((1‐\lambda_{i,j})\alpha_i + \lambda_{i,j} \alpha_j)t > \alpha_it.
$$

{\it Proof of part 5.} By \eqref{ijk_ne_i_k}, we have $I\sqcup J \ne \varnothing$. Let $i\in I\sqcup J$. Then $j>i$ (see \eqref{p1ps}, \eqref{ijk_def}), which implies $\alpha_j‐1/p_j< \alpha_i ‐1/p_i$ (see \eqref{non_emb}). Therefore,
$$
h(t)\ge h_1(t) \stackrel{\eqref{i01j01}, \eqref{h1_def}}{\ge} ((1‐\tilde\lambda_{i,j})\alpha_i + \tilde\lambda_{i,j} \alpha_j‐1/2)t + 1/2 > (\alpha_j ‐1/p_j)t + 1/2.
$$

{\it Proof of sufficiency in parts 1--3.} Let $n\in \N$, $t'\in [1, \, q/2]$ be such that $m_n := t'\log n\in \N$, $2^{m_n}\ge 2n$. Consider the magnitude $d_n(\cap _{j=1}^s 2^{‐m_n(\alpha_j + 1/q ‐ 1/p_j)}B^{2^{m_n}}_{p_j}, \, l_q^{2^{m_n}})$. From \eqref{non_emb} it follows that 
$$
1\le \frac{2^{‐m_n(\alpha_j + 1/q ‐ 1/p_j)}}{2^{‐m_n(\alpha_i + 1/q ‐ 1/p_i)}} \le 2^{m_n(1/p_j‐1/p_i)}, \quad i\le j.
$$
From Theorems \ref{glus}, \ref{p_s}, \ref{fin_dim_inters} and \eqref{ijk_ne_i_k}, \eqref{p_ne_2q}, \eqref{lam_ij_def}‐‐\eqref{h_def_012} we get that
$$
d_n(\cap _{j=1}^s 2^{‐m_n(\alpha_j + 1/q ‐ 1/p_j)}B^{2^{m_n}}_{p_j}, \, l_q^{2^{m_n}}) \underset{q}{\asymp}
$$
$$
\asymp \min \Bigl\{ \min _{j\in J} 2^{‐m_n(\alpha_j + 1/q ‐ 1/p_j)} (n^{‐1/2}2^{m_n/q})^{\frac{1/p_j‐1/q}{1/2‐1/q}},
$$
$$
\min _{i\in I, \, j\in J\sqcup K} 2^{‐m_n((1‐\lambda_{i,j})\alpha _i + \lambda_{i,j}\alpha_j)}, \, \min _{i\in I\sqcup J, \, j\in K} 2^{‐m_n((1‐\tilde\lambda_{i,j})\alpha _i + \tilde\lambda_{i,j}\alpha_j‐1/2)}n^{‐1/2}\Bigr\}= n^{‐h(t')}.
$$
Thus,
\begin{align}
\label{dn_h} d_n(\cap _{j=1}^s 2^{‐m_n(\alpha_j + 1/q ‐ 1/p_j)}B^{2^{m_n}}_{p_j}, \, l_q^{2^{m_n}}) \underset{q}{\asymp} n^{‐h(t')}.
\end{align}

Let \eqref{alp_jj} hold for $j_*\in J$, but $h(t) > \varphi_{j_*}(t)$. Then there are a neighborhood $U$ of the point $t$ and the number $c>0$ such that 
\begin{align}
\label{hgjst}
h(t')> \varphi_{j_*}(t') + c, \quad t'\in U.
\end{align}

We show that the set $T$ of numbers $t'\in [1, \, q/2]$ such that 
\begin{align}
\label{ajt_pr}
(\alpha_{j_*} ‐ \alpha_j)t' \ge \frac 12 \cdot \frac{1/p_{j_*} ‐ 1/p_j}{1/2 ‐ 1/q}(t'‐1), \quad 1\le j\le s,
\end{align}
has the non‐empty interior. Assume on the contrary that the interior of $T$ is empty; then $T=\{t\}$.

If $t\in (1, \, q/2)$, then there are indices $i, \, k\notin\{j_*\}$, $i\ne k$, such that
$$
(\alpha_{j_*} ‐ \alpha_i)t = \frac 12 \cdot \frac{1/p_{j_*} ‐ 1/p_i}{1/2 ‐ 1/q}(t‐1), \quad (\alpha_{j_*} ‐ \alpha_k)t = \frac 12 \cdot \frac{1/p_{j_*} ‐ 1/p_k}{1/2 ‐ 1/q}(t‐1).
$$
This contradicts condition \eqref{frac_aj}.

Let $t=1$. Then \eqref{alp_jj} has the form $\alpha_{j_*} ‐\alpha _j\ge 0$ for any $j\in \{1, \, \dots, \, s\}$; hence, by \eqref{non_emb}, we have $j_*=s$ and $\alpha_{j_*} ‐\alpha _j> 0$ for all $j\in \{1, \, \dots, \, s\} \backslash \{j_*\}$. Therefore, \eqref{ajt_pr} holds in a right semi‐neighborhood of $t=1$.

Let $t=q/2$. Then \eqref{alp_jj} has the form $\alpha_{j_*} ‐\alpha _j\ge 1/p_{j_*}‐1/p_j$ for all $j\in \{1, \, \dots, \, s\}$; hence, by \eqref{non_emb}, we have $j_*=1$ and $\alpha_{j_*} ‐\alpha _j> 1/p_{j_*}‐1/p_j$ for all $j\in \{1, \, \dots, \, s\} \backslash \{j_*\}$. Therefore, \eqref{ajt_pr} holds in a left semi‐neighborhood of $t=q/2$.

If $n\in \N$ is sufficiently large, then there is $t'\in U$ such that $m_n=t'\log n\in \N$, $2n \le 2^{m_n} \le n^{q/2}$ and \eqref{ajt_pr} holds. This together with \eqref{dn_h}, \eqref{hgjst} implies that
\begin{align}
\label{dn_cap_c}
d_n(\cap _{j=1}^s 2^{‐m_n(\alpha_j + 1/q ‐ 1/p_j)}B^{2^{m_n}}_{p_j}, \, l_q^{2^{m_n}}) \underset{q}{\lesssim} n^{‐\varphi_{j_*}(t')‐c}.
\end{align}

Now we apply Theorem \ref{expl}. Inequality \eqref{ajt_pr} is equivalent to condition \eqref{nuan12n1q} for $\alpha_*=j_*$, $\beta=j$, $N = 2^{m_n}$; hence 
$$
d_n(\cap _{j=1}^s 2^{‐m_n(\alpha_j + 1/q ‐ 1/p_j)}B^{2^{m_n}}_{p_j}, \, l_q^{2^{m_n}}) \stackrel{\eqref{dn_m0_j}}{\underset{q}{\asymp}}$$$$ \asymp 2^{‐m_n(\alpha_{j_*}+1/q‐1/p_{j_*})} (n^{‐1/2} 2^{m_n/q})^{\frac{1/p_{j_*}‐1/q}{1/2‐1/q}} = n^{‐\varphi_{j_*}(t')}.
$$
We get $n^{‐\varphi_{j_*}(t')} \stackrel{\eqref{dn_cap_c}}{\underset{q}{\lesssim}} n^{‐\varphi_{j_*}(t')‐c}$ and arrive to a contradiction.

Suppose that $I\ne \varnothing$ and \eqref{alp_jj0} holds, but $h(t) > h_0(t)$. By \eqref{p1ps}, \eqref{ijk_def}, we have $i_0<j_0$; hence the inequality
\begin{align}
\label{ai0aj0tpr}
(\alpha_{i_0} ‐ \alpha_{j_0}) t'\ge \frac 12 \cdot \frac{1/p_{i_0} ‐ 1/p_{j_0}}{1/2‐1/q}(t'‐1)
\end{align}
holds in a left semi‐neighborhood of $q/2$ (see \eqref{non_emb}).

The numbers $t'$ and $m_n$ are defined as in the proof of part 1, but \eqref{ajt_pr} is replaced by \eqref{ai0aj0tpr}, and the neighborhood $U$ is such that $h(t')> h_0(t')+c$ for $t'\in U$. Instead of \eqref{dn_cap_c} we get
\begin{align}
\label{dn_cap_nh0c}
d_n(\cap _{j=1}^s 2^{‐m_n(\alpha_j + 1/q ‐ 1/p_j)}B^{2^{m_n}}_{p_j}, \, l_q^{2^{m_n}}) \underset{q}{\lesssim} n^{‐h_0(t')‐c}.
\end{align}

We check that the conditions of part 4 of Theorem \ref{expl} hold.

The first inequality \eqref{nu_ab1_lem2} with $\alpha_*=i_0$, $\beta_*=j_0$, $N=2^{m_n}$ follows from \eqref{ai0aj0tpr}, and the second one, from \eqref{non_emb}. Inequalities \eqref{nu_ab} follow from \eqref{i01j01}.

Hence
$$
d_n(\cap _{j=1}^s 2^{‐m_n(\alpha_j + 1/q ‐ 1/p_j)}B^{2^{m_n}}_{p_j}, \, l_q^{2^{m_n}}) \underset{q}{\asymp}$$$$ \asymp 2^{‐m_n((1‐\lambda_{i_0,j_0})\alpha_{i_0}+ \lambda_{i_0,j_0} \alpha_{j_0})}= n^{‐h_0(t')},
$$
which contradicts \eqref{dn_cap_nh0c}.

Part 3 of the lemma is similar to the previous one; the inequality
$$
(\alpha_{i_1} ‐ \alpha_{j_1}) t'\le \frac 12 \cdot \frac{1/p_{i_1} ‐ 1/p_{j_1}}{1/2‐1/q}(t'‐1)
$$
holds in the right semi‐neighborhood of $1$.
\end{proof}

\begin{Lem}
\label{ij01} Let conditions \eqref{ijk_ne_i_k}, \eqref{p_ne_2q}, \eqref{frac_aj} hold. Then in \eqref{i01j01} the pairs $(i_0, \, j_0)$ and $(i_1, \, j_1)$ are unique (if $I\ne \varnothing$ and $K\ne \varnothing$, respectively).
\end{Lem}
\begin{proof}
We prove that $(i_0, \, j_0)$ is unique (for $(i_1, \, j_1)$, the proof is similar). Let $i\in I$, $j\in J\sqcup K$, $(i, \, j)\ne (i_0, \, j_0)$, $(1‐\lambda_{i_0,j_0})\alpha _{i_0} + \lambda_{i_0,j_0} \alpha _{j_0} = (1‐\lambda_{i,j})\alpha _i + \lambda_{i,j} \alpha _j$, $(i_0, \, j_0) \ne (i, \, j)$. Then
$$
(1‐\lambda_{i_0,j_0})\alpha _{i_0} + \lambda_{i_0,j_0} \alpha _{j_0}\ge (1‐\lambda_{i_0,j})\alpha _{i_0} + \lambda_{i_0,j} \alpha _j,
$$
$$
(1‐\lambda_{i_0,j_0})\alpha _{i_0} + \lambda_{i_0,j_0} \alpha _{j_0} \ge
(1‐\lambda_{i,j_0})\alpha _i + \lambda_{i,j_0} \alpha _{j_0}, 
$$
$$
(1‐\lambda_{i,j})\alpha _i + \lambda_{i,j} \alpha _j \ge (1‐\lambda_{i_0,j})\alpha _{i_0} + \lambda_{i_0,j} \alpha _j,
$$
$$
(1‐\lambda_{i,j})\alpha _i + \lambda_{i,j} \alpha _j \ge (1‐\lambda_{i,j_0})\alpha _i + \lambda_{i,j_0} \alpha _{j_0}.
$$
This together with \eqref{lam_ij_def} yields
$$
\frac{\alpha_{j_0} ‐ \alpha_{i_0}}{1/p_{j_0}‐1/p_{i_0}} \ge \frac{\alpha_j ‐ \alpha_{i_0}}{1/p_j‐1/p_{i_0}}, \quad \frac{\alpha_{j_0} ‐ \alpha_{i_0}}{1/p_{j_0}‐1/p_{i_0}} \le \frac{\alpha_{j_0} ‐ \alpha_i}{1/p_{j_0}‐1/p_i},
$$
$$
\frac{\alpha_j ‐ \alpha_i}{1/p_j‐1/p_i} \le \frac{\alpha_j ‐ \alpha_{i_0}}{1/p_j‐1/p_{i_0}}, \quad \frac{\alpha_j ‐ \alpha_i}{1/p_j‐1/p_i} \ge \frac{\alpha_{j_0} ‐ \alpha_i}{1/p_{j_0}‐1/p_i}.
$$
Hence $\frac{\alpha_j ‐ \alpha_i}{1/p_j‐1/p_i} = \frac{\alpha_{j_0} ‐ \alpha_{i_0}}{1/p_{j_0}‐1/p_{i_0}}$. This contradicts \eqref{frac_aj}.
\end{proof}

\renewcommand{\proofname}{\bf Proof of Proposition \ref{h_prop}}

\begin{proof}
First we prove assertion 2. Let $K= \varnothing$. Then $J\ne \varnothing$ (see \eqref{ijk_ne_i_k}). This together with \eqref{p1ps}, \eqref{ijk_def} yields that $s\in J$. The relation \eqref{alp_jj} for $j_*=s$ holds in a neighborhood of $t=1$ (see \eqref{non_emb}). By Lemma \ref{h_kriterii}, we get that $h(t) =\varphi_s(t)$ in a neighborhood of $1$.

Assertion 4 can be proved similarly: we get that $1\in J$ and in a neighborhood of $t=q/2$ relation \eqref{alp_jj} holds for $j_*=1$.

Now we prove assertions 3 and 5. Let $K\ne \varnothing$. We check that \eqref{alp_jj1} holds in a neighborhood of $t=1$. Indeed, since $p_{i_1}> 2 > p_{j_1}$ (see \eqref{ijk_def}, \eqref{i01j01}), we have $i_1\stackrel{\eqref{p1ps}}{<}j_1$ and $\alpha_{i_1} \stackrel{\eqref{non_emb}}{<} \alpha_{j_1}$. Hence \eqref{alp_jj1} holds for $t=1$, and, in addition, the inequality is strict.

By Lemma \ref{h_kriterii}, $h(t)=h_1(t)$ in a neighborhood of $t=1$.

It can be similarly proved that if $I\ne \varnothing$, then \eqref{alp_jj0} holds in a neighborhood of $t=q/2$; hence $h(t)=h_0(t)$ in a neighborhood of $t=q/2$.

We set $t_* = \max\{t\in [1, \, q/2]:\; \eqref{alp_jj1}\text{ holds at }t\}$ (in the case $K\ne \varnothing$), $t_{**} = \min\{t\in [1, \, q/2]:\;\eqref{alp_jj0}\text{ holds at }t\}$ (in the case $I\ne \varnothing$). From \eqref{non_emb} it follows that $t_*\in (1, \, q/2)$, $t_{**} \in (1, \, q/2)$.

Let $K\ne \varnothing$. By the definition of $t_*$ and the inequality $t_*<\frac q2$, we have 
\begin{align}
\label{ai1j112}
(\alpha_{i_1} ‐ \alpha _{j_1})t_* = \frac 12 \cdot \frac{1/p_{i_1} ‐ 1/p_{j_1}}{1/2‐1/q} (t_*‐1), 
\end{align}
which implies
\begin{align}
\label{h1tphii1}
h_1(t_*) = \varphi_{i_1}(t_*).
\end{align}

Let $i_1\in J$. We prove that the equality $h(t)= \varphi _{i_1}(t)$ holds in a right semi‐neighborhood of $t_*$.

We have two cases:
\begin{enumerate}
\item In a right semi‐neighborhood of $t_*$ the equality $h(t)= \varphi _j(t)$ holds for some $j\in J$. Then $\varphi_{i_1}(t_*)\stackrel{\eqref{h1tphii1}}{=}h_1(t_*) = \varphi_j(t_*)$. Hence
$$
(\alpha_{i_1} ‐ \alpha _j)t_* = \frac 12 \cdot \frac{1/p_{i_1} ‐ 1/p_j}{1/2‐1/q} (t_*‐1);
$$
this together with \eqref{ai1j112} and the condition $j_1\in K$ contradicts \eqref{frac_aj} for $j\ne i_1$.

\item In a right semi‐neighborhood of $t_*$ the equality  $h(t)= h_0(t)$ holds. Then $I\ne \varnothing$, $t_*=t_{**}$. This together with the definition of $t_{**}$ yields 
\begin{align}
\label{t_stst_ai0}
(\alpha_{i_0} ‐ \alpha _{j_0})t_* = \frac 12 \cdot \frac{1/p_{i_0} ‐ 1/p_{j_0}}{1/2‐1/q} (t_*‐1).
\end{align}
By \eqref{frac_aj}, the equality \eqref{ai1j112} together with \eqref{t_stst_ai0} is possible only if $(i_0, \, j_0) = (i_1, \, j_1)$. But $i_0\in I$, $i_1\in J$. We arrive to a contradiction.
\end{enumerate}

It can be similarly proved that if $I\ne \varnothing$, $j_0\in J$, then $h(t) = \varphi_{j_0}(t)$ in a left neighborhood of $t_{**}$.

Now let $K\ne \varnothing$, $i_1\in I$. We claim that $L=2$, $(i_0, \, j_0) = (i_1, \, j_1)$ and $h|_{[t_1, \, t_2]} = h_0|_{[t_1, \, t_2]}$. Indeed, from \eqref{ai1j112} it follows that $h(t_*) = h_{i_1,j_1}(t_*)$, where 
$$
h_{i_1,j_1}(t) = ((1‐\lambda_{i_1,j_1})\alpha_{i_1} + \lambda _{i_1,j_1}\alpha _{j_1})t.
$$
By Lemma \ref{ij01}, if $(i_0, \, j_0) \ne (i_1, \, j_1)$, then $h_0(t)> h_{i_1,j_1}(t)$, $t\in [1, \, q/2]$. Hence $h_0(t_*)>h_{i_1,j_1}(t_*) = h(t_*)$, which contradicts \eqref{h_def_012}.

Thus, $(i_0, \, j_0) = (i_1, \, j_1)$. From the definition of $t_*$ it follows that for $t\ge t_*$ we have $(\alpha_{i_1} ‐ \alpha _{j_1})t \ge \frac 12 \cdot \frac{1/p_{i_1} ‐ 1/p_{j_1}}{1/2‐1/q} (t‐1)$. By assertion 2 of Lemma \ref{h_kriterii}, we get $h(t) = h_0(t)$ for $t\in [t_*, \, q/2]$.

It can be similarly proved that if $I\ne \varnothing$, $j_0\in K$, then $L=2$, $(i_0, \, j_0) = (i_1, \, j_1)$ and $h|_{[t_0, \, t_1]} = h_1|_{[t_0, \, t_1]}$.

This completes the proof of assertions 3 and 5 of the present proposition.

Assertion 1 follows from assertions 3 and 5.

Assertion 6 follows from the inequalities $\varphi_i(1)\stackrel{\eqref{non_emb}}{>} \varphi_j(1)$, $\varphi_i(q/2)\stackrel{\eqref{non_emb}}{<} \varphi_j(q/2)$ for $i>j$, $i$, $j\in J$. 
\end{proof}

\renewcommand{\proofname}{\bf Proof}

Now we prove one more property of the function $h$; it will be used in the proof of the lower estimate for the widths.

Denote for $i\in I$, $j\in J\sqcup K$ and $i\in I\sqcup J$, $j\in K$, respectively,
\begin{align}
\label{hij_tilhij_def}
h_{i,j}(t) = ((1‐\lambda_{i,j})\alpha_i + \lambda _{i,j} \alpha_j)t, \quad \tilde h_{i,j}(t) = ((1‐\tilde\lambda_{i,j})\alpha_i + \tilde\lambda _{i,j} \alpha_j‐1/2)t + 1/2.
\end{align}

\begin{Sta}
\label{strict_ineq} Let the conditions \eqref{non_emb}, \eqref{ijk_ne_i_k}, \eqref{p_ne_2q}, \eqref{frac_aj} hold, and let $t_0<t_1<\dots<t_L$ be the partition of $[1, \, q/2]$ from Proposition {\rm \ref{h_prop}}, $0\le l\le L$. Then the following assertions hold.
\begin{enumerate}
\item If $i\in I$, $j\in J\sqcup K$, $(i, \, j)\ne (i_0, \, j_0)$, then $h_{i,j}(t_l)< h(t_l)$.

\item If $i\in I\sqcup J$, $j\in K$, $(i, \, j) \ne (i_1, \, j_1)$, then $\tilde h_{i,j}(t_l) < h(t_l)$.

\item If $i\in I$, then $\alpha_i t_l < h(t_l)$.

\item If $j\in K$, then $(\alpha_j‐1/p_j)t_l +1/2 < h(t_l)$.

\item If $$i\in \begin{cases}J \backslash \{j(l), \, j(l+1)\}, & 1\le l \le L‐1, \\ J \backslash \{s\}, & l=0,\; K=\varnothing, \\ J \backslash \{1\}, & l=L,\; I=\varnothing, \\ J, & l=0, \; K\ne \varnothing \text{ or } l = L, \; I\ne \varnothing,\end{cases}$$ then $\varphi_i(t_l) < h(t_l)$.
\end{enumerate}
\end{Sta}

\begin{proof}
The first two assertions follow from the inequalities $$h(t_l) \ge h_{i_0,j_0}(t_l) \stackrel{\eqref{i01j01}, \eqref{h0_def}}{>} h_{i,j}(t_l), \quad h(t_l) \ge \tilde h_{i_1,j_1}(t_l) \stackrel{\eqref{i01j01}, \eqref{h1_def}}{>} \tilde h_{i,j}(t_l);$$ the strict inequalities hold by Lemma \ref{ij01}.

Assertion 3 follows from the inequalities $$h(t_l)\ge h_0(t_l) \stackrel{\eqref{i01j01}, \eqref{h0_def}}{\ge} h_{i,j_0}(t_l) \stackrel{\eqref{non_emb},\eqref{p1ps}, \eqref{ijk_def}}{>} \alpha_i t_l,$$ and assertion 4, from the inequalities $$h(t_l) \ge h_1(t_l)\stackrel{\eqref{i01j01}, \eqref{h1_def}}{\ge} \tilde h_{i_1,j}(t_l) \stackrel{\eqref{non_emb},\eqref{p1ps}, \eqref{ijk_def}}{>} (\alpha_j ‐ 1/p_j)t_l + 1/2.$$

Now we prove assertion 5. The following cases are possible.
\begin{enumerate}
\item Let $1\le l \le L‐1$, $h|_{[t_{l‐1}, \, t_l]}=\varphi_{j(l)}|_{[t_{l‐1}, \, t_l]}$, $h|_{[t_l, \, t_{l+1}]}=\varphi_{j(l+1)}|_{[t_l, \, t_{l+1}]}$, $j(l)$, $j(l+1)\in J$. Then $h(t_l) = \varphi_{j(l)}(t_l) = \varphi_{j(l+1)}(t_l)$. If $i\in J \backslash \{j(l), \, j(l+1)\}$ and $\varphi_i(t_l)\ge h(t_l)$, then, by \eqref{h2_def}, \eqref{h_def_012}, we get $\varphi_i(t_l) = \varphi_{j(l)}(t_l) = \varphi_{j(l+1)}(t_l)$, which contradicts \eqref{frac_aj}.

\item Let $l=1$, $h|_{[t_0, \, t_1]} = h_1|_{[t_0, \, t_1]}$, $h|_{[t_1, \, t_2]} = \varphi_{i_1}|_{[t_1, \, t_2]}$, $i_1=j(2)\in J$ (see assertion 3 of Proposition \ref{h_prop} and Remark \ref{i1j2_j0jl1}). If $i\in J\backslash \{j(2)\}$ (the index $j(1)$ is undefined), we have $i\ne i_1$. It follows from the equality $h_1(t_l)= \varphi_{i_1}(t_l)$ that $\varphi_{i_1}(t_l) = \varphi_{j_1}(t_l)$. If $h(t_l)\le \varphi_i(t_l)$, then $\varphi_i(t_l) = \varphi_{i_1}(t_l) = \varphi_{j_1}(t_l)$. Since the indices $i$, $i_1$, $j_1$ are different, we arrive to a contradiction with \eqref{frac_aj}.

\item The case $l=L‐1$, $h|_{[t_{L‐1}, \, t_L]} = h_0|_{[t_{L‐1}, \, t_L]}$, $h|_{[t_{L‐2}, \, t_{L‐1}]} = \varphi_{j_0}|_{[t_{L‐2}, \, t_{L‐1}]}$, $j_0\in J$, $i\in J \backslash \{j(L‐1)\}$ can be considered similarly to the pevious one.

\item Let $L=2$, $l=1$, $h|_{[t_0, \, t_1]} = h_1|_{[t_0, \, t_1]}$, $h|_{[t_1, \, t_2]} = h_0|_{[t_1, \, t_2]}$. Then $(i_0, \, j_0) = (i_1, \, j_1)$ (see assertions 3, 5 of Proposition \ref{h_prop}). From the equality $h_{i_1,j_1}(t_1) = \tilde h_{i_1,j_1}(t_1)$ it follows that $\varphi_{i_1}(t_1) = \varphi _{j_1}(t_1) = h_{i_1,j_1}(t_1)$. If $h(t_1) \le \varphi_i(t_1)$, $i\in J$, then $\varphi_{i_1}(t_1) = \varphi _{j_1}(t_1) = \varphi_i(t_1)$. Since $i_1 = i_0\in I$, $j_1\in K$, $i\in J$, we get that the indices $i$, $i_1$, $j_1$ are different; hence we arrive to a contradiction with \eqref{frac_aj}.

\item Let $l=0$. If $K= \varnothing$, we have $h|_{[t_0, \, t_1]} = \varphi_s|_{[t_0, \, t_1]}$ (see assertion 2 of Proposition \ref{h_prop}). Hence, for any $i\in J\backslash \{s\}$, we get $\varphi_i(1) \stackrel{\eqref{non_emb}}{<} \varphi_s(1) = h(1)$.

Let $K\ne \varnothing$. Then $h|_{[t_0, \, t_1]} = h_1|_{[t_0, \, t_1]}$ (see assertion 3 of Proposition \ref{h_prop}). If $h(1)\le \varphi _i(1)$ for some $i\in J$, we get $h_1(1) = \varphi_i(1)$. We have $h_1(0)\stackrel{\eqref{h1_def}}{=}\frac 12 > \frac 12 \cdot \frac{1/p_i‐1/q}{1/2‐1/q} \stackrel{\eqref{phi_j_def}}{=} \varphi_i(0)$. Hence, for $t>1$, we get $h_1(t) < \varphi_i(t)$. Now we take $t\in (t_0, \, t_1]$ and arrive to a contradiction with \eqref{h2_def}, \eqref{h_def_012}.

\item Let $l=L$. If $I = \varnothing$, we have $h|_{[t_{L‐1}, \, t_L]} = \varphi_1|_{[t_{L‐1}, \, t_L]}$ (see assertion 4 of Proposition \ref{h_prop}). Hence, for any $i\in J\backslash \{1\}$, we get $\varphi_i(q/2) \stackrel{\eqref{non_emb}}{<} \varphi_1(q/2) = h(q/2)$.

Let $I\ne \varnothing$. Then $h|_{[t_{L‐1}, \, t_L]} = h_0|_{[t_{L‐1}, \, t_L]}$ (see assertion 5 of Proposition \ref{h_prop}). If $h(q/2)\le \varphi_i(q/2)$ for some $i \in J$, then $h_0(q/2) = \varphi_i(q/2)$. Since $h_0(0) \stackrel{\eqref{h0_def}}{=} 0 < \frac 12 \cdot \frac{1/p_i‐1/q}{1/2‐1/q}  \stackrel{\eqref{phi_j_def}}{=}\varphi_i(0)$, we get $h_0(t)< \varphi_i(t)$ for $t<q/2$. Now we take $t\in [t_{L‐1}, \, t_L)$ and arrive to a contradiction with \eqref{h2_def}, \eqref{h_def_012}.
\end{enumerate}
This completes the proof.
\end{proof}

\section{Proof of the upper estimates}

Let $t, \, k\in \N$, $1\le p<\infty$, $1\le \theta<\infty$. We denote by $l_{p,\theta}^{t,k}$ the space $\R^{tk}$ with the norm
$$
\|(x_{i,j})_{1\le i\le t, \, 1\le j\le k}\|_{l_{p,\theta}^{t,k}} = \Bigl(\sum \limits _{j=1}^k\Bigl(\sum \limits _{i=1}^t |x_{i,j}|^p\Bigr)^{\theta/p}\Bigr)^{1/\theta}
$$
for $p<\infty$, $\theta<\infty$; for $p=\infty$ or $\theta=\infty$ the definition is naturally modified.

In what follows we denote
\begin{align}
\label{km_md1} k_m = \max\{m^{d‐1}, \, 1\}.
\end{align}

From \eqref{card_j_m}, \eqref{x_b_norm}, \eqref{discr_eq} and the order equality
\begin{align}
\label{mmdm1} \# \{\overline{m}\in \Z_+^d:\; m_1+\dots + m_d=m\} \underset{d}{\asymp} \max\{m^{d‐1}, \, 1\}
\end{align}
we get
\begin{Cor}
\label{upper_est} Let $n\in \N$, $C\in \N$, $\nu_m\in \Z_+$ $(m\in \Z_+)$, $\sum \limits _{m\in \Z_+} \nu_m \le Cn$. Then there is a constant $C_1 = C_1(d)\in \N$ such that
\begin{align}
\label{up_estim} d_{C_1Cn}(\cap _{j=1}^s SB^{\overline{r}_j}_{p_j,\theta_j}(\mathbb{T}^d), \, B^{\overline{l}}_{q, \sigma}(\mathbb{T}^d)) \underset{\mathfrak{Z}}{\lesssim} \sum \limits _{m\in \Z_+} d_{\nu_m} (\cap _{j=1}^s 2^{‐m(\alpha_j + 1/q ‐1/p_j)} B_{p_j,\theta_j} ^{2^m,k_m}, \, l_{q,\sigma} ^{2^m,k_m}).
\end{align}
\end{Cor}

We define the value $\Phi = \Phi(p, \, \theta, \, q, \, \sigma, \, t, \, k, \, n)$ as follows:
\begin{enumerate}
\item for $p\ge q$, $\theta \ge \sigma$,
\begin{align}
\label{phi1} \Phi = t^{1/q-1/p}k^{1/\sigma-1/\theta};
\end{align}
\item for $p\ge q$, $\theta\le \sigma$,
\begin{align}
\label{phi2} \Phi = \min \left\{t^{1/q-1/p}, \, t^{1/q-1/p}(n^{-\frac 12}t^{\frac 12} k^{\frac{1}{\sigma}}) ^{\omega_{\theta,\sigma}}\right\};
\end{align}

\item for $\theta\ge \sigma$, $p\le q$,
\begin{align}
\label{phi3} \Phi = \min \left\{k^{1/\sigma-1/\theta}, \, k^{1/\sigma-1/\theta}(n^{-\frac 12}t^{\frac 1q} k^{\frac{1}{2}}) ^{\omega_{p,q}}\right\};
\end{align}

\item for $2\le p\le q$, $1\le \theta \le \sigma$, $\omega_{p,q}\le \omega_{\theta,\sigma}$,
\begin{align}
\label{phi4}
\Phi = \min \left\{1, \, (n^{-\frac 12}t^{\frac 1q} k^{\frac{1}{\sigma}})^{\omega_{p,q}}, \, t^{1/q-1/p}(n^{-\frac 12}t^{\frac 12} k^{\frac{1}{\sigma}})^{\omega_{\theta,\sigma}}\right\};
\end{align}

\item for $2\le \theta\le \sigma$, $1\le p \le q$, $\omega_{\theta,\sigma}\le \omega_{p,q}$,
\begin{align}
\label{phi5}
\Phi = \min \left\{1, \, (n^{-\frac 12}t^{\frac 1q} k^{\frac{1}{\sigma}})^{\omega_{\theta,\sigma}}, \, k^{1/\sigma-1/\theta}(n^{-\frac 12}t^{\frac 1q} k^{\frac{1}{2}})^{\omega_{p,q}}\right\};
\end{align}
\item for $1\le p\le 2$, $1\le \theta\le 2$,
\begin{align}
\label{phi6}
\Phi = \min \{1, \, n^{-\frac 12}t^{\frac 1q} k^{\frac{1}{\sigma}}\}.
\end{align}
\end{enumerate}

\begin{trma}
\label{1mixed}
{\rm (see \cite{vas_mix_sev, vas_besov}).} Let $2\le q<\infty$, $2\le \sigma < \infty$, $1\le p\le \infty$, $1\le \theta \le \infty$, $n\le \frac{tk}{2}$, the value $\Phi$ is defined by formulas \eqref{phi1}--\eqref{phi6}. Then 
$$d_n(B^{t,k}_{p,\theta}, \, l^{t,k}_{q,\sigma}) \underset{q,\sigma}{\asymp} \Phi.$$
\end{trma}

\begin{Rem}
\label{dn_d0_phi} If $n \le t^{2/q}k^{2/\sigma}$, we have $\Phi(p, \, \theta, \, q, \, \sigma, \, t, \, k, \, n) = \Phi(p, \, \theta, \, q, \, \sigma, \, t, \, k, \, 0)$. This together with Theorem \ref{inters_ball_dn} (see \S 4) yields that, for $\nu_1>0$, $\nu_2>0$, $1\le p_1, \, p_2, \, \theta_1, \, \theta_2\le \infty$, $2\le q, \, \sigma <\infty$, $n \le t^{2/q}k^{2/\sigma}$,
$$
d_n(\nu_1 B^{t,k}_{p_1,\theta_1} \cap \nu_2 B^{t,k}_{p_2,\theta_2}, \, l^{t,k}_{q,\sigma}) \underset{q,\sigma}{\asymp} d_0(\nu_1 B^{t,k}_{p_1,\theta_1} \cap \nu_2 B^{t,k}_{p_2,\theta_2}, \, l^{t,k}_{q,\sigma}).
$$
\end{Rem}

\begin{Rem}
\label{eeee_rem}
In addition to the domain  $2\le q, \, \sigma<\infty$, $1\le p, \, \theta\le \infty$, 
the widths  $d_n(B^{t,k}_{p,\theta}, \, l^{t,k}_{q,\sigma})$ were estimated also for some domains of 
the parameters  (see  \cite{mal_rjut, mal_rjut1, mal_rjut2, dir_ull}). 
However, there are domains of the parameters where order estimates are not know so far.
\end{Rem}

\begin{trma}
\label{emb_s} {\rm (see \cite{vas_mix_sev}).} Let $\nu_i>0$, $1\le p_i\le \infty$, $1\le \theta_i\le \infty$, $\tau_i\ge 0$, $1\le i\le s$, $\sum \limits _{i=1}^s \tau_i=1$. We define the numbers $p$, $\theta\in [1, \, \infty]$ by the equations
$\frac 1p = \sum \limits _{i=1}^s \frac{\tau_i}{p_i}$, $\frac{1}{\theta} = \sum \limits _{i=1}^s \frac{\tau_i}{\theta_i}$.
Then
$$
\cap _{i=1}^s \nu_iB_{p_i,\theta_i}^{t,k} \subset \nu_1^{\tau_1}\dots \nu_s^{\tau_s} B^{t,k} _{p,\theta}.
$$
\end{trma}

We will apply Corollary \ref{upper_est} in proofs of the upper estimates. The numbers $\nu_m\in \Z_+$ are defined by the formula
\begin{align}
\label{nu_m_def}
\nu_m = \begin{cases} \min \{2^m k_m, \, \lfloor n\cdot 2^{‐\varepsilon |m‐m_0(n)|}\rfloor\}, & 2^{2m/q} k_m^{2/\sigma} \le n, \\ 0, & 2^{2m/q} k_m^{2/\sigma} > n;\end{cases}
\end{align}
the numbers $m_0(n)\in \R_+$ and $\varepsilon>0$ will be defined later (they depend on the set of parameters $\mathfrak{Z}$); we also suppose that 
\begin{align}
\label{22m0qkm02s}
2^{2m_0(n)/q} k_{m_0(n)}^{2/\sigma} \le n \le 2^{m_0(n)} k_{m_0(n)}.
\end{align}

We first obtain the estimates for the sums \eqref{up_estim} for $\varepsilon = 0$, i.e., for the magnitude
$$
\Sigma := \sum \limits _{m: \, 2^{2m/q} k_m^{2/\sigma} \le n\le 2^m k_m} d_n(\cap _{j=1}^s 2^{‐(\alpha_j+1/q‐1/p_j)m} B_{p_j,\theta_j}^{2^m, k_m}, \, l_{q,\sigma}^{2^m, k_m}) +$$$$+ \sum \limits _{2^{2m/q} k_m^{2/\sigma} > n} d_0(\cap _{j=1}^s 2^{‐(\alpha_j+1/q‐1/p_j)m} B_{p_j,\theta_j}^{2^m, k_m}, \, l_{q,\sigma}^{2^m, k_m}),
$$
and choose $m_0(n)$. From the proofs we will see that $\Sigma$ is estimated from above by a sum
$$
\sum \limits _{0\le m\le m_0(n)} 2^{a_1m}m^{b_1(d‐1)}n^{c_1} + \sum \limits _{m\ge m_0(n)} 2^{‐a_2m}m^{b_2(d‐1)}n^{c_2},
$$
where $a_1$, $a_2>0$, $b_1$, $b_2$, $c_1$, $c_2\in \R$, $$2^{a_1m_0(n)}[m_0(n)]^{b_1(d‐1)}n^{c_1} \underset{\mathfrak{Z}}{\asymp} 2^{‐a_2m_0(n)}[m_0(n)]^{b_2(d‐1)}n^{c_2} \underset{\mathfrak{Z}}{\asymp} n^{‐\alpha_*} (\log n)^{\beta_*(d‐1)}.$$ After that we argue as in \cite{vas_int_sob}, \cite{vas_dif_der}: we take sufficiently small $\varepsilon=\varepsilon(\mathfrak{Z}) >0$, apply Theorem \ref{1mixed} and obtain that, for $\nu_m$ defined by \eqref{nu_m_def}, the right‐hand side of \eqref{up_estim} can be estimated from above by
$$
\sum \limits _{0\le m\le m_0(n)} 2^{a_1m}m^{b_1(d‐1)}n^{c_1}\cdot 2^{\varepsilon\varkappa(m_0‐m)} +$$$$+ \sum \limits _{m\ge m_0(n)} 2^{‐a_2m}m^{b_2(d‐1)}n^{c_2}\cdot 2^{\varepsilon\varkappa(m‐m_0)}\underset{\mathfrak{Z}}{\asymp} n^{‐\alpha_*} (\log n)^{\beta_*(d‐1)}
$$
(here $\varkappa = \varkappa(\mathfrak{Z})>0$).

Thus, it remains to estimate from above the value $\Sigma$.

We will apply the following assertion (its proof is elementary; see also \cite[Lemmas 2, 3]{vas_width} for a more general case).
\begin{Lem}
\label{x_log_eq}
Let $\alpha>0$, $\beta \in \R$. Then, for sufficiently large $x>0$, the equation $y^\alpha (\log y)^\beta = x$ has the unique solution $y(x)\in [2, \, \infty)$; in addition, $y(x) \underset{\alpha,\beta}{\asymp} x^{1/\alpha} (\log x)^{‐\beta/\alpha}$, $\log y(x) \underset{\alpha,\beta}{\asymp} \log x$.
\end{Lem}

\vskip 0.3cm

{\bf The upper estimate in Theorem \ref{main1}.} We have
$$
\Sigma \le \sum \limits _{m:\, 2^{2m/q} k_m^{2/\sigma}\le n\le 2^m k_m} d_n(2^{‐(\alpha_s+1/q‐1/p_s)m} B_{p_s,\theta_s}^{2^m, k_m}, \, l_{q,\sigma}^{2^m, k_m}) + $$$$ + \sum \limits _{2^{2m/q} k_m^{2/\sigma}> n} d_0(2^{‐(\alpha_s+1/q‐1/p_s)m} B_{p_s,\theta_s}^{2^m, k_m}, \, l_{q,\sigma}^{2^m, k_m})=:\Sigma_1.
$$

Recall that $k_m \stackrel{\eqref{km_md1}}{=} m^{d‐1}$ for $m\in \N$ and that $\alpha_s>0$. We apply Theorem \ref{1mixed}.

If $\theta_s\ge \sigma$, we have
$$
\Sigma_1 \stackrel{\eqref{phi1}}{\underset{q,\sigma}{\lesssim}} \sum \limits _{2^m k_m\ge n} 2^{‐(\alpha_s+1/q‐1/p_s)m} \cdot 2^{m(1/q‐1/p_s)} k_m^{1/\sigma ‐ 1/\theta_s} \underset{\mathfrak{Z}}{\lesssim} 2^{‐\alpha_s m_0(n)} m_0(n)^{(d‐1)(1/\sigma ‐ 1/\theta_s)} =: \Sigma_2,
$$
where $m_0(n)$ is defined by the equation $2^{m_0(n)} m_0(n)^{d‐1} = n$. Then, by Lemma \ref{x_log_eq}, we get
\begin{align}
\label{2m01}
2^{m_0(n)} \underset{d}{\asymp} n (\log n)^{‐d+1}, \quad m_0(n) \underset{d}{\asymp} \log n.
\end{align}
Hence
$$
\Sigma_2 \underset{\mathfrak{Z}}{\asymp} n^{‐\alpha_s} (\log n)^{(d‐1)(\alpha_s+1/\sigma ‐1/\theta_s)}.
$$

For $\theta_s<\sigma$,
$$
\Sigma_1 \stackrel{\eqref{phi2}}{\underset{q,\sigma}{\lesssim}} \sum \limits _{2^m k_m^{2/\sigma}>n} 2^{‐\alpha_s m} + \sum \limits _{m:\, 2^m k_m^{2/\sigma} \le n \le 2^m k_m} 2^{‐\alpha_s m} (n^{‐1/2} 2^{m/2} k_m^{1/\sigma}) ^{\omega_{\theta_s,\sigma}} =: \Sigma_2.
$$

If $\alpha_s > \frac{\omega_{\theta_s,\sigma}}{2}$, we define $m_0(n)$ by the equation $2^{m_0(n)} m_0(n)^{d‐1} = n$ and get
$$
\Sigma_2 \underset{\mathfrak{Z}}{\lesssim} 2^{‐\alpha_s m_0(n)} (n^{‐1/2} 2^{m_0(n)/2} k_{m_0(n)}^{1/\sigma}) ^{\omega_{\theta_s,\sigma}} \stackrel{\eqref{2m01}}{\underset{\mathfrak{Z}}{\asymp}} n^{‐\alpha_s} (\log n)^{(d‐1)(\alpha_s+ 1/\sigma‐1/\max\{\theta_s, \, 2\})}.
$$
If $\alpha_s< \frac{\omega_{\theta_s,\sigma}}{2}$, we define $m_0(n)$ by the equation $2^{m_0(n)} m_0(n)^{2(d‐1)/\sigma} = n$. Then $2^{m_0(n)} \underset{d,\sigma}{\asymp} n (\log n)^{‐2(d‐1)/\sigma}$ (see Lemma \ref{x_log_eq}) and
$$
\Sigma_2 \underset{\mathfrak{Z}}{\lesssim} 2^{‐\alpha_s m_0(n)} \underset{\mathfrak{Z}}{\asymp} n^{‐\alpha_s} (\log n) ^{(d‐1)\alpha_s \cdot 2/\sigma}.
$$

\begin{Rem}
\label{rem_t1} For all cases in Theorem \ref{main1}, inequalities \eqref{22m0qkm02s} hold; in addition, $2^{m_0(n)} \underset{\mathfrak{Z}}{\asymp} n(\log n)^{\kappa}$ for some $\kappa \in \R$ and $$2^{‐(\alpha_s+1/q‐1/p_s)m_0(n)} \Phi(p_s, \, \theta_s, \, q, \, \sigma, \, 2^{m_0(n)}, \, k_{m_0(n)}, \, n)\underset{\mathfrak{Z}}{\asymp} n^{‐\alpha_*}(\log n)^{\beta_*(d‐1)}.$$
\end{Rem}

\vskip 0.3cm

{\bf The upper estimate in Theorem \ref{main2}.} We have
$$
\Sigma \le \sum \limits _{m:\, 2^{2m/q}k_m^{2/\sigma} \le n \le 2^mk_m} d_n(2^{‐(\alpha_1 + 1/q ‐1/p_1)m} B_{p_1,\theta_1}^{2^m,k_m}, \, l_{q, \sigma} ^{2^m,k_m})+$$$$+ \sum \limits _{2^{2m/q}k_m^{2/\sigma}> n} d_0(2^{‐(\alpha_1 + 1/q ‐1/p_1)m} B_{p_1,\theta_1}^{2^m,k_m}, \, l_{q, \sigma} ^{2^m,k_m})=:\Sigma_1.
$$

Let $\theta_1 \le 2$. Then
$$
\Sigma_1 \stackrel{\eqref{phi6}}{\underset{q, \sigma}{\lesssim}} \sum \limits _{m:\, 2^{2m/q}k_m^{2/\sigma} \le n \le 2^m k_m} 2^{‐(\alpha_1 + 1/q ‐1/p_1)m} n^{‐1/2} 2^{m/q} k_m^{1/\sigma} + $$$$ +\sum \limits _{2^{2m/q}k_m^{2/\sigma}> n} 2^{‐(\alpha_1 + 1/q ‐1/p_1)m} =: \Sigma_2.
$$
If $\alpha_1 > 1/p_1$, we define $m_0(n)$ by the equation $2^{m_0(n)} m_0(n)^{d‐1} = n$. Then  
$$
\Sigma_2 \stackrel{\eqref{2m01}}{\underset{\mathfrak{Z}} {\lesssim}} n^{‐\alpha_1 ‐ 1/2 + 1/p_1} (\log n)^{(d‐1)(\alpha_1 +1/\sigma ‐ 1/p_1)}.
$$
If $\alpha_1< 1/p_1$, we define $m_0(n)$ by the equation $2^{2m_0(n)/q} m_0(n)^{2(d‐1)/\sigma} = n$. Then 
\begin{align}
\label{2m0q}
2^{m_0(n)} \underset{d,q,\sigma}{\asymp} n^{q/2} (\log n)^{‐(d‐1)q/\sigma},
\end{align}
$$
\Sigma_2 \underset{\mathfrak{Z}}{\lesssim} n^{‐\frac q2(\alpha_1 + 1/q ‐ 1/p_1)} (\log n)^{(d‐1)(\alpha_1 +1/q ‐ 1/p_1)q/\sigma}
$$
(recall that $\alpha_1+\frac 1q‐\frac{1}{p_1}>0$ by the conditions of the theorem).

Let now $2< \theta_1 < \sigma$. Then
$$
\Sigma_1 \stackrel{\eqref{phi5}}{\underset{q,\sigma}{\lesssim}} \sum \limits _{2^{2m/q} k_m^{2/\sigma}>n} 2^{‐(\alpha_1 + 1/q ‐ 1/p_1) m } + \sum \limits _{m:\, 2^{2m/q} k_m^{2/\sigma} \le n < 2^{2m/q} k_m} 2^{‐(\alpha_1 + 1/q ‐ 1/p_1) m} (n^{‐1/2} 2^{m/q} k_m^{1/\sigma}) ^{\omega _{\theta_1,\sigma}} + $$$$+\sum \limits _{m:\, 2^{2m/q} k_m \le n \le 2^m k_m} 2^{‐(\alpha_1 + 1/q ‐ 1/p_1) m} n^{‐1/2} 2^{m/q} k_m^{1/ \sigma ‐1/\theta_1 + 1/2} =: \Sigma_2.
$$
If $\alpha_1 > 1/p_1$, we define $m_0(n)$ by the equation $2^{m_0(n)} m_0(n)^{d‐1} = n$ and get
\begin{align}
\label{sigm2_1}
\Sigma_2 \stackrel{\eqref{2m01}}{\underset{\mathfrak{Z}}{\lesssim}} n^{‐\alpha_1 ‐1/2 + 1/p_1} (\log n)^{(d‐1) (\alpha_1 ‐1/p_1 + 1/\sigma ‐ 1/\theta_1 + 1/2)}.
\end{align}
If $\alpha_1 < 1/p_1$ and $\alpha_1 + \frac 1q ‐\frac{1}{p_1} > \frac{\omega_{\theta_1,\sigma}}{q}$, we define $m_0(n)$ by the equation $2^{2m_0(n)/q} m_0(n)^{d‐1} = n$ and get $2^{m_0(n)} \underset{q,d}{\asymp} n^{q/2} (\log n)^{(‐d+1)q/2}$, $m_0(n) \underset{q,d}{\asymp} \log n$,
\begin{align}
\label{sigm2_2}
\Sigma_2 \underset{\mathfrak{Z}}{\lesssim} n^{‐\frac q2(\alpha_1 +1/q ‐ 1/p_1)} (\log n) ^{(d‐1)[(\alpha_1 + 1/q‐1/p_1)q/2 + 1/\sigma ‐1/\theta_1]}.
\end{align}
If $\alpha_1 + \frac 1q ‐\frac{1}{p_1} < \frac{\omega_{\theta_1,\sigma}}{q}$, we define $m_0(n)$ by the equation $2^{2m_0(n)/q} m_0(n)^{2(d‐1)/\sigma} = n$, take into account that $\alpha_1 + \frac 1q ‐\frac{1}{p_1}>0$ by the conditions of the theorem and get
$$
\Sigma_2 \stackrel{\eqref{2m0q}}{\underset{\mathfrak{Z}}{\lesssim}} n^{‐\frac q2(\alpha_1 +1/q ‐ 1/p_1)} (\log n) ^{(d‐1)(\alpha_1 + 1/q ‐ 1/p_1) q/\sigma}.
$$

Finally, we consider the case $\theta_1\ge \sigma$. Then
$$
\Sigma_1 \stackrel{\eqref{phi3}}{\underset{q,\sigma}{\lesssim}} \sum \limits _{2^{2m/q} k_m\ge n} 2^{‐(\alpha_1 + 1/q ‐1/p_1)m} k_m ^{1/\sigma ‐ 1/\theta_1} + $$$$+\sum \limits _{m:\, 2^{2m/q} k_m < n \le 2^m k_m} 2^{‐(\alpha_1 + 1/q ‐1/p_1)m} n^{‐1/2} 2^{m/q} k_m^{1/\sigma ‐1/\theta_1 + 1/2} =: \Sigma_2.
$$
If $\alpha_1 > 1/p_1$, we define $m_0(n)$ by the equation $2^{m_0(n)} m_0(n)^{d‐1} = n$ and get
\eqref{sigm2_1}.

If $\alpha_1 < 1/p_1$, we define $m_0(n)$ by the equation $2^{2m_0(n)/q} m_0(n)^{d‐1} = n$ and get
\eqref{sigm2_2}.

\begin{Rem}
\label{rem_t2} In Theorem \ref{main2} for $\alpha_1>1/p_1$ we have $2^{m_0(n)} \underset{\mathfrak{Z}}{\asymp} n(\log n)^{\kappa}$, and, for $\alpha_1<1/p_1$, we have $2^{m_0(n)} \underset{\mathfrak{Z}}{\asymp} n^{q/2}(\log n)^{\kappa}$ for some $\kappa \in \R$. In both cases \eqref{22m0qkm02s} holds. In addition, $$2^{‐(\alpha_1+1/q‐1/p_1)m_0(n)} \Phi(p_1, \, \theta_1, \, q, \, \sigma, \, 2^{m_0(n)}, \, k_{m_0(n)}, \, n)\underset{\mathfrak{Z}}{\asymp} n^{‐\alpha_*}(\log n)^{\beta_*(d‐1)}.$$
\end{Rem}

\vskip 0.3cm

{\bf The upper estimate in Theorem \ref{main3}.} By assertion 2 of Proposition \ref{h_prop}, we have $s\in J$. Further,
$$
\Sigma \le \sum \limits _{m:\, 2^{2m/q} k_m^{2/\sigma} \le n \le 2^m k_m} 2^{‐(\alpha_s + 1/q ‐ 1/p_s)m} d_n(B_{p_s,\theta_s} ^{2^m,k_m}, \, l_{q, \sigma} ^{2^m, k_m}) + $$$$+ \sum \limits _{2^{2m/q} k_m^{2/\sigma} > n} 2^{‐(\alpha_s + 1/q ‐ 1/p_s)m} d_0(B_{p_s,\theta_s} ^{2^m,k_m}, \, l_{q, \sigma} ^{2^m, k_m}) =: \Sigma_1.
$$
If $\theta_s \ge \sigma$, then
$$
\Sigma_1 \stackrel{\eqref{phi3}}{\underset{q,\sigma}{\lesssim}} \sum \limits _{2^{2m/q} k_m\ge n} 2^{‐(\alpha_s + 1/q ‐1/p_s)m} k_m^{1/\sigma ‐1/\theta_s} +$$$$+ \sum \limits _{m:\, 2^{2m/q} k_m < n \le 2^mk_m} 2^{‐(\alpha_s + 1/q ‐1/p_s)m} k_m^{1/\sigma ‐ 1/\theta_s} (n^{‐1/2} 2^{m/q} k_m^{1/2}) ^{\omega_{p_s,q}} =: \Sigma_2.
$$
We define the number $m_0(n)$ by the equation $2^{m_0(n)} m_0(n) ^{d‐1} = n$. By conditions of the theorem, $\alpha_s> \frac{\omega_{p_s,q}}{2} = \frac 12\cdot \frac{1/p_s‐1/q}{1/2‐1/q}$ (the last equality holds since $s\in J$). Hence
\begin{align}
\label{sigma2_teor3}
\Sigma_2 \stackrel{\eqref{2m01}}{\underset{\mathfrak{Z}}{\lesssim}} n^{‐\alpha_s} (\log n)^{(d‐1) (\alpha_s + 1/\sigma ‐1/\theta_s)}.
\end{align}

Let $2< \theta_s<\sigma$, $\omega_{p_s,q} \ge \omega_{\theta_s,\sigma}$. Then
$$
\Sigma_1 \stackrel{\eqref{phi5}}{\underset{q, \sigma}{\lesssim}} \sum \limits _{2^{2m/q} k_m^{2/\sigma}> n} 2^{‐(\alpha_s + 1/q‐1/p_s) m} + \sum \limits _{m:\, 2^{2m/q} k_m^{2/\sigma} \le n \le 2^{2m/q} k_m} 2^{‐(\alpha_s + 1/q‐1/p_s) m} (n^{‐1/2} 2^{m/q} k_m ^{1/\sigma}) ^{\omega_{\theta_s,\sigma}} +
$$
$$
+ \sum \limits _{m:\, 2^{2m/q} k_m< n\le 2^m k_m} 2^{‐(\alpha_s + 1/q‐1/p_s) m} k_m^{1/\sigma ‐ 1/\theta_s} (n^{‐1/2} 2^{m/q} k_m ^{1/2}) ^{\omega_{p_s,q}}=:\Sigma_2.
$$
The number $m_0(n)$ as defined as in the previous case, and we get \eqref{sigma2_teor3}.

Let $1\le \theta_s < \sigma$, $\omega_{p_s,q} < \omega_{\theta_s,\sigma}$. Then
$$
\Sigma_1 \stackrel{\eqref{phi4}}{\underset{q,\sigma}{\lesssim}} \sum \limits _{2^{2m/q} k_m^{2/\sigma}> n} 2^{‐(\alpha_s + 1/q‐1/p_s)m} + \sum \limits _{m:\, 2^{2m/q} k_m^{2/\sigma} \le n \le 2^m k_m^{2/\sigma}} 2^{‐(\alpha_s + 1/q‐1/p_s)m} (n^{‐1/2} 2^{m/q} k_m^{1/\sigma}) ^{\omega _{p_s,q}} +
$$
$$
+ \sum \limits _{m:\, 2^m k_m^{2/\sigma}< n \le 2^m k_m} 2^{‐(\alpha_s + 1/q‐1/p_s)m} 2^{m(1/q ‐ 1/p_s)} (n^{‐1/2} 2^{m/2} k_m^{1/\sigma}) ^{\omega _{\theta_s, \sigma}} = \Sigma_2. 
$$
If $\alpha_s> \frac{\omega_{\theta_s,\sigma}}{2}$, we define $m_0(n)$ by the equation $2^{m_0(n)} m_0(n) ^{d‐1} = n$ and get
$$
\Sigma_2 \stackrel{\eqref{2m01}}{\underset{\mathfrak{Z}}{\lesssim}} n^{‐\alpha_s} (\log n)^{(d‐1)(\alpha_s + 1/\sigma ‐ 1/\max\{\theta_s,\, 2\})}.
$$
If $\alpha_s< \frac{\omega_{\theta_s,\sigma}}{2}$, we define $m_0(n)$ by the equation $2^{m_0(n)} m_0(n) ^{(d‐1)2/\sigma} = n$ and get $2^{m_0(n)} \underset{d,\sigma}{\asymp} n (\log n)^{2(‐d+1)/\sigma}$, 
$$
\Sigma_2 \underset{\mathfrak{Z}}{\lesssim} n^{‐\alpha_s}
(\log n)^{\frac{2(d‐1)\alpha_s}{\sigma}}
$$
(here we take into account that $\alpha_s> \frac{\omega_{p_s,q}}{2}$, $2<p_s<q$).

\begin{Rem}
\label{rem_t3} In Theorem \ref{main3} in all cases we get \eqref{22m0qkm02s} and $2^{m_0(n)} \underset{\mathfrak{Z}}{\asymp} n(\log n)^{\kappa}$ for some $\kappa \in \R$. In addition, $$2^{‐(\alpha_s+1/q‐1/p_s)m_0(n)} \Phi(p_s, \, \theta_s, \, q, \, \sigma, \, 2^{m_0(n)}, \, k_{m_0(n)}, \, n)\underset{\mathfrak{Z}}{\asymp} n^{‐\alpha_*}(\log n)^{\beta_*(d‐1)}.$$
\end{Rem}

\vskip 0.3cm

{\bf The upper estimate in Theorem \ref{main4}.} From Theorem \ref{emb_s} and \eqref{lam_ij_def}, \eqref{3al_st}, \eqref{3bet_st} it follows that $\cap _{j=1}^s 2^{‐(\alpha_j+1/q‐1/p_j)m} B_{p_j,\theta_j} ^{2^m,k_m} \subset 2^{‐(\alpha_*+1/q‐1/2)m} B^{2^m,k_m}_{2,\theta_*}$. Hence
$$
\Sigma \le \sum \limits _{m:\, 2^{2m/q} k_m^{2/\sigma} \le n \le 2^m k_m} 2^{‐(\alpha_* + 1/q ‐1/2) m} d_n(B_{2,\theta_*} ^{2^m,k_m}, \, l_{q,\sigma} ^{2^m,k_m})+$$$$+ \sum \limits _{2^{2m/q} k_m ^{2/\sigma} > n} 2^{‐(\alpha_* + 1/q ‐1/2) m} d_0(B_{2,\theta_*} ^{2^m,k_m}, \, l_{q,\sigma} ^{2^m,k_m}) =: \Sigma_1.
$$

If $\theta_*>2$, we can apply the formulas from the proof of the upper estimate in Theorem \ref{main3}, replacing $p_s$ by $2$, $\theta_s$ by $\theta_*$, $\alpha_s$ by $\alpha_*$. From the condition $\alpha_*> \frac 12$ we get that $\Sigma_1 \underset{\mathfrak{Z}}{\lesssim} n^{‐\alpha_*} (\log n)^{(d‐1)(\alpha_* + 1/\sigma ‐ 1/\theta_*)}$.

If $\theta_*\le 2$, we have
$$
\Sigma_1 \stackrel{\eqref{phi6}}{\underset{q,\sigma}{\lesssim}} \sum \limits _{m:\, 2^{2m/q} k_m^{2/\sigma} \le n \le 2^m k_m} 2^{‐(\alpha_* + 1/q ‐1/2) m} n^{‐1/2} 2^{m/q} k_m^{1/\sigma} + \sum \limits _{2^{2m/q} k_m ^{2/\sigma} > n} 2^{‐(\alpha_* + 1/q ‐1/2) m} =:\Sigma_2.
$$
We define the number $m_0(n)$ by the equation $2^{m_0(n)} m_0(n) ^{d‐1} = n$ and get
$$
\Sigma_2 \underset{\mathfrak{Z}}{\lesssim} n^{‐\alpha_*} (\log n)^{(d‐1)(\alpha_* + 1/\sigma ‐ 1/2)}.
$$

\begin{Rem}
\label{rem_t4} In Theorem \ref{main4} in all cases we have \eqref{22m0qkm02s} and $2^{m_0(n)} \underset{\mathfrak{Z}}{\asymp} n(\log n)^{\kappa}$ for some $\kappa \in \R$. In addition, $$2^{‐(\alpha_*+1/q‐1/2)m_0(n)} \Phi(2, \, \theta_*, \, q, \, \sigma, \, 2^{m_0(n)}, \, k_{m_0(n)}, \, n)\underset{\mathfrak{Z}}{\asymp} n^{‐\alpha_*}(\log n)^{\beta_*(d‐1)}.$$
\end{Rem}

\vskip 0.3cm

{\bf The upper estimate in Theorem \ref{main5}.} By assertion 4 of Proposition \ref{h_prop}, we have $1\in J$. Further,
$$
\Sigma \le \sum \limits _{m:\, 2^{2m/q} k_m^{2/\sigma} \le n \le 2^m k_m} 2^{‐(\alpha_1 + 1/q ‐1/p_1) m} d_n (B_{p_1,\theta_1} ^{2^m, k_m}, \, l_{q, \sigma}^{2^m,k_m}) +
$$
$$
+ \sum \limits _{2^{2m/q} k_m^{2/\sigma} > n} 2^{‐(\alpha_1 + 1/q ‐1/p_1) m} d_0 (B_{p_1,\theta_1} ^{2^m, k_m}, \, l_{q, \sigma}^{2^m,k_m}) =: \Sigma_1.
$$

Let $\theta_1 \ge \sigma$. Then
$$
\Sigma_1 \stackrel{\eqref{phi3}}{\underset{\mathfrak{Z}}{\lesssim}} \sum \limits _{2^{2m/q} k_m\ge n}2^{‐(\alpha_1 + 1/q ‐1/p_1) m} k_m ^{1/\sigma ‐1/\theta_1} + $$$$+\sum \limits _{m:\, 2^{2m/q} k_m < n\le 2^m k_m} 2^{‐(\alpha_1 + 1/q ‐1/p_1) m} k_m ^{1/\sigma ‐1/\theta_1} (n^{‐1/2} 2^{m/q} k_m^{1/2}) ^{\omega _{p_1,q}}=:\Sigma_2.
$$
We define the number $m_0(n)$ by the equation $2^{2m_0(n)/q} m_0(n)^{d‐1} = n$. Then $2^{m_0(n)} \underset{d,q}{\asymp} n^{q/2} (\log n)^{(‐d+1)q/2}$, $m_0(n) \underset{d,q}{\asymp} \log n$. By the conditions of the theorem, $\alpha_1 < \frac{\omega _{p_1,q}}{2}$, and, as it was mentioned above, $1\in J$. Therefore,
\begin{align}
\label{sigma2_teor4}
\Sigma_2 \underset{\mathfrak{Z}}{\lesssim} n^{‐\frac q2 (\alpha_1 + 1/q‐1/p_1)} (\log n) ^{(d‐1)[(\alpha_1+ 1/q ‐1/p_1)q/2 + 1/\sigma ‐1/\theta_1]}.
\end{align}

Let $2< \theta_1<\sigma$, $\omega_{p_1,q} \ge \omega _{\theta_1, \sigma}$. Then 
$$
\Sigma_1 \stackrel{\eqref{phi5}}{\underset{q,\sigma}{\lesssim}} \sum \limits _{2^{2m/q} k_m^{2/\sigma}> n} 2^{‐(\alpha_1 + 1/q ‐ 1/p_1)m} + \sum \limits _{m:\, 2^{2m/q} k_m^{2/\sigma} \le n \le 2^{2m/q} k_m} 2^{‐(\alpha_1 + 1/q ‐ 1/p_1)m} (n^{‐1/2} 2^{m/q} k_m^{1/\sigma}) ^{\omega_{\theta_1,\sigma}} +
$$
$$
+ \sum \limits _{m:\, 2^{2m/q} k_m < n \le 2^m k_m} 2^{‐(\alpha_1 + 1/q ‐ 1/p_1)m} k_m^{1/\sigma ‐1/\theta_1} (n^{‐1/2} 2^{m/q} k_m^{1/2}) ^{\omega_{p_1,q}} =: \Sigma_2.
$$

In the case $\alpha_1 + 1/q ‐1/p_1 > \frac{\omega_{\theta_1,\sigma}}{q}$, we define $m_0(n)$ by the equation $2^{2m_0(n)/q} m_0(n) ^{d‐1} = n$
and get \eqref{sigma2_teor4}.

In the case $\alpha_1 + 1/q ‐1/p_1 < \frac{\omega_{\theta_1,\sigma}}{q}$ we define the number $m_0(n)$ by the equation $2^{2m_0(n)/q} m_0(n)^{2(d‐1)/\sigma} = n$. Then $2^{m_0(n)} \underset{q,\sigma,d}{\asymp} n^{q/2} (\log n)^{(‐d+1)q/\sigma}$ and 
\begin{align}
\label{sigma2_teor4_1}
\Sigma_2 \underset{\mathfrak{Z}}{\lesssim} n^{‐\frac q2 (\alpha_1 + 1/q‐1/p_1)} (\log n)^{(d‐1)(\alpha_1 + 1/q ‐1/p_1) q/\sigma}.
\end{align}

Let $1\le \theta_1< \sigma$, $\omega _{p_1,q} < \omega _{\theta_1, \sigma}$. Then
$$
\Sigma_1 \stackrel{\eqref{phi4}}{\underset{q,\sigma}{\lesssim}} \sum \limits _{2^{2m/q} k_m^{2/\sigma}> n} 2^{‐(\alpha_1 + 1/q ‐1/p_1) m} + \sum \limits _{m:\, 2^{2m/q} k_m^{2/\sigma} \le n \le 2^{m} k_m^{2/\sigma}} 2^{‐(\alpha_1 + 1/q ‐1/p_1) m} (n^{‐1/2} 2^{m/q} k_m^{1/\sigma}) ^{\omega _{p_1,q}} +
$$
$$
+ \sum \limits _{m:\, 2^{m} k_m^{2/\sigma} < n \le 2^m k_m} 2^{‐(\alpha_1 + 1/q ‐1/p_1) m} \cdot 2^{m(1/q‐1/p_1)} (n^{‐1/2} 2^{m/2} k_m^{1/\sigma}) ^{\omega _{\theta_1,\sigma}} =: \Sigma_2.
$$
Recall that $\alpha_1 < \frac{\omega _{p_1,q}}{2}$. We define the number $m_0(n)$ by the equation $2^{2m_0(n)/q} m_0(n)^{2(d‐1)/\sigma} = n$ and obtain \eqref{sigma2_teor4_1}.

\begin{Rem}
\label{rem_t5} In Theorem \ref{main5} in all cases we have \eqref{22m0qkm02s} and $2^{m_0(n)} \underset{\mathfrak{Z}}{\asymp} n^{q/2}(\log n)^{\kappa}$ for some $\kappa \in \R$. In addition, $$2^{‐(\alpha_1+1/q‐1/p_1)m_0(n)} \Phi(p_1, \, \theta_1, \, q, \, \sigma, \, 2^{m_0(n)}, \, k_{m_0(n)}, \, n)\underset{\mathfrak{Z}}{\asymp} n^{‐\alpha_*}(\log n)^{\beta_*(d‐1)}.$$
\end{Rem}

\vskip 0.3cm

{\bf Upper estimate in Theorem \ref{main6}.} From the definition of the numbers $\hat{\alpha}_i$, $\hat{p}_i$, $\hat{\theta}_i$, \eqref{lam_ij_def} and Theorem \ref{emb_s} it follows that
$$
\cap _{j=1}^s 2^{‐m(\alpha_j+1/q‐1/p_j)} B^{2^m,k_m} _{p_j,\theta_j} \subset 2^{‐(\hat\alpha_i + 1/q ‐1/\hat p_i)m} B_{\hat p_i, \hat \theta _i}^{2^m, k_m}, \quad i=1, \, 2.
$$
Hence
$$
\Sigma \le \sum \limits _{m:\, 2^{2m/q}k_m^{2/\sigma} \le n \le 2^mk_m} d_n(2^{‐(\hat\alpha_1 + 1/q ‐1/\hat p_1)m} B_{\hat p_1, \hat \theta _1}^{2^m, k_m} \cap 2^{‐(\hat\alpha_2 + 1/q ‐1/\hat p_2)m} B_{\hat p_2, \hat \theta _2}^{2^m, k_m}, \, l_{q,\sigma}^{2^m, k_m})+
$$
$$
+\sum \limits _{2^{2m/q}k_m^{2/\sigma} > n} d_0(2^{‐(\hat\alpha_1 + 1/q ‐1/\hat p_1)m} B_{\hat p_1, \hat \theta _1}^{2^m, k_m} \cap 2^{‐(\hat\alpha_2 + 1/q ‐1/\hat p_2)m} B_{\hat p_2, \hat \theta _2}^{2^m, k_m}, \, l_{q,\sigma}^{2^m, k_m}) \underset{\mathfrak{Z}}{\asymp}
$$
$$
\asymp \sum \limits _{2^mk_m \ge n} d_n(2^{‐(\hat\alpha_1 + 1/q ‐1/\hat p_1)m} B_{\hat p_1, \hat \theta _1}^{2^m, k_m} \cap 2^{‐(\hat\alpha_2 + 1/q ‐1/\hat p_2)m} B_{\hat p_2, \hat \theta _2}^{2^m, k_m}, \, l_{q,\sigma}^{2^m, k_m}) =: \Sigma_1
$$
(the last order equality follows from Remark \ref{dn_d0_phi}).

By Remark \ref{rem_pi_hat}, we have $2\le \hat p_i\le q$; hence $\omega_{\hat p_i,q} = \omega'_{\hat p_i,q}$ ($i=1, \, 2$).

From \eqref{h0_def}, \eqref{h1_def}, \eqref{phi_j_def} and the definition of $\hat \alpha_i$, $\hat p_i$ for $i=1, \, 2$ it follows that
\begin{align}
\label{h_neighborhood}
h(t) = \begin{cases} t\Bigl(\hat\alpha_1 ‐\frac 12 \cdot \frac{1/\hat p_1‐1/q}{1/2‐1/q}\Bigr) +\frac 12 \cdot \frac{1/\hat p_1‐1/q}{1/2‐1/q}, & t \in [t_{l_*‐1}, \, t_{l_*}], \\ t\Bigl(\hat\alpha_2 ‐\frac 12 \cdot \frac{1/\hat p_2‐1/q}{1/2‐1/q}\Bigr) +\frac 12 \cdot \frac{1/\hat p_2‐1/q}{1/2‐1/q}, & t \in [t_{l_*}, \, t_{l_*+1}]. \end{cases}
\end{align}

Since $t_{l_*}\in (1, \, q/2)$ is the point of strict minimum of the function $h$, we have
\begin{align}
\label{str_min_h} \hat \alpha_1 ‐ \frac{\omega_{\hat p_1,q}}{2} < 0, \quad \hat \alpha_2 ‐ \frac{\omega_{\hat p_2,q}}{2} > 0.
\end{align}

Let $\omega' _{\hat p_1,q} \ge \omega' _{\hat \theta_1,\sigma}$, $\omega' _{\hat p_2,q} \ge \omega' _{\hat \theta_2,\sigma}$. Then
$$
\Sigma_1 \stackrel{\eqref{phi3}, \eqref{phi5}}{\underset{q,\sigma}{\lesssim}} \sum \limits _{2^m k_m \ge n} \min \{2^{‐(\hat \alpha_1 + 1/q ‐1/\hat p_1)m} k_m^{1/\sigma ‐1/\hat \theta_1} (n^{‐1/2}2^{m/q} k_m^{1/2}) ^{\omega _{\hat p_1,q}},$$$$ 2^{‐(\hat \alpha_2 + 1/q ‐1/\hat p_2)m} k_m^{1/\sigma ‐1/\hat \theta_2} (n^{‐1/2}2^{m/q} k_m^{1/2}) ^{\omega _{\hat p_2,q}}\} =: \Sigma_2.
$$
We define $m_0(n)$ as the solution for $m$ of the equation
\begin{align}
\label{m0n1}
\begin{array}{c}
2^{‐(\hat \alpha_1 + 1/q ‐1/\hat p_1)m} k_m^{1/\sigma ‐1/\hat \theta_1} (n^{‐1/2}2^{m/q} k_m^{1/2}) ^{\omega _{\hat p_1,q}} = \\ = 2^{‐(\hat \alpha_2 + 1/q ‐1/\hat p_2)m} k_m^{1/\sigma ‐1/\hat \theta_2} (n^{‐1/2}2^{m/q} k_m^{1/2}) ^{\omega _{\hat p_2,q}},
\end{array}
\end{align}
take into account \eqref{str_min_h} and get by direct calculations that $\Sigma _2 \underset{\mathfrak{Z}}{\lesssim} n^{‐\alpha_*}(\log n)^{(d-1)\beta_*^1}$. In addition, $2^{m_0(n)} \underset{\mathfrak{Z}}{\asymp} n^{t_{l_*}} (\log n)^\kappa$ for some $\kappa \in \R$ (see \eqref{h_neighborhood} and Lemma \ref{x_log_eq}); hence, for sufficiently large $n$, we have \eqref{22m0qkm02s}.

Let $\omega' _{\hat p_1,q} \le \omega' _{\hat \theta_1,\sigma}$, $\omega' _{\hat p_2,q} \le \omega' _{\hat \theta_2,\sigma}$. Then
$$
\Sigma_1 \stackrel{\eqref{phi4}}{\underset{q,\sigma}{\lesssim}} \sum \limits _{2^m k_m \ge n} \min \{2^{‐(\hat \alpha_1 + 1/q ‐1/\hat p_1)m} (n^{‐1/2}2^{m/q} k_m^{1/\sigma}) ^{\omega _{\hat p_1,q}},$$$$ 2^{‐(\hat \alpha_2 + 1/q ‐1/\hat p_2)m} (n^{‐1/2}2^{m/q} k_m^{1/\sigma}) ^{\omega _{\hat p_2,q}}\} =: \Sigma_2.
$$
We define the number $m_0(n)$ as the solution in $m$ of the equation
\begin{align}
\label{m0n2}
\begin{array}{c}
2^{‐(\hat \alpha_1 + 1/q ‐1/\hat p_1)m} (n^{‐1/2}2^{m/q} k_m^{1/\sigma}) ^{\omega _{\hat p_1,q}} = 
\\
= 2^{‐(\hat \alpha_2 + 1/q ‐1/\hat p_2)m} (n^{‐1/2}2^{m/q} k_m^{1/\sigma}) ^{\omega _{\hat p_2,q}},
\end{array}
\end{align}
apply \eqref{str_min_h} and get $\Sigma _2 \underset{\mathfrak{Z}}{\lesssim} n^{‐\alpha_*}(\log n)^{(d‐1)\beta_*^2}$. As in the previous case, we have \eqref{22m0qkm02s} for large $n$.

Let $\omega' _{\hat p_1,q} < \omega' _{\hat \theta_1, \sigma}$, $\omega' _{\hat p_2,q} > \omega' _{\hat \theta_2, \sigma}$. Then, by \eqref{hat_p_th_def}, \eqref{hat_alp_def}, Remark \ref{hat_lam_0_1} and Theorem \ref{emb_s}, we have
$$
2^{‐(\hat\alpha_1 + 1/q ‐1/\hat p_1)m} B_{\hat p_1, \hat \theta _1}^{2^m, k_m} \cap 2^{‐(\hat\alpha_2 + 1/q ‐1/\hat p_2)m} B_{\hat p_2, \hat \theta _2}^{2^m, k_m} \subset 2^{‐(\hat\alpha + 1/q ‐1/\hat p)m} B_{\hat p, \hat \theta}^{2^m, k_m}.
$$
This implies that
$$
\Sigma_1 \stackrel{\eqref{phi3}, \eqref{phi4}, \eqref{phi5}}{\underset{q,\sigma}{\lesssim}} \sum \limits _{2^m k_m \ge n} \min \{2^{‐(\hat \alpha_1 + 1/q ‐1/\hat p_1)m} (n^{‐1/2}2^{m/q} k_m^{1/\sigma}) ^{\omega _{\hat p_1,q}},$$$$ 2^{‐(\hat \alpha_2 + 1/q ‐1/\hat p_2)m} k_m^{1/\sigma ‐1/\hat \theta_2} (n^{‐1/2}2^{m/q} k_m^{1/2}) ^{\omega _{\hat p_2,q}}, $$$$
2^{‐(\hat \alpha + 1/q ‐1/\hat p)m} (n^{‐1/2}2^{m/q} k_m^{1/\sigma}) ^{\omega _{\hat p,q}}\} =: \Sigma_2.
$$
We also notice that from the equality $\omega'_{\hat p,q} = \omega'_{\hat \theta,\sigma}$ and the inequalities $2\le \hat p_1, \, \hat p_2\le q$ it follows that
\begin{align}
\label{www} 2^{‐(\hat \alpha + 1/q ‐1/\hat p)m} (n^{‐1/2}2^{m/q} k_m^{1/\sigma}) ^{\omega _{\hat p,q}} = 2^{‐(\hat \alpha + 1/q ‐1/\hat p)m} k_m^{1/\sigma ‐1/\hat \theta} (n^{‐1/2}2^{m/q} k_m^{1/2}) ^{\omega _{\hat p,q}}.
\end{align}

We define the number $m_1(n)$ as the solution of the equation
$$
2^{‐(\hat \alpha_1 + 1/q ‐1/\hat p_1)m} (n^{‐1/2}2^{m/q} k_m^{1/\sigma}) ^{\omega _{\hat p_1,q}} =
$$
$$
= 2^{‐(\hat \alpha + 1/q ‐1/\hat p)m} (n^{‐1/2}2^{m/q} k_m^{1/\sigma}) ^{\omega _{\hat p,q}},
$$
and $m_2(n)$, as the solution of the equation 
$$
2^{‐(\hat \alpha_2 + 1/q ‐1/\hat p_2)m} k_m^{1/\sigma ‐1/\hat \theta_2} (n^{‐1/2}2^{m/q} k_m^{1/2}) ^{\omega _{\hat p_2,q}} =
$$
$$
= 2^{‐(\hat \alpha + 1/q ‐1/\hat p)m} (n^{‐1/2}2^{m/q} k_m^{1/\sigma}) ^{\omega _{\hat p,q}}.
$$
From \eqref{hat_p_th_def}, \eqref{hat_alp_def}, \eqref{www} it follows that $m_1(n)$ is the solution of \eqref{m0n2}, and $m_2(n)$ is the solution of \eqref{m0n1}.

We claim that $m_1(n)< m_2(n)$ for sufficiently large $n$. Indeed, from \eqref{m0n1}, \eqref{m0n2} and Lemma \ref{x_log_eq} it follows that
$$
2^{m_1(n)} \underset{\mathfrak{Z}}{\asymp} n^{\tilde \alpha} (\log n)^{\tilde \beta_1}, \quad 2^{m_2(n)} \underset{\mathfrak{Z}}{\asymp} n^{\tilde \alpha} (\log n)^{\tilde \beta_2},
$$
where
$$
\tilde \alpha = \frac{\omega_{\hat p_1,q}‐\omega_{\hat p_2,q}}{2\gamma}, \quad \tilde \beta_1 = \frac{1‐d}{\gamma}\cdot \frac{\omega_{\hat p_1,q}‐\omega_{\hat p_2,q}}{\sigma},
$$
$$
\tilde \beta_2 = \frac{1‐d}{\gamma}\cdot (1/\hat \theta_2‐1/\hat \theta_1 +(\omega_{\hat p_1,q}‐\omega_{\hat p_2,q})/2), \quad \gamma = \hat \alpha_2‐\hat \alpha_1+ \frac{\omega_{\hat p_1,q} ‐ \omega_{\hat p_2,q}}{2}.
$$
From \eqref{str_min_h} we get that $\gamma>0$. Hence, by the condition $\omega' _{\hat p_1,q} < \omega' _{\hat \theta_1, \sigma}$, $\omega' _{\hat p_2,q} > \omega' _{\hat \theta_2, \sigma}$, we have $\tilde \beta_1<\tilde \beta_2$.

Recall that $\zeta = \hat \alpha ‐ \frac{\omega'_{\hat p,q}}{2} \ne 0$ (see \eqref{zeta_def}).
If $\zeta >0$, we set $m_0(n) = m_1(n)$; if $\zeta <0$, we set $m_0(n) = m_2(n)$. In both cases we apply \eqref{str_min_h} and get $\Sigma _2 \underset{\mathfrak{Z}}{\lesssim} n^{‐\alpha_*}(\log n)^{(d-1)\beta_*}$.

The case $\omega' _{\hat p_1,q} > \omega' _{\hat \theta_1, \sigma}$, $\omega' _{\hat p_2,q} < \omega' _{\hat \theta_2, \sigma}$ is similar; here $m_1(n)$ is the solution of \eqref{m0n1}, $m_2(n)$ is the solution of \eqref{m0n2}, $m_0(n)= m_1(n)$ for $\zeta>0$, $m_0(n) = m_2(n)$ for $\zeta<0$. We again get $m_1(n)<m_2(n)$ for sufficiently large $n$.

\begin{Rem}
\label{rem_t7} In case 1 of Theorem \ref{main6} the number $m_0(n)$ is the solution of \eqref{m0n1}, and in case 2, it is the solution of \eqref{m0n2}. In case 3 of Theorem \ref{main6}, $m_0(n)$ is the solution of \eqref{m0n1} for $(\omega' _{\hat p_1,q} ‐ \omega' _{\hat \theta_1, \sigma})\zeta>0$, and it is the solution of \eqref{m0n2} for $(\omega' _{\hat p_1,q} ‐ \omega' _{\hat \theta_1, \sigma})\zeta<0$.
If $m_0(n)$ is the solution of \eqref{m0n1}, the value of the left‐ and right‐hand sides of this equation at the point $m_0(n)$ equal to $n^{‐\alpha_*}(\log n)^{(d‐1)\beta_*^1}$. If $m_0(n)$ is the solution of \eqref{m0n2}, the value of both sides of this equation at the point $m_0(n)$ equals to $n^{‐\alpha_*}(\log n)^{(d‐1)\beta_*^2}$. In both cases $2^{m_0(n)} \underset{\mathfrak{Z}}{\asymp} n^{t_{l_*}} (\log n)^\kappa$ for some $\kappa \in \R$; in particular, \eqref{22m0qkm02s} holds.
\end{Rem}

\begin{Rem}
\label{rem_m0m1_7071} In case 3 of Theorem \ref{main6} we have $m_0(n)=m_1(n)$ for $\zeta>0$, $m_0(n) = m_2(n)$ for $\zeta<0$. If $m_1(n)$ is the solution of \eqref{m0n1}, then $m_2(n)$ is the solution of \eqref{m0n2}, and vice versa.
\end{Rem}

\begin{Rem}
\label{m1m2} It was shown that for $(\omega' _{\hat p_1,q} ‐ \omega' _{\hat \theta_1, \sigma})(\omega' _{\hat p_2,q} ‐ \omega' _{\hat \theta_2, \sigma})<0$ and for sufficiently large $n$ we have $m_1(n) < m_2(n)$.
\end{Rem}

\section{Proof of lower estimates}

From \eqref{card_j_m}, \eqref{x_b_norm}, \eqref{discr_eq} and \eqref{mmdm1} we get
\begin{Cor}
\label{lower_est} Let $n\in \Z_+$, $m\in \N$. Then
\begin{align}
\label{low_estim} d_n(\cap _{j=1}^s SB^{\overline{r}_j}_{p_j,\theta_j}(\mathbb{T}^d), \, B^{\overline{l}}_{q, \sigma}(\mathbb{T}^d)) \underset{\mathfrak{Z}}{\gtrsim} d_n (\cap _{j=1}^s 2^{‐m(\alpha_j + 1/q ‐1/p_j)} B_{p_j,\theta_j} ^{2^m,k_m}, \, l_{q,\sigma} ^{2^m,k_m}),
\end{align}
where $k_m=m^{d‐1}$ (see \eqref{km_md1}).
\end{Cor}

We first formulate the theorem about estimates for the widths of an intersection of finite‐dimensional balls in a mixed norm. 

Let $A\ne \varnothing$ and let, for each $\alpha\in A$, the numbers $\nu_\alpha>0$, $p_\alpha\in [1, \, \infty]$ and $\theta_\alpha\in [1, \, \infty]$ be given. In \cite{vas_mix_sev} order estimates for the Kolmogorov $n$‐widths of $\cap _{\alpha\in A} \nu_\alpha B^{t,k}_{p_\alpha, \, \theta_\alpha}$ in $l^{t,k}_{q,\sigma}$ were obtained for $2\le q, \, \sigma<\infty$, $n\le \frac{tk}{2}$.
In order to formulate this result we need some notation.

Given $\alpha$, $\beta$, $\gamma \in A$, we denote by $\Delta_{\alpha,\beta,\gamma}$ the triangle with vertices $(1/p_\alpha, \, 1/\theta_\alpha)$, $(1/p_\beta, \, 1/\theta_\beta)$, $(1/p_\gamma, \, 1/\theta_\gamma)$. We write $\Delta_{\alpha,\beta,\gamma} \in {\cal R}$ if its vertices do not lie on the same line, i.e., they are affinely independent.

The value $\Phi(p, \, \theta, \, q, \, \sigma, \, t, \, k, \, n)$ was defined in the previous section.

We set $P=\{(p_\alpha, \, \theta_\alpha)\}_{\alpha\in A}$. The function $\nu:A \rightarrow (0, \, \infty)$ is defined by the formula $\alpha \stackrel{\nu}{\mapsto}\nu_\alpha$. The sets ${\cal N}_j={\cal N}_j(P; \, q, \, \sigma)$ ($1\le j\le 7$) and the magnitudes $\Psi_j= \Psi_j(P; \, q, \, \sigma; \, t, \, k, \, n; \, \nu)$ ($0\le j\le 7$) are defined as follows:
\begin{align}
\label{n1} \begin{array}{c} {\cal N}_1= \left\{ (\alpha, \, \beta)\in A\times A:\;p_\alpha \ne q, \; \exists \lambda _{\alpha,\beta} \in (0, \, 1):\; \frac 1q = \frac{1- \lambda _{\alpha,\beta}}{p_\alpha} + \frac{\lambda _{\alpha,\beta}}{p_\beta}\right\}, \\ \frac{1}{\theta_{\alpha,\beta}} := \frac{1-\lambda _{\alpha,\beta}}{\theta_\alpha} + \frac{\lambda _{\alpha,\beta}}{\theta_\beta}, \quad (\alpha, \, \beta) \in {\cal N}_1, \end{array}
\end{align}
\begin{align}
\label{n2} \begin{array}{c}{\cal N}_2 = \left\{ (\alpha, \, \beta)\in A\times A:\;\theta_\alpha \ne \sigma, \; \exists  \mu _{\alpha,\beta} \in (0, \, 1):\; \frac{1}{\sigma} = \frac{1-\mu _{\alpha,\beta}}{\theta_\alpha} + \frac{ \mu _{\alpha,\beta}}{\theta_\beta}\right\}, \\ \frac{1}{ p_{\alpha,\beta}} := \frac{1- \mu _{\alpha,\beta}}{p_\alpha} + \frac{\mu _{\alpha,\beta}}{p_\beta}, \quad (\alpha, \, \beta) \in {\cal N}_2, \end{array}
\end{align}
\begin{align}
\label{n3} \begin{array}{c}{\cal N}_3 = \left\{ (\alpha, \, \beta)\in A\times A:\; p_\alpha \ne 2, \; \exists \tilde \lambda _{\alpha,\beta} \in (0, \, 1):\; \frac 12 = \frac{1-\tilde \lambda _{\alpha,\beta}}{p_\alpha} + \frac{\tilde \lambda _{\alpha,\beta}}{p_\beta}\right\}, \\ \frac{1}{\tilde \theta_{\alpha,\beta}} := \frac{1-\tilde \lambda _{\alpha,\beta}}{\theta_\alpha} + \frac{\tilde \lambda _{\alpha,\beta}}{\theta_\beta}, \quad (\alpha, \, \beta) \in {\cal N}_3, \end{array}
\end{align}
\begin{align}
\label{n4} \begin{array}{c}{\cal N}_4 = \left\{ (\alpha, \, \beta)\in A\times A:\;\theta_\alpha \ne 2, \; \exists \tilde \mu _{\alpha,\beta} \in (0, \, 1):\; \frac{1}{2} = \frac{1-\tilde \mu _{\alpha,\beta}}{\theta_\alpha} + \frac{\tilde \mu _{\alpha,\beta}}{\theta_\beta}\right\}, \\ \frac{1}{\tilde p_{\alpha,\beta}} := \frac{1-\tilde \mu _{\alpha,\beta}}{p_\alpha} + \frac{\tilde \mu _{\alpha,\beta}}{p_\beta},\quad (\alpha, \, \beta) \in {\cal N}_4,  \end{array}
\end{align}
\begin{align}
\label{n5} \begin{array}{c}{\cal N}_5 = \left\{ (\alpha, \, \beta)\in A\times A:\; \exists \hat\lambda _{\alpha,\beta} \in (0, \, 1), \; \hat p_{\alpha,\beta}\in (2, \, q), \; \hat \theta_{\alpha,\beta} \in (2, \, \sigma):\right. \\ \left. \frac{1}{\hat p_{\alpha,\beta}} = \frac{1-\hat\lambda _{\alpha,\beta}}{p_\alpha} + \frac{\hat\lambda _{\alpha,\beta}}{p_\beta}, \; \frac{1}{\hat\theta_{\alpha,\beta}} = \frac{1-\hat\lambda _{\alpha,\beta}}{\theta_\alpha} + \frac{\hat\lambda _{\alpha,\beta}}{\theta_\beta}, \; \omega'_{\hat p_{\alpha,\beta},q}=\omega'_{\hat\theta_{\alpha,\beta},\sigma}\; \text{and}\; \omega'_{p_{\alpha},q} \ne \omega'_{\theta_{\alpha},\sigma}\right\}, \end{array}
\end{align}
\begin{align}
\label{n6} \begin{array}{c}{\cal N}_6 = \left\{ (\alpha, \, \beta, \, \gamma)\in A\times A \times A:\; \exists \tau_\alpha, \, \tau_\beta, \, \tau_\gamma > 0:\; \tau_\alpha+\tau_\beta+\tau_\gamma =1, \right. \\ \left.\frac{1}{q} = \frac{\tau_\alpha}{p_\alpha} + \frac{\tau_\beta}{p_\beta}+\frac{\tau_\gamma}{p_\gamma}, \; \frac{1}{\sigma} = \frac{\tau_\alpha}{\theta_\alpha} + \frac{\tau_\beta}{\theta_\beta}+\frac{\tau_\gamma}{\theta_\gamma}, \; \Delta _{\alpha,\beta,\gamma}\in {\cal R}\right\},\end{array}
\end{align}
\begin{align}
\label{n7} \begin{array}{c}{\cal N}_7 = \left\{ (\alpha, \, \beta, \, \gamma)\in A\times A \times A:\; \exists \overline{\tau}_\alpha, \, \overline{\tau}_\beta, \, \overline{\tau}_\gamma > 0:\; \overline{\tau}_\alpha+\overline{\tau}_\beta+\overline{\tau}_\gamma =1, \right. \\ \left.\frac{1}{2} = \frac{\overline{\tau}_\alpha}{p_\alpha} + \frac{\overline{\tau}_\beta}{p_\beta}+\frac{\overline{\tau}_\gamma}{p_\gamma}, \; \frac{1}{2} = \frac{\overline{\tau}_\alpha}{\theta_\alpha} + \frac{\overline{\tau}_\beta}{\theta_\beta}+\frac{\overline{\tau}_\gamma}{\theta_\gamma}, \; \Delta _{\alpha,\beta,\gamma}\in {\cal R}\right\},\end{array}
\end{align}
\begin{align}
\label{psi0} \Psi_0 = \inf _{\alpha \in A} \nu_\alpha \Phi(p_\alpha, \, \theta _\alpha, \, q, \, \sigma, \, t, \, k, \, n),
\end{align}
\begin{align}
\label{psi1} \Psi_1 = \inf _{(\alpha,\, \beta) \in {\cal N}_1} \nu_\alpha^{1-\lambda_{\alpha,\beta}} \nu_\beta ^{ \lambda_{\alpha,\beta}} \Phi(q, \, \theta _{\alpha,\beta}, \, q, \, \sigma, \, t, \, k, \, n),
\end{align}
\begin{align}
\label{psi2} \Psi_2 = \inf _{(\alpha,\, \beta) \in {\cal N}_2} \nu_\alpha^{1-\mu_{\alpha,\beta}} \nu_\beta ^{ \mu_{\alpha,\beta}} \Phi(p_{\alpha,\beta}, \, \sigma, \, q, \, \sigma, \, t, \, k, \, n),
\end{align}
\begin{align}
\label{psi3} \Psi_3 = \inf _{(\alpha,\, \beta) \in {\cal N}_3} \nu_\alpha^{1-\tilde \lambda_{\alpha,\beta}} \nu_\beta ^{\tilde \lambda_{\alpha,\beta}} \Phi(2, \, \tilde\theta _{\alpha,\beta}, \, q, \, \sigma, \, t, \, k, \, n),
\end{align}
\begin{align}
\label{psi4} \Psi_4 = \inf _{(\alpha,\, \beta) \in {\cal N}_4} \nu_\alpha^{1-\tilde \mu_{\alpha,\beta}} \nu_\beta ^{\tilde \mu_{\alpha,\beta}} \Phi(\tilde p_{\alpha,\beta}, \, 2, \, q, \, \sigma, \, t, \, k, \, n),
\end{align}
\begin{align}
\label{psi5} \Psi_5 = \inf _{(\alpha,\, \beta) \in {\cal N}_5} \nu_\alpha^{1-\hat\lambda_{\alpha,\beta}} \nu_\beta ^{\hat\lambda_{\alpha,\beta}} \Phi(\hat p_{\alpha,\beta}, \, \hat \theta_{\alpha,\beta}, \, q, \, \sigma, \, t, \, k, \, n),
\end{align}
\begin{align}
\label{psi6} \Psi_6 = \inf _{(\alpha,\, \beta,\, \gamma) \in {\cal N}_6} \nu_\alpha^{\tau_\alpha} \nu_\beta ^{\tau_\beta} \nu_\gamma^{\tau_\gamma} \Phi(q, \, \sigma, \, q, \, \sigma, \, t, \, k, \, n),
\end{align}
\begin{align}
\label{psi7} \Psi_7 = \inf _{(\alpha,\, \beta,\, \gamma) \in {\cal N}_7} \nu_\alpha^{\overline{\tau}_\alpha} \nu_\beta ^{\overline{\tau}_\beta} \nu_\gamma^{\overline{\tau}_\gamma} \Phi(2, \, 2, \, q, \, \sigma, \, t, \, k, \, n)
\end{align}
(the infimum of the empty set is $+\infty$); the numbers $\theta_{\alpha,\beta}$, $ \lambda_{\alpha,\beta}$, etc., are defined by (\ref{n1})--(\ref{n7}).

\begin{trma}
\label{inters_ball_dn} {\rm (see \cite{vas_mix_sev}).} Let $2\le q<\infty$, $2\le \sigma<\infty$, $t$, $k\in \N$, $n\in \Z_+$, $n \le \frac{tk}{2}$, and let $\Psi_j=\Psi_j(P; \, q, \, \sigma; \, t, \, k, \, n; \, \nu)$ be defined by \eqref{psi0}--\eqref{psi7}. Then
\begin{align}
\label{dn_main}
d_n(\cap _{\alpha\in A} \nu_\alpha B^{t,k}_{p_\alpha, \, \theta_\alpha}, \, l^{t,k}_{q,\sigma}) \underset{q,\sigma}{\asymp} \min _{0\le j\le 7} \Psi_j(P; \, q, \, \sigma; \, t, \, k, \, n; \, \nu).
\end{align}
\end{trma}

Now we prove the lower estimates in Theorems \ref{main1}‐‐\ref{main6}.

We begin with the most complicated Theorem \ref{main6}. For the other theorems the arguments are similar, in the end of this section we write the analogues of Lemma \ref{och} (see below).

Recall that $t_{l_*}\in (1, \, q/2)$ is the minimum point of the function $h$.

The number $m_0(n)$ is defined as in the proof of the upper estimate in Theorem \ref{main6} (see Remark \ref{rem_t7}); let $m'(n)$ be the nearest to $m_0(n)$ natural number such that $2^{m'(n)} k_{m'(n)}\ge 2n$. Since 
\begin{align}
\label{2m0n_ntlsk}
2^{m_0(n)} \underset{\mathfrak{Z}}{\asymp} n^{t_{l_*}}(\log n)^\kappa
\end{align}
(see Remark \ref{rem_t7}) and $1<t_{l_*}<q/2$, we have
\begin{align}
\label{2m0n_2mprn}
2^{m_0(n)} \underset{\mathfrak{Z}}{\asymp} 2^{m'(n)}.
\end{align}
By \eqref{low_estim}, we get
$$
d_n(\cap _{j=1}^s SB^{\overline{r}_j}_{p_j,\theta_j}(\mathbb{T}^d), \, B^{\overline{l}}_{q, \sigma}(\mathbb{T}^d)) \underset{\mathfrak{Z}}{\gtrsim} d_n (\cap _{j=1}^s 2^{‐m'(n)(\alpha_j + 1/q ‐1/p_j)} B_{p_j,\theta_j} ^{2^{m'(n)},k_{m'(n)}}, \, l_{q,\sigma} ^{2^{m'(n)},k_{m'(n)}}).
$$

We apply Theorems \ref{1mixed}, \ref{inters_ball_dn}. The problem is reduced to estimating from below the widths of the form
$$
d_n(2^{‐m'(n)(\tilde \alpha +1/q‐1/\tilde p)}B_{\tilde p,\tilde \theta} ^{2^{m'(n)},k_{m'(n)}}, \, l_{q,\sigma} ^{2^{m'(n)},k_{m'(n)}}),
$$
where $\tilde \alpha$, $\tilde p$, $\tilde \sigma$ are as follows:

{\bf Case 1.} $\tilde \alpha = \alpha_i$, $\tilde p = p_i$, $\tilde \theta = \theta_i$ for some $i=1, \, \dots, \, s$ (see \eqref{psi0}).

{\bf Case 2.} $\tilde \alpha =(1‐\lambda) \alpha_i + \lambda \alpha_j$, $\frac{1}{\tilde p} = \frac{1‐\lambda}{p_i} + \frac{\lambda}{p_j}$, $\frac{1}{\tilde \theta} = \frac{1‐\lambda}{\theta_i} + \frac{\lambda}{\theta_j}$, $i< j$, $\lambda \in (0, \, 1)$ is such that $\tilde p \in \{2, \, q\}$, $\tilde \theta\in \{2, \, \sigma\}$ or $\omega' _{\tilde p, q} = \omega' _{\tilde \theta,\sigma} \in (0, \, 1)$, $\omega'_{p_i,q}\ne \omega'_{\theta_i,\sigma}$ (see \eqref{psi1}‐‐\eqref{psi5});

{\bf Case 3.} $\tilde \alpha = \lambda_i \alpha_i + \lambda_j \alpha_j + \lambda_k \alpha_k$, $\frac{1}{\tilde p} = \frac{\lambda_i}{p_i} + \frac{\lambda_j}{p_j} + \frac{\lambda_k}{p_k}$, $\frac{1}{\tilde \theta} = \frac{\lambda_i}{\theta_i} + \frac{\lambda_j}{\theta_j} + \frac{\lambda_k}{\theta_k}$, $\lambda_i>0$, $\lambda_j>0$, $\lambda_k>0$, $\lambda_i+\lambda_j + \lambda_k = 1$, $(\tilde p, \, \tilde \theta) = (q, \, \sigma)$ or $(\tilde p, \, \tilde \theta) = (2, \, 2)$, and the points $\{(1/p_i, \, 1/\theta_i), \, (1/p_j, \, 1/\theta_j), \, (1/p_k, \, 1/\theta_k)\}$ are affinely independent (see \eqref{psi6}‐‐\eqref{psi7}).

\vskip 0.3cm

We define the indices $i_*$, $j_*$ as follows (see Proposition \ref{h_prop}). If $h|_{[t_{l_*‐1}, \, t_{l_*}]} =\varphi_{j(l_*)}$, $h|_{[t_{l_*}, \, t_{l_*+1}]} = \varphi_{j(l_*+1)}$, then $i_*:=j(l_*+1)\in J$, $j_*:=j(l_*)\in J$; if $l_*=1$, $h|_{[t_0, \, t_1]} = h_1|_{[t_0, \, t_1]}$, then $i_*:=i_1\in I\sqcup J$, $j_*:=j_1\in K$; if $l_*=L‐1$, $h|_{[t_{L‐1}, \, t_L]} = h_0|_{[t_{L‐1}, \, t_L]}$, then $i_*:=i_0\in I$, $j_*:=j_0\in J\sqcup K$. 

\begin{Rem}
\label{abcde}
From assertion 6 of Proposition \ref{h_prop}, \eqref{p1ps} and \eqref{ijk_def} it follows that $i_*<j_*$.
\end{Rem}

\begin{Rem}
\label{jjljl1} The equality $J\backslash \{j(l_*), \, j(l_*+1)\} = J\backslash \{i_*, \, j_*\}$ holds (it follows from the definition of $i_*$ and $j_*$, assertions 3, 5 of Proposition \ref{h_prop} and Remark \ref{i1j2_j0jl1}).
\end{Rem}

\begin{Lem}
\label{och} There is a constant $c = c(\mathfrak{Z})>0$ such that {\rm 1)} in case {\rm 1} for $i\notin J\cap\{i_*, \, j_*\}$, {\rm 2)} in case {\rm 2} for $(i, \, j) \ne (i_*, \, j_*)$ or $(i, \, j) = (i_*, \, j_*)$, $\tilde p \notin [2, \, q]$, {\rm 3)} in case {\rm 3} we have
\begin{align}
\label{dddd}
d_n(2^{‐m'(n)(\tilde \alpha +1/q‐1/\tilde p)}B_{\tilde p,\tilde \theta} ^{2^{m'(n)},k_{m'(n)}}, \, l_{q,\sigma} ^{2^{m'(n)},k_{m'(n)}}) \underset{\mathfrak{Z}}{\gtrsim} n^{‐h(t_{l_*})+c}.
\end{align}
\end{Lem}
\begin{proof}
Recall that $\hat p_1, \, \hat p_2\in [2, \, q]$ (see Remark \ref{rem_pi_hat}), and $m_0(n)$ in the proof of the upper estimate in Theorem \ref{main6} was the solution of \eqref{m0n1} or \eqref{m0n2} (see Remark \ref{rem_t7}).

Let $m_*(n)$ be defined by the equation
\begin{align}
\label{2mstn_nt}
2^{m_*(n)} = n^{t_{l_*}}, 
\end{align}
and let $m''(n)$ be the nearest to $m_*(n)$ natural number such that $2^{m''(n)} \ge 2n$. Then
\begin{align}
\label{2mbis_mst} 2^{m_*(n)}\underset{\mathfrak{Z}}{\asymp} 2^{m''(n)}.
\end{align}

From \eqref{2m0n_ntlsk}, \eqref{2m0n_2mprn}, \eqref{2mstn_nt}, \eqref{2mbis_mst}, the condition $t_{l_*}\in (1, \, q/2)$ and \eqref{km_md1} we get that there is $\gamma =\gamma(\mathfrak{Z}) > 0$ such that, for sufficiently large $n$,
\begin{align}
\label{a_a} 2n \le 2^{m'(n)}k_{m'(n)} \le n^{q/2}, \quad 2n \le 2^{m''(n)} \le n^{q/2},
\end{align}
\begin{align}
\label{2mpr_log} 
(\log n)^{‐\gamma} \le 2^{m_0(n)‐m_*(n)} \le (\log n)^\gamma , \quad
(\log n)^{‐\gamma} \le 2^{m'(n)‐m''(n)} \le (\log n)^\gamma.
\end{align}

Further,
\begin{align}
\label{anis_is}
\begin{array}{c}
d_n(2^{‐m'(n)(\tilde \alpha +1/q‐1/\tilde p)}B_{\tilde p,\tilde \theta} ^{2^{m'(n)},k_{m'(n)}}, \, l_{q,\sigma} ^{2^{m'(n)},k_{m'(n)}}) \underset{\mathfrak{Z}}{\gtrsim} \\ \gtrsim k_{m'(n)}^{‐2} d_n(2^{‐m'(n)(\tilde \alpha +1/q‐1/\tilde p)}B_{\tilde p} ^{2^{m'(n)}k_{m'(n)}}, \, l_q ^{2^{m'(n)}k_{m'(n)}}) \underset{\mathfrak{Z}}{\gtrsim}\\ \gtrsim (\log n)^{‐\beta} d_n(2^{‐m''(n)(\tilde \alpha +1/q‐1/\tilde p)}B_{\tilde p} ^{2^{m''(n)}}, \, l_q ^{2^{m''(n)}}),
\end{array}
\end{align}
where $\beta = \beta(\mathfrak{Z}) > 0$; the last order inequality follows from \eqref{km_md1}, \eqref{a_a}, \eqref{2mpr_log} and Theorems \ref{glus}, \ref{p_s}.

We set
\begin{align}
\label{tilde_phi_def}
\tilde \varphi_i(t) = \begin{cases} \varphi_i(t), & i\in J, \\ \alpha_it, & i\in I, \\ (\alpha_i‐1/p_i)t+1/2, & i\in K.\end{cases}
\end{align}

Consider case 1, i.e., $\tilde p = p_i$, $i\notin J\cap\{i_*, \, j_*\}$. Then, by Theorems \ref{glus}, \ref{p_s},
\begin{align}
\label{s_lognb}
(\log n)^{‐\beta} d_n(2^{‐m''(n)(\tilde \alpha +1/q‐1/\tilde p)}B_{\tilde p} ^{2^{m''(n)}}, \, l_q ^{2^{m''(n)}}) \stackrel{\eqref{phi_j_def}, \eqref{2mstn_nt}, \eqref{2mbis_mst}, \eqref{a_a}}{\underset{\mathfrak{Z}}{\asymp}} (\log n)^{‐\beta} n^{‐ \tilde\varphi_i(t_{l_*})} =: S.
\end{align}
From Remark \ref{jjljl1} and assertions 3, 4, 5 of Proposition \ref{strict_ineq} it follows that there is $c_1 = c_1(\mathfrak{Z}) >0$ such that $\tilde\varphi_i(t_{l_*}) \le h(t_{l_*}) ‐ c_1$; hence $$S \ge (\log n)^{‐\beta} n^{‐ h(t_{l_*}) + c_1}\underset{\mathfrak{Z}}{\gtrsim} n^{‐h(t_{l_*})+c}$$
for some $c = c(\mathfrak{Z})>0$; this together with \eqref{anis_is}, \eqref{s_lognb} implies the assertion of Lemma \ref{och} in case 1.

Consider case 2, i.e.,
\begin{align}
\label{case2}
\frac{1}{\tilde p} = \frac{1‐\lambda}{p_i} + \frac{\lambda}{p_j}, \quad 0<\lambda < 1, \quad i< j,
\end{align}
where $(i, \, j) \ne (i_*, \, j_*)$ or $(i, \, j)=(i_*, \, j_*)$, $\tilde p\notin [2, \, q]$. Notice that $1/p_i<1/p_j$ by \eqref{p1ps}. We split the interval $[1/p_i, \, 1/p_j]$ by points $1/q$ and $1/2$ (or, to be precise, by those points which lie in its interior). Then $1/\tilde p$ lies on one segment of this partition; we denote this segment by $[1/p_{ij}^1, \, 1/p_{ij}^2]$ (here $1/p_{ij}^1<1/p_{ij}^2$). We get that
\begin{align}
\label{til_p_pij1}
\frac{1}{\tilde p} = \frac{1‐\tilde \lambda}{p_{ij}^1} + \frac{\tilde \lambda}{p_{ij}^2},
\end{align}
where $\tilde \lambda \in [0, \, 1]$. From Theorems \ref{glus}, \ref{p_s} and \eqref{a_a} it follows that
$$
d_n(B_{\tilde p} ^{2^{m''(n)}}, \, l_q ^{2^{m''(n)}}) \underset{\mathfrak{Z}}{\asymp} [d_n(B_{p_{ij}^1} ^{2^{m''(n)}}, \, l_q ^{2^{m''(n)}})] ^{1‐\tilde \lambda} [d_n(B_{p_{ij}^2} ^{2^{m''(n)}}, \, l_q ^{2^{m''(n)}})]^{\tilde \lambda}.
$$
We have $\frac{1}{p_{ij}^1} = \frac{1‐\lambda_1}{p_i} + \frac{\lambda_1}{p_j}$, $\frac{1}{p_{ij}^2} = \frac{1‐\lambda_2}{p_i} + \frac{\lambda_2}{p_j}$, where $0\le \lambda_1<\lambda_2\le 1$. We set 
$\alpha_{ij}^1 = (1‐\lambda_1)\alpha_i + \lambda_1 \alpha_j$, $\alpha_{ij}^2 = (1‐\lambda_2)\alpha_i + \lambda_2 \alpha_j$. Then $\tilde \alpha = (1‐\tilde \lambda) \alpha_{ij}^1 + \tilde \lambda \alpha_{ij}^2$,
\begin{align}
\label{lnndnl1l}
\begin{array}{c}
(\log n)^{‐\beta} d_n(2^{‐m''(n)(\tilde \alpha +1/q‐1/\tilde p)}B_{\tilde p} ^{2^{m''(n)}}, \, l_q ^{2^{m''(n)}}) \underset{\mathfrak{Z}}{\asymp}
\\
\asymp (\log n)^{‐\beta} [d_n(2^{‐m''(n)(\alpha_{ij}^1 +1/q‐1/p_{ij}^1)}B_{p_{ij}^1} ^{2^{m''(n)}}, \, l_q ^{2^{m''(n)}})] ^{1‐\tilde \lambda} \times
\\
\times [d_n(2^{‐m''(n)(\alpha_{ij}^2 +1/q‐1/p_{ij}^2)}B_{p_{ij}^2} ^{2^{m''(n)}}, \, l_q ^{2^{m''(n)}})]^{\tilde \lambda}.
\end{array}
\end{align}

By Theorems \ref{glus}, \ref{p_s}, \eqref{lam_ij_def}, \eqref{phi_j_def}, \eqref{hij_tilhij_def}, \eqref{2mstn_nt}, \eqref{2mbis_mst}, \eqref{a_a} and \eqref{tilde_phi_def}, the magnitude
\begin{align}
\label{d1_def}
D_1:= d_n(2^{‐m''(n)(\alpha_{ij}^1 +1/q‐1/p_{ij}^1)}B_{p_{ij}^1} ^{2^{m''(n)}}, \, l_q ^{2^{m''(n)}})
\end{align}
has the order $n^{‐\tilde\varphi_i(t_{l_*})}$, $n^{‐h_{i,j}(t_{l_*})}$ or $n^{‐\tilde h_{i,j}(t_{l_*})}$ (if $p_{ij}^1$ is equal to $p_i$, $q$ or $2$, respectively); the magnitude
\begin{align}
\label{d2_def}
D_2:=d_n(2^{‐m''(n)(\alpha_{ij}^2 +1/q‐1/p_{ij}^2)}B_{p_{ij}^2} ^{2^{m''(n)}}, \, l_q ^{2^{m''(n)}})
\end{align}
has the order $n^{‐\tilde\varphi_j(t_{l_*})}$, $n^{‐h_{i,j}(t_{l_*})}$ or $n^{‐\tilde h_{i,j}(t_{l_*})}$ (if $p_{ij}^2$ ie equal to $p_j$, $q$ or $2$, respectively). From Proposition \ref{strict_ineq}, \eqref{i01j01}‐‐\eqref{h_def_012} and \eqref{tilde_phi_def} it follows that
\begin{align}
\label{d12_com} D_1 \underset{\mathfrak{Z}}{\gtrsim} n^{‐h(t_{l_*})}, \quad D_2 \underset{\mathfrak{Z}}{\gtrsim} n^{‐h(t_{l_*})}.
\end{align}

We will obtain sharper estimates of these values and show that, for some $\tilde c = \tilde c(\mathfrak{Z})>0$,
\begin{align}
\label{d1ld2l} (\log n)^{‐\beta} D_1^{1‐\tilde\lambda} D_2^{\tilde\lambda} \underset{\mathfrak{Z}}{\gtrsim} n^{‐h(t_{l_*})+\tilde c}.
\end{align}
This together with \eqref{anis_is} and \eqref{lnndnl1l} implies the assertion of the Lemma in case 2.

Let $(i, \, j) \ne (i_*, \, j_*)$.

We show that there is $c_1 = c_1(\mathfrak{Z})>0$ such that
\begin{align}
\label{d1_est_c1}
\begin{array}{c}
D_1 \underset{\mathfrak{Z}}{\gtrsim} n^{‐h(t_{l_*})+c_1} \; \text{for } p_{ij}^1=q, \; p_{ij}^1 =2 \text{ and }p_{ij}^1 = p_i, \; i \notin J\cap \{i_*, \, j_*\}, \\ 
D_2 \underset{\mathfrak{Z}}{\gtrsim} n^{‐h(t_{l_*})+c_1} \; \text{for } p_{ij}^2=q, \; p_{ij}^2 =2 \text{ and }p_{ij}^2 = p_j, \; j \notin J\cap \{i_*, \, j_*\}.
\end{array}
\end{align}

We prove the first estimate (the second one is similar).

Let $p^1_{ij}=q$. Since $i<j$, by \eqref{p1ps}, we have $p_i>q$, $p_j<q$. We claim that $n^{‐h_{i,j}(t_{l_*})} \ge n^{‐h(t_{l_*}) + c_1}$. Indeed, if $(i, \, j)\ne (i_0, \, j_0)$, the inequality follows from assertion 1 of Proposition \ref{strict_ineq}. If $(i, \, j)=(i_0, \, j_0)$, then, by the condition $(i, \, j) \ne (i_*, \, j_*)$, we get $(i_*, \, j_*)\ne (i_0, \, j_0)$; therefore, $l_*<L‐1$ by assertion 5 of Proposition \ref{h_prop} and the definition of $i_*$, $j_*$; hence $h_{i_0,j_0}(t_{l_*}) < h(t_{l_*})$.

Let $p^1_{ij}=2$. Then $p_i>2$, $p_j<2$. We claim that $n^{‐\tilde h_{i,j}(t_{l_*})} \ge n^{‐h(t_{l_*}) + c_1}$. Indeed, if $(i, \, j)\ne (i_1, \, j_1)$, the inequality follows from assertion 2 of Proposition \ref{strict_ineq}. If $(i, \, j)=(i_1, \, j_1)$, then, by the condition $(i, \, j) \ne (i_*, \, j_*)$, we get $(i_*, \, j_*)\ne (i_1, \, j_1)$; hence $l_*>1$ by assertion 3 of Proposition \ref{h_prop} and the definition of $i_*$, $j_*$. Therefore, $\tilde h_{i,j}(t_{l_*}) < h(t_{l_*})$.

Let $p^1_{ij}=p_i$. If $i \notin J\cap\{i_*, \, j_*\}$, then $n^{‐\tilde\varphi_i(t_{l_*})} \ge n^{‐h(t_{l_*}) + c_1}$ (see \eqref{tilde_phi_def}, Remark \ref{jjljl1} and assertions 3, 4, 5 of Proposition \ref{strict_ineq}).

This completes the proof of \eqref{d1_est_c1}.

Since $(i, \, j) \ne (i_*, \, j_*)$ and $i<j$, $i_*<j_*$ (see \eqref{case2} and Remark \ref{abcde}), it follows from \eqref{d1_est_c1} that $D_1 \underset{\mathfrak{Z}}{\gtrsim} n^{‐h(t_{l_*}) + c_1}$ or $D_2 \underset{\mathfrak{Z}}{\gtrsim} n^{‐h(t_{l_*}) + c_1}$. Hence, if $\tilde \lambda \in (0, \, 1)$, we get \eqref{d1ld2l}. Let $\tilde \lambda = 0$ (the case $\tilde \lambda=1$ is similar). If $p_{ij}^1\in \{2, \, q\}$, then $D_1 \underset{\mathfrak{Z}}{\gtrsim} n^{‐h(t_{l_*})+c_1}$. If $p_{ij}^1 = p_i$, then by \eqref{til_p_pij1} we get $\tilde p=p_i$, which contradicts with the condition $\lambda \in (0, \, 1)$ in \eqref{case2}.

Let $(i, \, j)=(i_*, \, j_*)$, $\tilde p\notin [2, \, q]$. Then $p^1_{ij}=p_i\notin [2, \, q]$ or $p^2_{ij} = p_j \notin [2, \, q]$. By \eqref{tilde_phi_def} and assertions 3, 4 of Proposition \ref{strict_ineq}, 
\begin{align}
\label{d1111}
D_1=n^{‐\tilde \varphi_i(t_{l_*})}\ge n^{‐h(t_{l_*})+c_1}
\end{align}
or, respectively,
\begin{align}
\label{d2222}
D_2=n^{‐\tilde \varphi_j(t_{l_*})}\ge n^{‐h(t_{l_*})+c_1},
\end{align}
where $c_1 = c_1(\mathfrak{Z})>0$. Let, without loss of generality, $p^1_{ij}=p_i\notin [2, \, q]$ (the second case is similar). If $p_i<2$, then from the inequality $1/p_{ij}^2>1/p_{ij}^1$ we get $p^2_{ij}=p_j<2$; then the inequalities \eqref{d1111} and \eqref{d2222} hold simultaneously, which implies \eqref{d1ld2l}. If $p_i>q$, then either $p^2_{ij}=p_j>q$ (hence the inequalities \eqref{d1111} and \eqref{d2222} hold simultaneously), or $p_{ij}^2= q$ and $\tilde \lambda <1$; in the last case \eqref{d1ld2l} follows from \eqref{d1111} and the second estimate in \eqref{d12_com}.

Now we consider case 3. Then
\begin{align}
\label{case3_al_p} \tilde \alpha =\lambda_i\alpha_i+\lambda_j\alpha_j+\lambda_k\alpha_k, \quad \frac{1}{\tilde p} = \frac{\lambda_i}{p_i} +\frac{\lambda_j}{p_j} + \frac{\lambda_k}{p_k},
\end{align}
where $\lambda_i>0$, $\lambda_j>0$, $\lambda_k>0$, $\lambda_i+\lambda_j+\lambda_k=1$,
\begin{align}
\label{aa1}
\begin{array}{c}
(\log n)^{‐\beta} d_n(2^{‐m''(n)(\tilde \alpha +1/q‐1/\tilde p)}B_{\tilde p} ^{2^{m''(n)}}, \, l_q ^{2^{m''(n)}}) =\\= (\log n)^{‐\beta} d_n(2^{‐m''(n)\tilde \alpha }B_q ^{2^{m''(n)}}, \, l_q ^{2^{m''(n)}})
\end{array}
\end{align}
or
\begin{align}
\label{aa2}
\begin{array}{c}
(\log n)^{‐\beta} d_n(2^{‐m''(n)(\tilde \alpha +1/q‐1/\tilde p)}B_{\tilde p} ^{2^{m''(n)}}, \, l_q ^{2^{m''(n)}}) =\\= (\log n)^{‐\beta} d_n(2^{‐m''(n)(\tilde \alpha +1/q‐1/2)}B_2 ^{2^{m''(n)}}, \, l_q ^{2^{m''(n)}}).
\end{array}
\end{align}

Let \eqref{aa1} hold. Consider the intersection of the triangle $\Delta_{i,j,k}$ (its vertices are $(1/p_i, \, 1/\theta_i)$, $(1/p_j, \, 1/\theta_j)$, $(1/p_k, \, 1/\theta_k)$) and the segment $[(1/q, \, 0), \, (1/q, \, 1)]$. It is a segment whose endpoints lie on the sides of the triangle; without loss of generality, these sides are $[(1/p_i, \, 1/\theta_i), \, (1/p_j, \, 1/\theta_j)]$ and $[(1/p_i, \, 1/\theta_i), \, (1/p_k, \, 1/\theta_k)]$. Consider the case $p_i\ge q$ (the case $p_i\le q$ is similar).

From \eqref{p_ne_2q}, \eqref{lam_ij_def} we get that $1/q = (1‐\lambda_{i,j})/p_i + \lambda_{i,j}/p_j = (1‐\lambda_{i,k})/p_i + \lambda_{i,k}/p_k$, where $\lambda_{i,j}\in (0, \, 1)$, $\lambda_{i,k}\in (0, \, 1)$. Hence
$$1/q = (1‐\tilde \lambda)((1‐\lambda_{i,j})/p_i + \lambda_{i,j}/p_j) + \tilde \lambda ((1‐\lambda_{i,k})/p_i + \lambda_{i,k}/p_k),$$ $$1/\sigma = (1‐\tilde \lambda)((1‐\lambda_{i,j})/\theta_i + \lambda_{i,j}/\theta_j) + \tilde \lambda ((1‐\lambda_{i,k})/\theta_i + \lambda_{i,k}/\theta_k),$$ where $\tilde \lambda \in (0, \, 1)$, since $(1/q, \, 1/\sigma)$ is the interior point of $\Delta_{i,j,k}$.

Since the vertices of the triangle $\Delta_{i,j,k}$ are affinely independent, we have $$\lambda_i = (1‐\tilde \lambda)(1‐\lambda_{i,j}) + \tilde \lambda (1‐\lambda_{i,k}),\quad \lambda_j =  (1‐\tilde \lambda)\lambda_{i,j}, \quad \lambda_k = \tilde \lambda \lambda_{i,k}.$$
Hence, by Theorem \ref{p_s} and \eqref{case3_al_p},
$$
(\log n)^{‐\beta} d_n(2^{‐m''(n)\tilde \alpha}B_q ^{2^{m''(n)}}, \, l_q ^{2^{m''(n)}}) = 
$$
$$
= (\log n)^{‐\beta} (2^{‐m''(n)((1‐\lambda_{i,j})\alpha_i + \lambda_{i,j}\alpha_j)})^{1‐\tilde \lambda} (2^{‐m''(n)((1‐\lambda_{i,k})\alpha_i + \lambda_{i,k}\alpha_k)})^{\tilde \lambda} \underset{\mathfrak{Z}}{\asymp}
$$
$$
\asymp (\log n)^{‐\beta} [n^{‐h_{i,j}(t_{l_*})}]^{1‐\tilde \lambda} [n^{‐h_{i,k}(t_{l_*})}]^{\tilde \lambda}=: S;
$$
the last order equality follows from \eqref{2mstn_nt}, \eqref{2mbis_mst} and \eqref{hij_tilhij_def}.
We have $\{i, \, j, \, k\}\not \subset \{i_0, \, j_0\}$. Hence, by assertion 1 of Proposition \ref{strict_ineq} and \eqref{i01j01}, \eqref{h0_def}, \eqref{h_def_012}, we get $n^{‐h_{i,j}(t_{l_*})} \ge n^{‐h(t_{l_*}) + c_1}$, $n^{‐h_{i,k}(t_{l_*})} \ge n^{‐h(t_{l_*})}$ or $n^{‐h_{i,k}(t_{l_*})} \ge n^{‐h(t_{l_*}) + c_1}$, $n^{‐h_{i,j}(t_{l_*})} \ge n^{‐h(t_{l_*})}$, where $c_1=c_1(\mathfrak{Z})>0$. Since $\tilde \lambda \in (0, \, 1)$, we have $S \underset{\mathfrak{Z}}{\gtrsim} n^{‐h(t_{l_*})+c}$ for some $c=c(\mathfrak{Z})>0$.

The case \eqref{aa2} is similar; here we use assertion 2 of Proposition \ref{strict_ineq}, \eqref{a_a} and Theorem \ref{glus}.
\end{proof}

We define the numbers $\lambda_{i_*,j_*}$, $\mu_{i_*,j_*}$, $\tilde \lambda_{i_*,j_*}$, $\tilde \mu_{i_*,j_*}$, $\hat \lambda_{i_*,j_*}$, $\theta_{i_*,j_*}$, $p_{i_*,j_*}$, $\tilde \theta_{i_*,j_*}$, $\tilde p_{i_*,j_*}$, $\hat \theta_{i_*,j_*}$, $\hat p_{i_*,j_*}$ according to \eqref{n1}‐‐\eqref{n5} for $(\alpha, \, \beta) = (i_*, \, j_*)$; notice that the notation for $\lambda_{i_*,j_*}$ and $\tilde \lambda_{i_*,j_*}$ is consistent with \eqref{lam_ij_def}.

It follows from Lemma \ref{och} that it remains to obtain the lower estimates for
\begin{gather}
\label{dn1}
d_n(2^{‐m'(n)(\alpha _{i_*} + 1/q‐1/p_{i_*})} B_{p_{i_*}, \theta_{i_*}}^{2^{m'(n)}, k_{m'(n)}}, \, l_{q,\sigma} ^{2^{m'(n)}, k_{m'(n)}}) \; \text{for}\;i_*\in J,
\\
\label{dn2}
d_n(2^{‐m'(n)(\alpha _{j_*} + 1/q‐1/p_{j_*})} B_{p_{j_*}, \theta_{j_*}}^{2^{m'(n)}, k_{m'(n)}}, \, l_{q,\sigma} ^{2^{m'(n)}, k_{m'(n)}}) \; \text{for}\;j_*\in J,
\\
\label{dn3}
2^{‐m'(n)((1‐\lambda_{i_*,j_*})\alpha _{i_*} +\lambda_{i_*,j_*}\alpha_{j_*})} d_n(B_{q, \theta_{i_*,j_*}}^{2^{m'(n)}, k_{m'(n)}}, \, l_{q,\sigma} ^{2^{m'(n)}, k_{m'(n)}}) \; \text{for}\; i_*\in I,
\\
\label{dn4}
2^{‐m'(n)((1‐\mu_{i_*,j_*})\alpha _{i_*} +\mu_{i_*,j_*}\alpha_{j_*}+1/q‐1/p_{i_*,j_*})} d_n(B_{p_{i_*,j_*}, \sigma}^{2^{m'(n)}, k_{m'(n)}}, \, l_{q,\sigma} ^{2^{m'(n)}, k_{m'(n)}}) \; \text{for}\; p_{i_*,j_*}\in [2, \, q],
\\
\label{dn5}
2^{‐m'(n)((1‐\tilde\lambda_{i_*,j_*})\alpha _{i_*} +\tilde\lambda_{i_*,j_*}\alpha_{j_*}+1/q‐1/2)} d_n(B_{2, \tilde\theta_{i_*,j_*}}^{2^{m'(n)}, k_{m'(n)}}, \, l_{q,\sigma} ^{2^{m'(n)}, k_{m'(n)}}) \; \text{for}\; j_*\in K,
\\
\label{dn6}
2^{‐m'(n)((1‐\tilde\mu_{i_*,j_*})\alpha _{i_*} +\tilde\mu_{i_*,j_*}\alpha_{j_*}+1/q‐1/\tilde p_{i_*,j_*})} d_n(B_{\tilde p_{i_*,j_*}, 2}^{2^{m'(n)}, k_{m'(n)}}, \, l_{q,\sigma} ^{2^{m'(n)}, k_{m'(n)}}) \; \text{for}\; \tilde p_{i_*,j_*}\in [2, \, q],
\\
\label{dn7}
\begin{array}{c}
2^{‐m'(n)((1‐\hat\lambda_{i_*,j_*})\alpha _{i_*} +\hat\lambda_{i_*,j_*}\alpha_{j_*}+1/q‐1/\hat p_{i_*,j_*})} d_n(B_{\hat p_{i_*,j_*}, \hat \theta_{i_*,j_*}}^{2^{m'(n)}, k_{m'(n)}}, \, l_{q,\sigma} ^{2^{m'(n)}, k_{m'(n)}})
\\
\text{for }(\omega' _{p_{i_*},q} ‐\omega' _{\theta_{i_*},\sigma})(\omega' _{p_{j_*},q} ‐\omega' _{\theta_{j_*},\sigma})<0, \quad \omega' _{\hat p_{i_*,j_*},q}\in (0, \, 1).
\end{array}
\end{gather}
\begin{Sta}
\label{sta_min}
The minimum of \eqref{dn1}‐‐\eqref{dn7} has the order $n^{‐\alpha_*} (\log n)^{(d‐1)\beta_*}$. 
\end{Sta}
\begin{proof}
From the condition $1<t_{l_*}<q/2$, \eqref{km_md1} and \eqref{2m0n_ntlsk}, \eqref{2m0n_2mprn} it follows that $2^{2m'(n)/q}k_{m'(n)} < n < 2^{m'(n)} k_{m'(n)}^{2/\sigma}$ for large $n$. Therefore, for $2\le p\le q$, $1\le \theta\le \infty$ we have
\begin{align}
\label{dn_pth_55} \begin{array}{c} d_n(B_{p,\theta}^{2^{m'(n)},k_{m'(n)}}, \, l_{q,\sigma}^{2^{m'(n)},k_{m'(n)}}) \stackrel{\eqref{phi3}, \eqref{phi4}, \eqref{phi5}}{\underset{q,\sigma}{\asymp}} \\ \asymp \begin{cases} k_{m'(n)}^{1/\sigma‐1/\theta} (n^{‐1/2}2^{m'(n)/q}k_{m'(n)}^{1/2}) ^{\omega_{p,q}}, & \text{if }\omega'_{\theta,\sigma}\le \omega'_{p,q}, \\ (n^{‐1/2}2^{m'(n)/q}k_{m'(n)}^{1/\sigma})^{\omega_{p,q}}, & \text{if }\omega'_{\theta,\sigma}\ge \omega'_{p,q}.\end{cases} \end{array}
\end{align}

Given $\mu\in [0, \, 1]$, we define the numbers $p(\mu)$, $\theta(\mu)$ and $\alpha(\mu)$ by the equations $\frac{1}{p(\mu)} = \frac{1‐\mu}{\hat p_1} + \frac{\mu}{\hat p_2}$, $\frac{1}{\theta(\mu)} = \frac{1‐\mu}{\hat \theta_1} + \frac{\mu}{\hat \theta_2}$, $\alpha(\mu) = (1‐\mu)\hat \alpha_1 + \mu \hat \alpha_2$ (see notation before Theorem \ref{main6}).

Taking into account the definition of $\hat p_1$ and $\hat p_2$ and of the indices $i_*$, $j_*$, we get that the magnitudes \eqref{dn1}‐‐\eqref{dn7} have the form
$$
\varkappa_n(\lambda) := 2^{‐m'(n)(\alpha(\lambda)+1/q‐1/p(\lambda))} d_n(B_{p(\lambda),\theta(\lambda)}^{2^{m'(n)},k_{m'(n)}}, \, l_{q,\sigma}^{2^{m'(n)},k_{m'(n)}}),
$$
where the number $\lambda\in [0, \, 1]$ is such that $(1/p(\lambda), \, 1/\theta(\lambda))$ is either an endpoint of the interval $\Delta:= [(1/\hat p_1, \, 1/\hat \theta_1), \, (1/\hat p_2, \, 1/\hat \theta_2)]$ or the point of its intersection with one of the straigt lines $L_1=\{(x, \, y)\in \R^2:\, y=1/\sigma\}$, $L_2=\{(x, \, y)\in \R^2:\; y=1/2\}$ or $L_3= \Bigl\{(x, \, y)\in \R^2:\; \frac{x‐1/q}{1/2‐1/q} = \frac{y‐1/\sigma}{1/2‐1/\sigma}\Bigr\}$, which are not collinear with this interval.

Let $(1/p(\lambda), \, 1/\theta(\lambda))$ be the intersection point of $\Delta$ with $L_1$ or $L_2$. Then $\frac{1}{p(\lambda)} = \frac{1‐\tau}{p(\lambda_1)}+\frac{\tau}{p(\lambda_2)}$, $\frac{1}{\theta(\lambda)} = \frac{1‐\tau}{\theta(\lambda_1)}+\frac{\tau}{\theta(\lambda_2)}$,  where $\tau\in [0, \, 1]$, and the numbers $\lambda_1$, $\lambda_2\in [0, \, 1]$ are such that $(1/p(\lambda_i), \, 1/\theta(\lambda_i))$ is either an endpoint of the segment $\Delta$ or is the point of its intersection  
with $L_3$ ($i=1, \, 2$); here, $(\omega'_{p(\lambda_1),q} ‐ \omega'_{\theta(\lambda_1),\sigma})(\omega'_{p(\lambda_2),q} ‐ \omega'_{\theta(\lambda_2),\sigma})\ge 0$. This together with \eqref{dn_pth_55} implies that $\varkappa_n(\lambda) \underset{\mathfrak{Z}}{\asymp} \varkappa_n(\lambda_1)^{1‐\tau} \varkappa_n(\lambda_2) ^\tau$. Therefore, the minimum of the magnitudes \eqref{dn1}‐‐\eqref{dn7} is estimated from below by the minimum of $\varkappa_n(0)$, $\varkappa_n(1)$ and $\varkappa_n(\hat \lambda)$ (the last one is considered in case 3 of Theorem \ref{main6}; recall that $\hat \lambda\in (0, \, 1)$ was defined in the formulation of the theorem).

In case 1 of Theorem \ref{main6} the number $m_0(n)$ is the solution of \eqref{m0n1} (see Remark \ref{rem_t7}), the orders of the widths for $\varkappa_n(0)$ and $\varkappa_n(1)$ are given by the upper part of formula \eqref{dn_pth_55} for $(p, \, \theta) = (\hat p_1, \, \hat \theta_1)$ and $(p, \, \theta) = (\hat p_2, \, \hat \theta_2)$, respectively; by \eqref{2m0n_2mprn}, $\varkappa_n(0)$ and $\varkappa_n(1)$ have the same order as the left‐ and the right‐hand sides of \eqref{m0n1} for $m=m_0(n)$. We again apply Remark \ref{rem_t7} and obtain that $\varkappa_n(0) \underset{\mathfrak{Z}}{\asymp} \varkappa_n(1) \underset{\mathfrak{Z}}{\asymp} n^{‐\alpha_*} (\log n)^{(d‐1)\beta_*^1}$.

In case 2 the number $m_0(n)$ is the solution of \eqref{m0n2}; we similarly obtain that $\varkappa_n(0) \underset{\mathfrak{Z}}{\asymp} \varkappa_n(1) \underset{\mathfrak{Z}}{\asymp} n^{‐\alpha_*} (\log n)^{(d‐1)\beta_*^2}$.

Case 3 will be considered in details for $\zeta>0$, $(\omega'_{\hat p_1,q}‐\omega'_{\hat \theta_1,\sigma})\zeta >0$. Then the number $m_0(n)$ is the solution of \eqref{m0n1}. We apply \eqref{dn_pth_55} to $(p, \, \theta) := (p(0), \, \theta(0)) = (\hat p_1, \, \hat \theta_1)$, $(p, \, \theta) := (p(\hat \lambda), \, \theta(\hat \lambda)) = (\hat p, \, \hat \theta)$ and $(p, \, \theta) := (p(1), \, \theta(1)) = (\hat p_2, \, \hat \theta_2)$. In the first case we use the upper part of the formula, in the third case, the lower one, and in the second case, we can use both parts. Hence $\varkappa_n(0)$ and $\varkappa_n(\hat \lambda)$ have the same order as the left‐ and the right‐hand sides of the equation \eqref{m0n1} at the point $m_0(n)$; i.e., by Remark \ref{rem_t7}, $\varkappa_n(0) \underset{\mathfrak{Z}}{\asymp} \varkappa_n(\hat \lambda) \underset{\mathfrak{Z}}{\asymp} n^{‐\alpha_*} (\log n)^{(d‐1)\beta_*^1}$. We claim that
\begin{align}
\label{kappa_n_1} \varkappa_n(1) \underset{\mathfrak{Z}}{\gtrsim} \varkappa_n(\hat \lambda).
\end{align}
Indeed, we set
$$
f_{n,p,\theta,\alpha}(m) = 2^{‐m(\alpha+1/q‐1/p)} (n^{‐1/2} 2^{m/q}k_m^{1/\sigma})^{\omega_{p,q}}.
$$
Then $\varkappa_n(1)\underset{\mathfrak{Z}}{\asymp} f_{n,\hat p_2,\hat \theta_2,\hat \alpha_2}(m_0(n))$, $\varkappa_n(\hat \lambda)\underset{\mathfrak{Z}}{\asymp} f_{n,\hat p,\hat \theta,\hat \alpha}(m_0(n))$. In order to prove \eqref{kappa_n_1}, we need to check that
\begin{align}
\label{kappa_n_2} f_{n,\hat p_2,\hat \theta_2,\hat \alpha_2}(m_0(n)) \underset{\mathfrak{Z}}{\gtrsim} f_{n,\hat p,\hat \theta,\hat \alpha}(m_0(n)).
\end{align}

By Remark \ref{rem_m0m1_7071}, $m_0(n)=m_1(n)$, and $m_2(n)$ is the solution of \eqref{m0n2}; hence $$f_{n,\hat p_2,\hat \theta_2,\hat \alpha_2}(m_2(n))= f_{n,\hat p,\hat \theta,\hat \alpha}(m_2(n)).$$ In addition, $m_1(n)<m_2(n)$ (see Remark \ref{m1m2}) and $\hat \alpha_2 ‐ \frac{\omega_{\hat p_2,q}}{2} > \hat \alpha ‐ \frac{\omega_{\hat p,q}}{2}$ by \eqref{str_min_h}. Therefore, $f_{n,\hat p_2,\hat \theta_2,\hat \alpha_2}(m_1(n))> f_{n,\hat p,\hat \theta,\hat \alpha}(m_1(n))$ for large $n$. This completes the proof of \eqref{kappa_n_2}.

The other cases (when $\zeta<0$ or $(\omega'_{\hat p_1,q}‐\omega'_{\hat \theta_1,\sigma})\zeta <0$) can be considered similarly.
\end{proof}

Applying Lemma \ref{och} and Proposition \ref{sta_min}, we complete the proof of Theorem \ref{main6}.

\vskip 0.3cm

Now we write the analogues of Lemma \ref{och} for Theorems \ref{main1}‐‐\ref{main_non_emb}. They yield lower estimates for the widths.

We start from Theorems \ref{main3}‐‐\ref{main_non_emb}; then cases 1--3 are the same as for Theorem \ref{main6}.

\begin{Lem}
\label{och5}
Let the conditions of Theorem {\rm \ref{main3}} hold, let $m_0(n)$ be as in proof of the upper estimate, let $m'(n)$ be the nearest to $m_0(n)$ natural number such that $2^{m'(n)}k_{m'(n)} \ge 2n$. Then \eqref{dddd} holds with $l_*=0$ in cases {\rm 2}, {\rm 3}, as well as in case {\rm 1} for $i\ne s$.
\end{Lem}
\begin{proof}
The arguments are the same as in proof of Lemma \ref{och}. The number $m_*(n)$ is defined by the equation $2^{m_*(n)}=n$. Let $m''(n)$ be the nearest to $m_0(n)$ natural number such that $2^{m''(n)} \ge 2n$. Relations \eqref{a_a} and \eqref{2mpr_log} follow from Remark \ref{rem_t3}; this implies \eqref{anis_is}.

In case 1 we get \eqref{s_lognb} with $l_*=0$. Further we use assertions 3, 4, 5 of Proposition \ref{strict_ineq} (assertion 5 is considered for the case $l=0$, $K=\varnothing$, $i\in J\backslash \{s\}$).

Consider case 2. As in Lemma \ref{och}, we get \eqref{lnndnl1l}. The magnitudes $D_1$ and $D_2$ are defined by formulas \eqref{d1_def} and \eqref{d2_def}, respectively. Since $K=\varnothing$, the number $p_{ij}^1$ can be equal to $p_i$ or $q$, and $D_1$ has the same order as $n^{‐\tilde \varphi_i(1)}$ or $n^{‐h_{i,j}(1)}$, respectively; in addition, $i<j\le s$. In the first case, by assertions 3, 4, 5 of Proposition \ref{strict_ineq} we get $\tilde \varphi_i(1)<h(1)$. In the second case for $(i, \, j)\ne (i_0, \, j_0)$ we use assertion 1 of Proposition \ref{strict_ineq}, and for $(i, \, j)= (i_0, \, j_0)$, we use that $h|_{[t_0, \, t_1]}=\varphi_s|_{[t_0, \, t_1]}$ and $h|_{[t_{L‐1}, \, t_L]} = h_0|_{[t_{L‐1}, \, t_L]}$ (here we apply assertions 2, 5 of Proposition \ref{h_prop} and the conditions $K=\varnothing$, $I\ne \varnothing$); we get $h_{i_0,j_0}(1) < h(1)$. Hence $D_1 \underset{\mathfrak{Z}}{\gtrsim} n^{‐h(1)+c_1}$. Also $p_{ij}^2$ can be equal to $p_j$ or $q$, and $D_2$ has the same order as $n^{‐\tilde \varphi_j(1)}$ or $n^{‐h_{i,j}(1)}$. If $p_{ij}^2=q$, we have $D_2 \underset{\mathfrak{Z}}{\gtrsim} n^{‐h(1)+c_1}$ (this estimate can be proved similarly as for $D_1$). If $p_{ij}^2=p_j$, then $D_2 \underset{\mathfrak{Z}}{\gtrsim} n^{‐h(1)}$; in addition, $\tilde \lambda < 1$ (otherwise, $\tilde p = p_{ij}^2 = p_j$, and we arrivw to a contradiction with the condition $\lambda \in (0, \, 1)$).

Case 3 is the same as in Lemma \ref{och}.
\end{proof}

Thus, in order to prove Theorem \ref{main3}, it remains to estimate from below the magnitude
\begin{align}
\label{dn_al_s}
d_n(2^{‐m'(n)(\alpha_s+1/q ‐1/p_s)} B_{p_s,\theta_s} ^{2^{m'(n)},k_{m'(n)}}, \, l_{q,\sigma} ^{2^{m'(n)},k_{m'(n)}});
\end{align}
to this end, we use the second assertion from Remark \ref{rem_t3}.

\begin{Lem}
\label{och6}
Let the conditions of Theorem {\rm\ref{main4}} hold, let $m_0(n)$ be as in proof of the upper estimate, let $m'(n)$ be the nearest to $m_0(n)$ natural number such that $2^{m'(n)}k_{m'(n)} \ge 2n$. Then \eqref{dddd} holds with $l_*=0$ in cases {\rm 1}, {\rm 3}, as well as in case {\rm 2} for $(i, \, j)\ne (i_1, \, j_1)$ or $(i, \, j) = (i_1, \, j_1)$, $\tilde p\ne 2$.
\end{Lem}
\begin{proof}
The numbers $m_*(n)$ and $m''(n)$ are defined as in the previous lemma. Relations \eqref{a_a} and \eqref{2mpr_log} follow from Remark \ref{rem_t4}; this implies \eqref{anis_is}.

In case 1 we get \eqref{s_lognb} with $l_*=0$; further we use assertions 3, 4, 5 of Proposition \ref{strict_ineq} (assertion 5 is considered for $i\in J$, $l=0$, $K\ne \varnothing$).

In case 3 we argue as in proof of Lemma \ref{och}.

In case 2, as in Lemma \ref{och}, we obtain \eqref{lnndnl1l} and \eqref{d12_com}. Let us prove \eqref{d1ld2l}. For $(i, \, j)\ne (i_1, \, j_1)$ we argue as in Lemma \ref{och}, we only need to modify the proof of analogue of \eqref{d1_est_c1}, i.e.,
\begin{align}
\label{d1_est_c100} \begin{array}{c}
D_1 \underset{\mathfrak{Z}}{\gtrsim} n^{‐h(t_{l_*})+c_1} \; \text{for } p_{ij}^1=q, \; p_{ij}^1 =2 \text{ and }p_{ij}^1 = p_i, \\ 
D_2 \underset{\mathfrak{Z}}{\gtrsim} n^{‐h(t_{l_*})+c_1} \; \text{for } p_{ij}^2=q, \; p_{ij}^2 =2 \text{ and }p_{ij}^2 = p_j
\end{array}
\end{align}
(here we do not need the conditions $i$, $j\notin J\cap\{i_*, \, j_*\}$). Then \eqref{d1ld2l} immediately follows from \eqref{d1_est_c100}. For $p_{ij}^1=q$ and $p_{ij}^2=q$ in the case $(i, \, j)\ne (i_0, \, j_0)$ we use assertion 1 of Proposition \ref{strict_ineq}; in case $(i, \, j)=(i_0, \, j_0)$ we use the equalities $h|_{[t_0, \, t_1]} = h_1|_{[t_0, \, t_1]}$, $h|_{[t_{L‐1}, \, t_L]} = h_0|_{[t_{L‐1}, \, t_L]}$ (they follow from assertions 3, 5 of Proposition \ref{h_prop} and the conditions $I\ne \varnothing$, $K\ne \varnothing$); for $p_{ij}^1=2$ or $p_{ij}^2=2$, since $(i, \, j)\ne (i_1, \, j_1)$, we use assertion 2 of Proposition \ref{strict_ineq}; for $p_{ij}^1=p_i$ $(i\in J)$ or $p_{ij}^2=p_j$ $(j\in J)$ we apply assertions 3, 4, 5 of Proposition \ref{strict_ineq} (assertion 5 is considered for $l=0$, $K\ne \varnothing$).

Now we consider case 2 with $(i, \, j) = (i_1, \, j_1)$, $\tilde p \ne 2$.

If $\tilde p < 2$, then $p_{ij}^2=p_j<2$, and by assertion 4 of Proposition \ref{strict_ineq}, $D_2 \underset{\mathfrak{Z}}{\gtrsim} n^{‐h(t_{l_*})+c_1}$. If $p_{ij}^1<2$, then similarly $D_1 \underset{\mathfrak{Z}}{\gtrsim} n^{‐h(t_{l_*})+c_1}$. If $p_{ij}^1 = 2$, then $D_1 \stackrel{\eqref{d12_com}}{\underset{\mathfrak{Z}}{\gtrsim}} n^{‐h(t_{l_*})}$ and $\tilde \lambda >0$.

Let $\tilde p>2$. From Proposition \ref{strict_ineq} it follows that $D_1 \underset{\mathfrak{Z}}{\gtrsim} n^{‐h(t_{l_*})+c_1}$ (if $p_{ij}^1\in (2, \, q)$, then we use assertion 5 for $l=0$, $K\ne \varnothing$; if $p_{ij}^1=q$, then we use assertion 1 and, for $(i, \, j) = (i_0, \, j_0)$, the equalities $h|_{[t_0, \, t_1]} = h_1|_{[t_0, \, t_1]}$, $h|_{[t_{L‐1}, \, t_L]} = h_0|_{[t_{L‐1}, \, t_L]}$; if $p_{ij}^1>q$, then we use assertion 3). If $p_{ij}^2>2$, then similarly $D_2 \underset{\mathfrak{Z}}{\gtrsim} n^{‐h(t_{l_*})+c_1}$; if $p_{ij}^2=2$, then $D_2 \stackrel{\eqref{d12_com}}{\underset{\mathfrak{Z}}{\gtrsim}} n^{‐h(t_{l_*})}$ and $\tilde \lambda<1$.
\end{proof}

Thus, in order to prove Theorem \ref{main4}, it remains to estimate from below the value
$$
d_n(2^{‐m'(n)(\alpha_*+1/q ‐1/2)} B_{2,\theta_*} ^{2^{m'(n)},k_{m'(n)}}, \, l_{q,\sigma} ^{2^{m'(n)},k_{m'(n)}});
$$
to this end, we apply the second assertion of Remark \ref{rem_t4}.

\begin{Lem}
Let the conditions of Theorem {\rm \ref{main5}} hold, let the number $m_0(n)$ be as in proof of the upper estimate, let $m'(n)$ be the nearest to $m_0(n)$ natural number such that $2^{m'(n)}k_{m'(n)} \ge 2n$. Then \eqref{dddd} holds with $l_*=L$ in cases {\rm 2}, {\rm 3}, as well as in case {\rm 1} for $j\ne 1$.
\end{Lem}

The proof is the same as for Lemma \ref{och5}; here we apply Remark \ref{rem_t5}. The number $m_*(n)$ is defined by the equation $2^{m_*(n)} = n^{q/2}$, $m''(n)$ is the nearest to $m_*(n)$ natural number such that $2n\le 2^{m''(n)}\le n^{q/2}$. The inequality $2^{m'(n)}k_{m'(n)} \le n^{q/2}$ from \eqref{a_a} may fail, but \eqref{anis_is} holds by Remark \ref{rem_t5}.

Thus, in order to prove Theorem \ref{main5}, it remains to estimate the value
\begin{align}
\label{dn_al_1}
d_n(2^{‐m'(n)(\alpha_1+1/q ‐1/p_1)} B_{p_1,\theta_1} ^{2^{m'(n)},k_{m'(n)}}, \, l_{q,\sigma} ^{2^{m'(n)},k_{m'(n)}});
\end{align}
to this end, we use the second assertion from Remark \ref{rem_t5}.

\begin{Lem}
Let the conditions of Theorem {\rm \ref{main_non_emb}} hold, let $m_0(n) = \frac q2 \log n$, and let $m'(n)$ be the nearest to $m_0(n)$ natural number such that $m'(n)\le m_0(n)$. Then \eqref{dddd} holds with $l_*=L$ in cases {\rm 1}, {\rm 3}, as well as in case {\rm 2} for $(i, \, j)\ne (i_0, \, j_0)$ or $(i, \, j) = (i_0, \, j_0)$, $\tilde p\ne q$.
\end{Lem}

The proof is the same as for Lemma \ref{och6}; here $m_*(n):=m_0(n)$, $m''(n):=m'(n)$.

Hence, in order to prove Theorem \ref{main_non_emb}, it remains to estimate the value
$$
d_n:=d_n(2^{‐m'(n)\alpha_*} B_{q,\theta_*} ^{2^{m'(n)},k_{m'(n)}}, \, l_{q,\sigma} ^{2^{m'(n)},k_{m'(n)}}),
$$
where $\alpha_*=(1‐\lambda_{i_0,j_0})\alpha_{i_0}+ \lambda_{i_0,j_0}\alpha_{j_0}$, $\frac{1}{\theta_*} = \frac{1‐\lambda_{i_0,j_0}}{\theta_{i_0}} + \frac{\lambda_{i_0,j_0}}{\theta_{j_0}}$. By conditions of the theorem, $\alpha_*<0$. Therefore, $d_n \underset{n\to \infty}{\to} \infty$ by \eqref{phi1}, \eqref{phi2}.

Under the conditions of Theorems \ref{main1}, \ref{main2} the set of triples $(i, \, j, \, k)$ from case 3 is empty, and, in case 2, the set of pairs $(i, \, j)$ with $\tilde p\in \{2, \, q\}$ and $\omega_{\tilde p,q} = \omega _{\tilde \theta,\sigma}\in (0, \, 1)$ is empty.

In order to prove Theorem \ref{main1}, we use

\begin{Lem}
\label{och9}
Let the conditions of Theorem {\rm \ref{main1}} hold, let the number $m_0(n)$ be the same as in proof of the upper estimate, let $m'(n)$ be the nearest to $m_0(n)$ natural number such that $2^{m'(n)}k_{m'(n)} \ge 2n$. Then
$$
d_n(2^{‐m'(n)(\tilde \alpha +1/q‐1/\tilde p)}B_{\tilde p,\tilde \theta} ^{2^{m'(n)},k_{m'(n)}}, \, l_{q,\sigma} ^{2^{m'(n)},k_{m'(n)}}) \underset{\mathfrak{Z}}{\gtrsim} n^{‐\alpha_*+c}
$$
holds in case {\rm 1} for $j\ne s$ and in case {\rm 2}; here $c=c(\mathfrak{Z})>0$.
\end{Lem}

In order to prove Theorem \ref{main2}, we use

\begin{Lem}
\label{och10}
Let the conditions of Theorem {\rm \ref{main2}} hold, let the number $m_0(n)$ be the same as in proof of the upper estimate, let $m'(n)$ be the nearest to $m_0(n)$ natural number such that $2^{m'(n)}k_{m'(n)} \ge 2n$. Then
$$
d_n(2^{‐m'(n)(\tilde \alpha +1/q‐1/\tilde p)}B_{\tilde p,\tilde \theta} ^{2^{m'(n)},k_{m'(n)}}, \, l_{q,\sigma} ^{2^{m'(n)},k_{m'(n)}}) \underset{\mathfrak{Z}}{\gtrsim} n^{‐\alpha_*+c}
$$
holds in case {\rm 1} for $j\ne 1$ and in case {\rm 2}; here $c=c(\mathfrak{Z})>0$.
\end{Lem}

In proof of Lemmas \ref{och9}, \ref{och10} the number $m_*(n)$ is defined by the equation $2^{m_*(n)} = n$, if $2^{m_0(n)} \underset{\mathfrak{Z}}{\asymp} n(\log n)^\kappa$, and by equation $2^{m_*(n)} = n^{q/2}$, if $2^{m_0(n)} \underset{\mathfrak{Z}}{\asymp} n^{q/2}(\log n)^\kappa$ (see Remarks \ref{rem_t1}, \ref{rem_t2}).
As $m''(n)$ we take the nearest to $m_*(n)$ natural number such that $2n\le 2^{m''(n)}\le n^{q/2}$. After that we obtain \eqref{anis_is} and use \eqref{non_emb}.

Thus, in order to prove Theorem \ref{main1}, by Lemma \ref{och9}, it suffices to estimate the value \eqref{dn_al_s} (to this end, we apply Remark \ref{rem_t1}).
In order to prove Theorem \ref{main2}, by Lemma \ref{och10}, it suffices to estimate the value \eqref{dn_al_1} (to this end, we apply Remark \ref{rem_t2}).

This completes the proof of the main results.

\vskip 0.3cm

The author expresses her sincere gratitude to D.B. Bazarkhanov for useful discussion and references.

\end{document}